\documentclass[oneside]{amsart}

\usepackage[width=160mm,top=30mm,bottom=32mm]{geometry}
\usepackage{caption}
\usepackage{subcaption}
\usepackage[centertags]{amsmath}
\usepackage{bbm, amsfonts,amssymb, amsthm}
\usepackage{mathrsfs}
\usepackage{fancyhdr}
\usepackage{graphicx}
\usepackage{amsmath}
\usepackage{tikz-cd} 
\usepackage{cancel}
\usepackage{physics}
\usepackage[sorting=nyt]{biblatex}
\usepackage{mathtools}
\usepackage{hyperref}
\usepackage{lipsum}

\newtheorem*{Theorem*}{Theorem}
\newtheorem*{Problema*}{Problem}
\newtheorem{TheoremA}{Theorem}

\newtheorem{TheoremB}{Theorem}

\newtheorem{TheoremC}{Theorem}

\newtheorem{thm}{Theorem}[section]

\newtheorem{lema}[thm]{Lemma}
\newtheorem{prop}[thm]{Proposition}
\newtheorem{cor}[thm]{Corollary}
\newtheorem{exam}[thm]{Example}
\newtheorem{defn}[thm]{Definition} 
\newtheorem{rem}[thm]{Remark}

\newcommand{\cD}{\mathcal{D}}

\newcommand{\cJ}{\mathcal{J}}

\newcommand{\bB}{\mathbb{B}}
\newcommand{\bC}{\mathbb{C}}

\newcommand{\bN}{\mathbb{N}}

\newcommand{\bQ}{\mathbb{Q}}\newcommand{\bR}{\mathbb{R}}
\newcommand{\bS}{\mathbb{S}}\newcommand{\bT}{\mathbb{T}}

\newcommand{\bZ}{\mathbb{Z}}

\newcommand{\nc}{\newcommand}
\nc{\p}{\partial}
\nc{\ol}{\overline}

\def\@maketitle{%
  \normalfont\normalsize
  \let\@makefnmark\relax  \let\@thefnmark\relax
  \ifx\@empty\@date\else \@footnotetext{\@setdate}\fi
  \ifx\@empty\@subjclass\else \@footnotetext{\@setsubjclass}\fi
  \ifx\@empty\@keywords\else \@footnotetext{\@setkeywords}\fi
  \ifx\@empty\thankses\else \@footnotetext{%
    \def\par{\let\par\@par}\@setthanks}\fi
  \@mkboth{\@nx\shortauthors}{\@nx\shorttitle}%
  \global\topskip42\p@ 
  \@settitle
  \ifx\@empty\authors \else \@setauthors \fi
  \ifx\@empty\@commby
  \else
    \baselineskip18\p@
    \vtop{\centering{\footnotesize\@commby\@@par}%
      \global\dimen@i\prevdepth}\prevdepth\dimen@i
  \fi
  \ifx\@empty\@dedicatory
  \else
    \baselineskip18\p@
    \vtop{\centering{\footnotesize\itshape\@dedicatory\@@par}%
      \global\dimen@i\prevdepth}\prevdepth\dimen@i
  \fi
  \@setabstract
  \normalsize
  \dimen@34\p@ \advance\dimen@-\baselineskip
  \vskip\dimen@\relax
}

\definecolor{myteal}{HTML}{00797d}
\hypersetup{
    colorlinks=false,%
    citebordercolor=green,%
    filebordercolor=red,%
    linkbordercolor=blue,%
    urlbordercolor=red}

\title[Singly periodic minimal Wente torus with ends]{Singly periodic minimal Wente torus with ends}
\author{CARLOS ANDRÉS TORO CARDONA}
\address{Instituto de Matemática Pura e Aplicada-Instituto de Matemática e Estatística, Universidade Federal do Rio Grande do Sul, Brazil}
\email{carlos.toroc@impa.br}
\date{}

\begin{document}
\begin{abstract}
We construct singly periodic minimal surfaces in $\bR^3$ parametrized by rhombic lattices and foliated by spherical curvature lines. As an application we find new examples of free boundary and capillary minimal annuli into two spheres of different center and same radius, some of them embedded.

\end{abstract}

\maketitle

\section{Introduction}
The Hopf conjecture, in the theory of constant mean curvature surfaces, claimed that the only compact immersed constant mean curvature surface without boundary in $\bR^3$ is the round sphere. H. Hopf himself proved the conjecture to be true when the surface has the topology of a sphere \cite{Hopf}. A. Alexandrov also solved the problem in the affirmative when the surface is embedded in $\bR^3$ \cite{Alexandrov}. However in 1986, H. Wente \cite{WenteCounterExample} proved the existence of a closed immersed torus in $\bR^3$ with constant mean curvature, thus disproving Hopf's conjecture in the general case.

The Wente torus has a further property, namely, it is foliated by one family of spherical curvature lines. That is, for each line of curvature in this family, there exists a sphere of a certain radius and center which fully contains it. Even though the centers may not coincide for different spherical curvature lines, Wente proved, in a further work \cite{Wentespherical}, that all the centers belong to a fixed straight line.

In 1987, U. Abresch \cite{Abresch} was the first to obtain computational simulations of Wente tori and observed that one family of curvature lines looked planar. Investigating further this property, he classified all the tori of constant mean curvature with one family of planar curvature lines in $\bR^3$. By the same year, R. Walter \cite{Walter} was able to determine explicitly all the immersed tori in $\bR^3$ with constant mean curvature and only spherical curvature lines.

In this article we investigate the existence of a non-compact minimal version of the compact non-zero mean curvature Wente torus in $\bR^3$
\begin{Problema*}
    \textit{Are there any complete minimal tori with ends in $\bR^3$ foliated by spherical lines of curvature?}
\end{Problema*}
It is worth to point out that our contributions to this problem are partial and a complete solution still remains open. Inspired by the work of Wente \cite{Wentespherical}, we introduce in Definition \ref{minimalwentetorus} a family of potentially good candidates to solve the problem, which we call \textit{minimal Wente tori with ends}. A necessary condition to construct such minimal surfaces is that the subjacent lattice is of rhombic type, which lead us to explore a generalization of the methods developed in \cite{Isabel} to the context of rhombic lattices. However we can prove that such objects do not exist
\begin{TheoremA}[{Theorem \ref{nonexistence}}]\label{theoremnonexistence}
    There are no minimal Wente torus with ends. 
\end{TheoremA}
 In the development of the investigation, we produced a family of minimal tori with embedded flat ends, foliated by one family of spherical curvature lines and finite total curvature in the quotient space $\bR^3/\vec{a}$, where $\vec{a}\in \bR^3$ is a non-zero-translational vector. This family can also be viewed as a collection of singly periodic minimal surfaces in $\bR^3$ foliated by spherical curvature lines and with infinitely many ends.
\begin{TheoremB}[{Theorem \ref{singlyperiodic}}]\label{theorem7}
\textit{There exists an infinite countable subset $\cJ\subset \bR$ such that, for each $\tau \in \cJ$ and for all $\hat{g}_0\in \bR^+$ there is a complete minimal immersion
    \begin{equation*}
        \hat{\Psi}_{\tau,\hat{g}_0}:\bT_{\tau}^2\setminus \{P_1^{\tau},\ldots P_{n(\tau)^2}^{\tau}\}\to \bR^3/\vec{a}_{\tau,\hat{g}_0}
    \end{equation*}
    of a torus with $n(\tau)^2\in \bZ$ embedded flat ends and finite total curvature $-4\pi n(\tau)^2$ into $\bR^3/\vec{a}_{\tau,\hat{g}_0}$ with $\vec{a}_{\tau,\hat{g}_0}\in \bR^3\setminus \{0\}$, which lifts to a complete minimal immersion 
    \begin{equation*}
        \Psi_{\tau,\hat{g}_0}:\bC\setminus L_{\tau}\to \bR^3
    \end{equation*}
    defined on the complex plane punctured at the vertices of a rhombic lattice $L_\tau$ and that is moreover foliated by one family of spherical curvature lines.}
\end{TheoremB}
\begin{figure}
     \centering
     \begin{subfigure}[b]{0.45\textwidth}
         \centering
         \includegraphics[width=\textwidth]{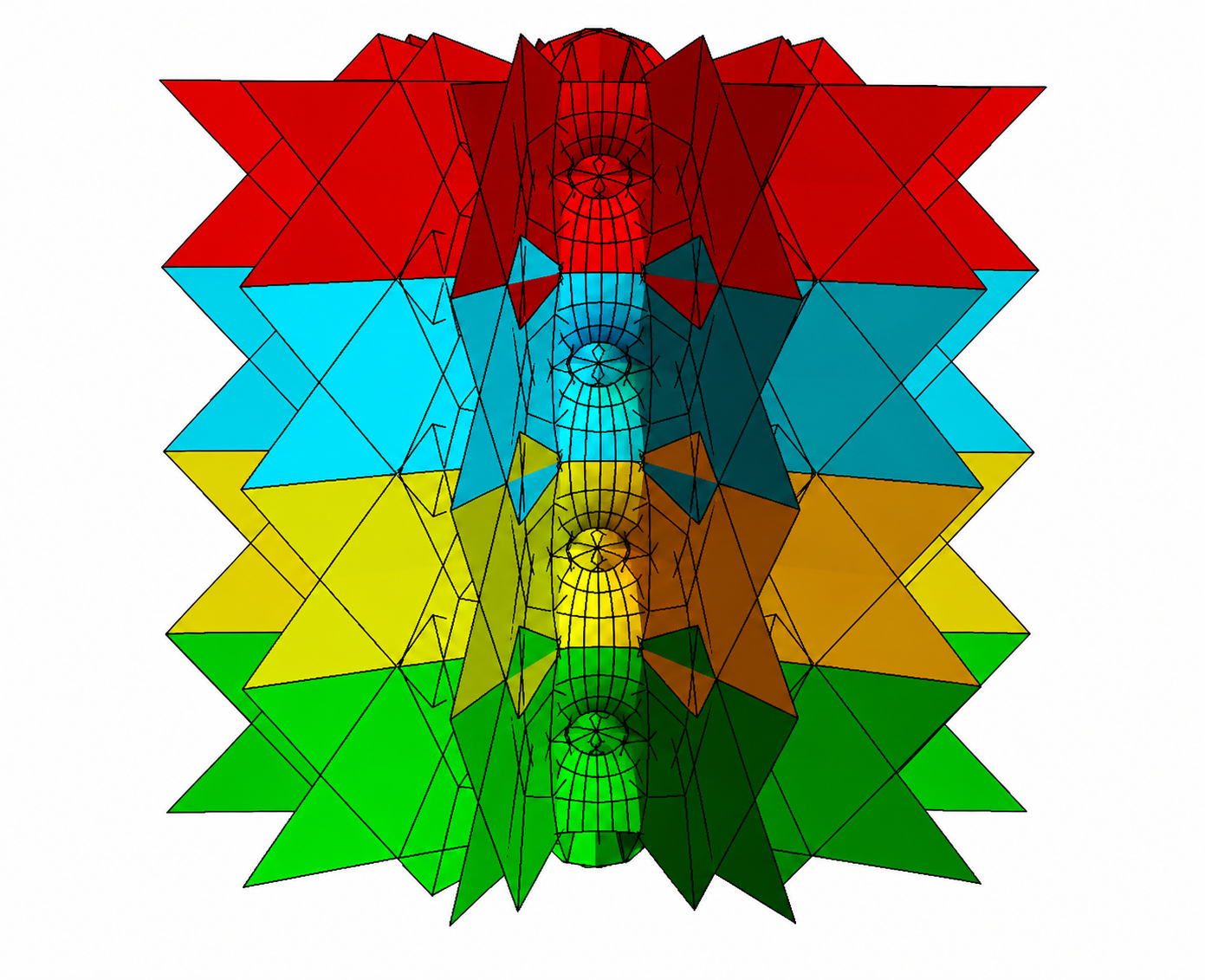}
    \label{fig:Extension n=5 g0=1 vertical}
     \end{subfigure}
     \hfill
     \begin{subfigure}[b]{0.45\textwidth}
         \centering
         \includegraphics[width=\textwidth]{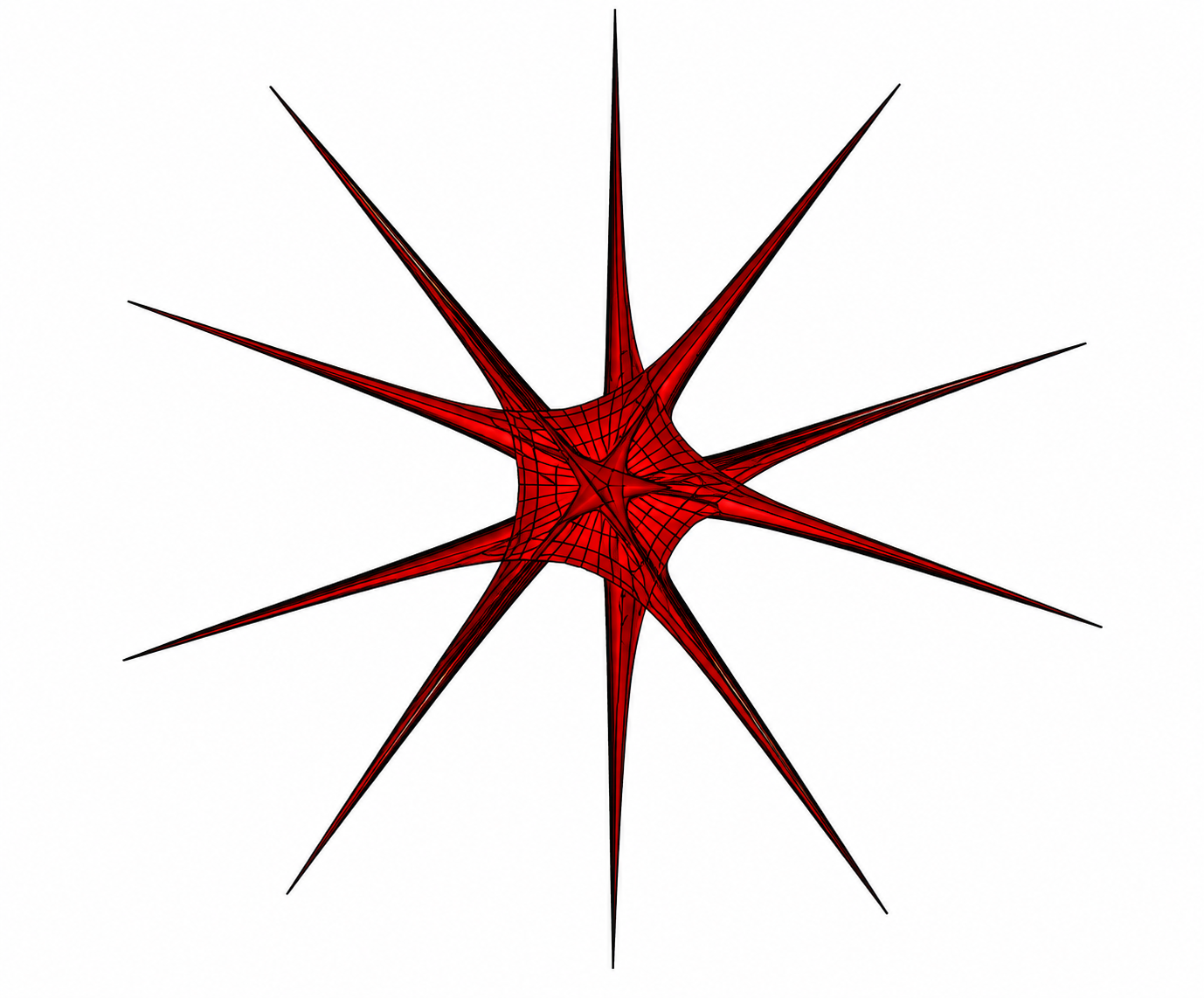}
         \label{fig:Extension n=5 g0=1 top}
     \end{subfigure}
     \caption{Side and top view of a singly periodic minimal surface foliated by spherical curvature lines of Theorem \ref{theorem7}.}
     \end{figure}
     Motivated by the construction of the free boundary minimal annuli of \cite{Isabel}, we wonder if these new extended singly periodic minimal immersions in $\bR^3$ foliated by spherical curvature lines produced also new examples of free boundary minimal annuli in the ball. In our constructions based on rhombic lattices, we were able to find new examples of both free boundary and capillary minimal annuli with respect to two spheres of different center and same radius. 

\begin{TheoremC}[{Theorem \ref{freeboundaryrhombictwospheres}}]\label{theorem8}
  \textit{There exists a family of free boundary minimal annuli in two spheres of same radius in $\mathbb{R}^3$ which are restrictions of an extended singly periodic minimal surface with dihedral symmetry group and foliated by one family of spherical curvature lines.}
\end{TheoremC}

In 1995, Wente \cite{Wenteannuli} constructed non-embedded examples of both capillary and free boundary annuli in the unit ball $\bB^3$ with constant non-zero mean curvature, and wondered whether any embedded capillary constant mean curvature annulus in $\bB^3$ should be rotationally symmetric. This question is the capillary CMC analog of a problem introduced in 1985 by J. Nitsche, asking about the uniqueness of the critical catenoid among the free boundary minimal annuli in $\bB^3$. Inspired by Wente's work, I. Fernández, P. Mira and L. Hauswirth \cite{Isabel} disproved Nitsche's conjecture, by constructing immersed free boundary minimal annuli in the unit ball that are not rotationally symmetric.
\begin{Theorem*}[{I. Fernández, P. Mira, L. Hauswirth} \cite{Isabel}]\label{Isabel1}
    There exists an infinite, countable family of non-embedded free boundary minimal annuli immersed in $\bB^3$ that are foliated by one family of spherical curvature lines and have a dihedral symmetry group.
\end{Theorem*}

\begin{figure}
    \centering
    \begin{subfigure}[b]{0.4\textwidth}
    \includegraphics[width=\linewidth]{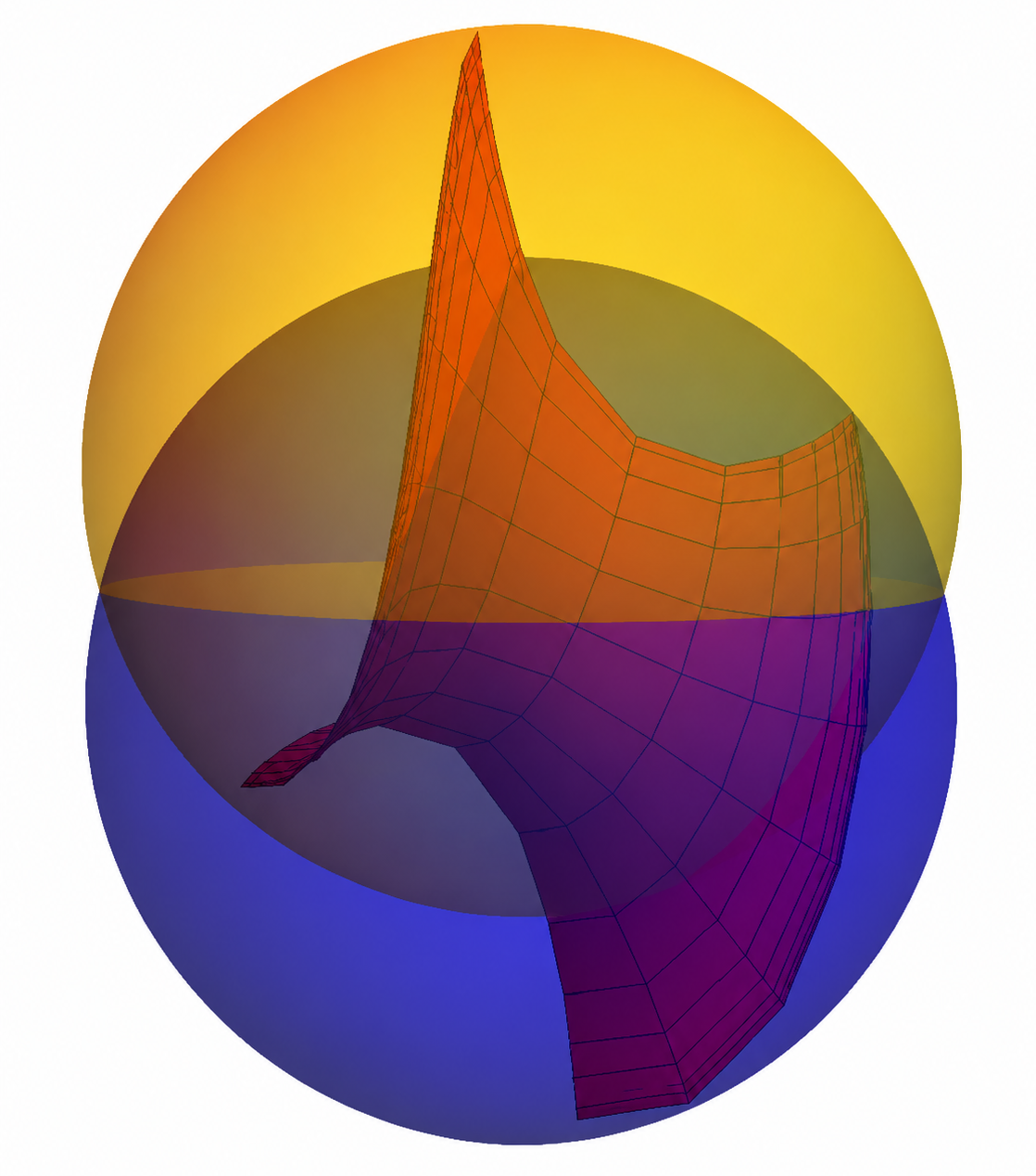}
    \label{fig:FMBS_IN_TWO_SPHERES}
    \end{subfigure}
     \hfill
     \begin{subfigure}[b]{0.4\textwidth}
         \centering
         \includegraphics[width=\textwidth]{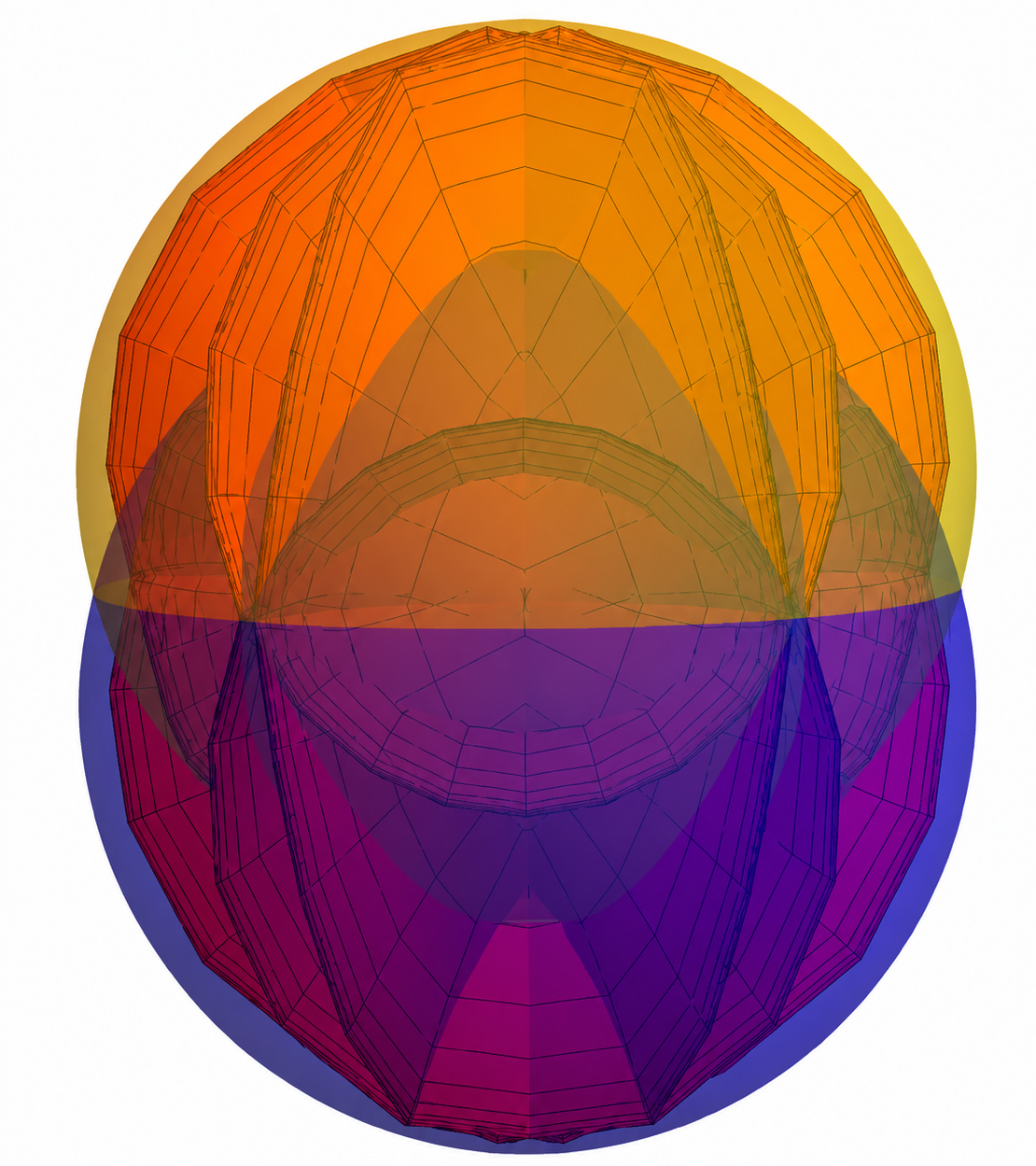}
         \label{fig:FBMS_IN_TWO_SPHERES_FULL}
     \end{subfigure}
     \caption{The fundamental domain and the full free boundary minimal annulus in two spheres of same radius of Theorem \ref{theorem8}.}
\end{figure}

 The method of construction of those annuli was to make a restriction to a vertical infinite strip of an extended minimal immersion foliated by spherical curvature lines and parametrized by a rectangular lattice. Moreover, in the same work \cite{Isabel}, they were also able to answer in the negative Wente's question in the following way:
\begin{Theorem*}[{I. Fernández, P. Mira, L. Hauswirth} \cite{Isabel}]
    There exist compact, embedded and non-rotationally minimal annuli in $\bB^3$ which are capillary, not free boundary and are foliated by one family of spherical curvature lines.
\end{Theorem*}
Regarding the classification problem of free boundary minimal annuli in the unit ball $\bB^3$, J. Lee and E. Yeon \cite{JLee} have discovered a connection between these surfaces and a Liouville type boundary value problem. More specifically, let $F:A(1,R)\to \bB^3$ be a free boundary minimal annulus with conformal factor $\lambda$, \textit{i.e.} $F^{*}can_{\bR^3}=\lambda^2|dz|^2$, and define on the annulus the function $v(re^{i\theta})=\log\left(\frac{1}{r^4\lambda^2}\right)$. Then the function $v$ is a solution to the following problem, where $C_0,\ \epsilon$ are constants:
\begin{gather*}
    E[R,\epsilon,C_0]\quad \begin{cases}
        \Delta v+2C_0^2e^v=0,\quad \text{in}\ A(1-\epsilon,R+\epsilon),\\ \frac{\p v}{\p n}=2e^{-\frac{1}{2}v}-2,\quad |z|=1,\\ \frac{\p v}{\p n}=\frac{2}{R^2}e^{-\frac{1}{2}v}+\frac{2}{R},\quad |z|=R.
    \end{cases}
\end{gather*}
Conversely, any solution to the $E[R,\epsilon,C_0]$ problem produces a minimal immersion in $\bR^3$ which \textit{restricts} to a piece of a free boundary minimal surface orthogonal to two spheres of possibly different center, with the same radius.

Consider the subsets
\begin{equation*}
            H_{R,\epsilon 
            }\coloneq\left\lbrace \xi \in \bC|  \log(1-\epsilon)<Im \xi< \log(R+\epsilon)\right\rbrace,\quad  H_{R,\epsilon}^0\coloneq H_{R,\epsilon}\cap \{\xi \in \bC\mid 0\leq Re\ \xi \leq 2\pi \}.
\end{equation*}
\begin{Theorem*}[{J. Lee, E. Yeon \cite{JLee}}]\label{theoremleeyeon}
    Suppose that $v$ is a solution to $E[R,\epsilon, C_0]$. There exist a conformal minimal immersion $X: H_{R,\epsilon}\to \bR^3$ and unit spheres $S_{O_1},\ S_{O_2}$ (centered at $O_1$ and $O_2$ respectively) satisfying the following conditions:
    \begin{enumerate}
        \item For the conformal factor $\Lambda$ of $X$ given by $X^{*}can_{\bR^3}=\Lambda|d\xi|^2$,
        \begin{equation*}
            \forall\xi \in H_{R,\epsilon 
            }\quad v(e^{-i\xi})=\log \frac{1}{\Lambda^2(\xi)\abs{e^{-i\xi}}^2},
        \end{equation*}
     \item $S_{O_1}$ and $S_{O_2}$ intersect $X(H_{R,\epsilon})$ orthogonally along the level curves $Im\xi=0$ and $Im \xi=\log R$ respectively.
     \item $X(H_{R,\epsilon})=\bigcup_{n\in \bZ}T^n\cdot X(H^0_{R,\epsilon})$ for some rigid motion $T$ in $\bR^3$ such that $T\cdot X(H_{R,\epsilon}^0)=X(H_{R,\epsilon}^0+2\pi).$
    \end{enumerate}
    The subset $X(H_{R,\epsilon}^0)$ is called the fundamental piece of the immersion $X$.
\end{Theorem*}
In the same work, Lee and Yeon proposed two interesting questions. On one hand, they ask whether or not there exists a free boundary minimal surface which is orthogonal to two spheres of different center and same radius \cite[Question 5.5]{JLee}. On the other hand, in the case of a free boundary minimal annulus which is orthogonal to one or two spheres and which is not a part of the catenoid, they wonder if the surface can be divided into more than one congruent piece \cite[Question 5.6]{JLee}.

It is worth to point out that the constructions contained on \cite{Isabel} already answer the previous two questions in the affirmative case. Indeed, consider the immersed minimal surface $\Sigma_{\tau}$ in $\bR^3$ of \cite[Definition 5.9]{Isabel}, which is possibly non-compact. According to \cite[Proposition 7.1 Item (3)]{Isabel}, each of the two boundary components of $\p\Sigma_{\tau}$ intersect orthogonally a corresponding sphere of a certain radius $R(\tau)>0$, whose centers are symmetric with respect to the plane $x_3=0$. Finally, from \cite[Theorem 6.3]{Isabel}, we see that the period map is strictly increasing along the lines transversal to its level sets, implying that the period map assumes rational values. Therefore for those choices of parameters in the complement of the nodal set $\mathfrak{h}^{-1}(0)$ of the height map \cite[Definition 7.2]{Isabel}, we have by \cite[Proposition 7.1 Item (i)]{Isabel} that $\Sigma_{\tau}$ is a compact minimal annuli in $\bR^3$ which is free boundary in two spheres of same radius and different center. Theorem \ref{theorem8} shows that both questions \cite[Question 5.5, 5.6]{JLee} can also be answered affirmatively from the rhombic lattices.

This article is organized as follows. In Section \ref{Section 5.2}, we review the work of H. Wente on the Weierstrass data of minimal immersions of Enneper type, \textit{i.e.} foliated by one family of spherical curvature lines and defined on the complex plane punctured at the points of either a rectangular or a rhombic lattice. We introduce in Definition \ref{minimalwentetorus} the concept of a \textit{minimal Wente torus with ends} and of a \textit{singly periodic minimal Wente torus with ends}. We define complex period maps to investigate the existence of such surfaces and show the importance of the rhombic lattices, see Theorem \ref{importancerombic}. We prove that if the period map is \textit{real} and \textit{rational} then the minimal immersion of Enneper type is singly periodic with respect to a sub-lattice, see Lemma \ref{translationalperiod}.

In Section \ref{Section 5.3}, inspired by the work of \cite{Isabel}, we introduce a three dimensional space that parametrizes minimal immersions of Enneper type. We find a surface inside the parameter space such that the period map is purely real, see Corollary \ref{curveCs}. Subsequently we prove that, along this surface the real part of the period is a smooth increasing function modulo an integer, and therefore it assumes rational values, see Corollary \ref{periodisrationalalongCs}. As a consequence we obtain an infinite countable number of singly periodic minimal Wente tori in the quotient $\bR^3/\vec{a}$ of Euclidean space by a non-zero translational vector, proving Theorem \ref{theorem7}. We prove in Theorem \ref{canonicalconstruction} that the method we described is the canonical way to construct both Wente minimal tori and singly periodic minimal tori in $\bR^3$ which in turns implies Theorem \ref{theoremnonexistence}.

In Section \ref{Section 5.4}, we study the symmetries of the singly periodic minimal Wente torus in terms of the period map. In the general case, the surface is invariant with respect to reflection about planes containing the straight line of centers, see Theorem \ref{symmetriesg0neq1}. We also prove that, in a particular subspace of parameters, the singly periodic minimal Wente torus is invariant under the reflection on a plane which is perpendicular to the line of centers, see Theorem \ref{symmetriesg0=1}.

In Section \ref{Section 5.5}, we study the Hamiltonian system of a pair of functions which control the radii of the spheres containing the spherical curvature lines, their centers and the angles of intersection of the surface with the corresponding spheres. We establish the explicit dependence between the initial conditions of the Hamiltonian system with the space of parameters (Proposition \ref{initialconditionsparameterspace}) and as a result, we show that the Hamiltonian system is periodic, see Corollary \ref{alphabetaomega3periodic}.

In Section \ref{Section 5.6}, we explore the qualitative behavior of the Hamiltonian system in terms of an associated system whose dynamical behavior is easier to understand, see Theorem \ref{alphabetaqualitative}.

In Section \ref{Section 5.7}, we do an analysis of the center of the spheres containing the spherical curvature lines, and prove that the height of the center with respect to a plane perpendicular to the straight line of centers is monotone decreasing (Theorem \ref{centeranalysis}), implying that the singly periodic minimal Wente torus contains minimal strips whose boundaries are contained in two spheres of different center.

Finally, in Section \ref{Section 5.8}, we focus in a special choice of parameters for which the singly periodic minimal Wente torus is more symmetric. In this particular situation, we have more control of the solutions of the Hamiltonian system and we show that it satisfies an additional symmetry property, see Theorem \ref{hamiltoniansystemg0=1}. Consequently, as an application to the theory of capillary and free boundary minimal annuli in two spheres we prove Theorem \ref{theorem8}.\\\\
\textbf{Acknowledgements.} I thank the X Workshop of Differential Geometry in Maceió-Brazil where the idea for this project was originated during a lecture of professor Isabel Fernández, whom I also thank for her inspiration and corrections that helped me to develop this work. I thank my PhD advisor Lucas Ambrozio for the insightful ideas and constant encouragement to pursue the problem. I am grateful to professor Vanderson Lima for his valuable questions and improvements on preliminary versions of this work. I thank the XXII Escola de Geometria Diferencial in Teresina-Brazil for the opportunity to present this research project, which gave me the energy and motivation to conclude it. Finally, I thank FAPERJ (Grant number E26/202.321/2024) for supporting this work.
\section{Geometrical considerations}\label{Section 5.2}

Let $\Omega\subset \bC$ be a simply connected region and $\Psi:\Omega\to \bR^3$ be a conformally immersed minimal surface in $\bR^3$ parametrized by its lines of curvature and with fundamental forms
\begin{gather}\label{fundamentalforms1}
    I=e^{2\omega}(du^2+dv^2),
\end{gather}
\begin{gather}\label{fundamentalforms2}
    II=-du^2+dv^2.
\end{gather}
The vectors $e^{-\omega}\p_u$ and $e^{-\omega}\p_v$, form an orthonormal basis with respect to the first fundamental form $I$, and point into the principal directions of the surface, with principal curvatures $k_1=-e^{-2\omega}$ and $k_2=e^{-2\omega}$, respectively.

On one hand, by the Gauss equation in $\bR^3$
\begin{equation*}
    K=k_1k_2=-e^{-4\omega},
\end{equation*}
while on the other hand, by the transformation property of the Gaussian curvature $K$ under a conformal change we know that
\begin{equation*}
    Ke^{2\omega}=-\Delta_{can}\omega.
\end{equation*}
Thus, $\omega(u,v)$ is a solution to the equation
\begin{equation*}
    \Delta_{can} \omega-e^{-2\omega}=0.
\end{equation*}
Assume that the immersion is of \textit{Enneper type}, that is, suppose that one family of curvature lines, the vertical $v$-lines $\Psi(u,\cdot)$ of the immersed surface, is spherical. The center of the spheres containing the family of spherical lines of curvature lies in a straight line $l$ (see \cite[Proof of Theorem 4.1]{Wentespherical}), which after a rotation of the surface in $\bR^3$, we can assume to be parallel to the $x_3$ axis. Let
\begin{gather}\label{Weierstrassdata}
    \Phi(z)=\left(\frac{1}{2}\left(\frac{1}{g(z)}-g(z)\right),\frac{i}{2}\left(\frac{1}{g(z)}+g(z)\right),1\right)\phi(z)dz
\end{gather}
be the Weierstrass representation data of $\Psi(u,v)$. By \cite[equation 4.11]{Wentespherical} the function $\phi(z)$ is, up to reparametrization and scaling, a solution to the differential equation
\begin{equation}\label{equationphi}
    \phi'(z)^2=-\phi(z)^3+p\phi(z)^2+q\phi(z)+1
\end{equation}
where $p,q$ are real constants. The solution to this differential equation is well-known from the theory of elliptic functions.
\begin{defn}
\normalfont{
    Let $2\omega_1,2\omega_2\in \bC$ be two linearly independent vectors in $\bR^2$. A meromorphic function $f:\bC\to \bC$ is called \textit{elliptic} if it is doubly periodic, \textit{i.e.} for all $z\in \bC$,
    \begin{equation*}
        f(z+2\omega_1)=f(z),\quad f(z+2\omega_2)=f(z).
    \end{equation*}
    }
\end{defn}
The classical example of an elliptic function in the complex plane is the Weierstrass function:
\begin{exam}
\normalfont{
    Let $\omega_1,\omega_2\in \bC$ be such that $Im\left(\frac{\omega_2}{\omega_1}\right)>0$ and consider the lattice $L=\{\Omega_{m,n}|m,n\in \bZ\}$ in the complex plane, where $\Omega_{m,n}=2m\omega_1+2n\omega_2$. Then the following series 
    \begin{equation*}
        \wp(z)\coloneq \frac{1}{z^2}+\sum_{(m,n)\in \bZ^2\setminus(0,0)}\frac{1}{(z-\Omega_{m,n})^2}-\frac{1}{(\Omega_{m,n})^2}
    \end{equation*}
    defines an elliptic function with periods $2\omega_1,2\omega_2$, which has poles of order exactly two at the lattice points $\Omega_{m,n}$ \cite[Section 20.2]{Whittaker}. Moreover, there exists complex invariants $g_2,g_3$, independent of the choice of generators of the lattice $L$, given by
    \begin{gather*}
    g_2=60\sum_{(m,n)\in \bZ^2\setminus(0,0)}\frac{1}{\Omega_{m,n}^4},\quad  g_3=140\sum_{(m,n)\in \bZ^2\setminus(0,0)}\frac{1}{\Omega_{m,n}^6},
\end{gather*}
such that the Weierstrass $\wp(z)$ function satisfies the differential equation
\begin{equation}\label{differentialequationp}
    \wp'(z)^2=4\wp(z)^3-g_2\wp(z)-g_3.
\end{equation}
Associated with the Weierstrass $\wp$ function, we can introduce two meromorphic functions $\zeta(z),\ \sigma(z)$ on the complex plane, defined by
\begin{equation*}
    \zeta'(z)=-\wp(z),\quad \frac{\sigma'(z)}{\sigma(z)}=\zeta(z).
\end{equation*}
The $\zeta$ function has simple poles at the lattice points, whereas the $\sigma$ function is holomorphic and has simple zeros at the lattice points. These functions are called \textit{quasi-elliptic}, since they satisfied the following translational relations for $j=1,2$:
\begin{gather*}
\zeta(z+2\omega_j)=\zeta(z)+2\eta_j,\quad \sigma(z+2\omega_j)=(-1)^je^{2\eta_j(z+\omega_j)}\sigma(z),\quad \eta_j\coloneq\zeta(\omega_j)
\end{gather*}
Conversely, given complex invariants $g_2,g_3$ such that their \textit{modular discriminant} $\Delta_{mod}=g_2^3-27g_3^2\neq 0$, there exist a \textit{non-degenerate} Weierstrass elliptic function $\wp(z)$ associated with these invariants \cite[Section 21.73]{Whittaker}.
}
\end{exam}
\begin{thm}[{\textit{Cf.} \cite[Theorem 4.2]{Wentespherical}}]\label{caradephi}
    The function $\phi(z)$ is a solution to equation \eqref{equationphi} with $p,\ q\in \bR$ if and only if
    \begin{equation}\label{phifunction}
        \phi(z)=b-4\wp(z),
    \end{equation}
    where $b=\frac{p}{3}$ and $\wp(z)$ is the Weierstrass function with real invariants $g_2,g_3$ defined by the equations
    \begin{equation*}
        g_2=\frac{p^2+3q}{12},
    \end{equation*}
    \begin{equation}\label{cubicequationb}
        b^3-4g_2b-16g_3=1,\quad b\in \bR.
    \end{equation}
\end{thm}
Let $L$ be a lattice in the complex plane and suppose that $b$ satisfies equation \ref{cubicequationb}. Notice that, by the properties of the Weierstrass function, the solutions to the equation $b=4\wp(x)$ are of the form $x=\pm \mu+L$ with $\mu$ inside the fundamental parallelogram. Since $b$ satisfies equation \eqref{cubicequationb}, we have
    \begin{equation*}
        16(4\wp^3(\pm \mu)-g_2\wp(\pm \mu)-g_3)=1,
    \end{equation*}
    which implies, by the differential equation \eqref{differentialequationp} that
    \begin{equation*}
        \wp'(\pm \mu)^2=\frac{1}{16}.
    \end{equation*}
     Since $\wp'(z)$ is an odd function, we can define $\mu$ so that $\wp'(\mu)=-\frac{1}{4}$ and $\wp'(-\mu)=\frac{1}{4}.$
\begin{defn}\label{mudefinition}
\normalfont
    Let $L$ be a lattice in the complex plane and suppose that $b\in \bR$ is a solution to equation \ref{cubicequationb}. We define $\mu=\mu(b,L)$ as the unique solution in the fundamental parallelogram to the equations
    \begin{equation*}
    \wp(\mu|L)=\frac{b}{4},\quad \wp'(\mu|L)=-\frac{1}{4}.
    \end{equation*}
\end{defn}
\begin{rem}
\normalfont{
    The Weierstrass function $\wp(z)$ generated by the \textit{real} invariants $g_2,g_3$ could be \textit{degenerated} when the modular discriminant $\Delta_{mod}=0$. There are four cases to consider according to \cite[Section 2.2]{Isabel}:
    \begin{itemize}
        \item[I.] $g_2=g_3=0$. In this case $\wp(z)=\frac{1}{z^2}$ and the associated lattice is trivial $L={0}$.
        \item[II.] The modular discriminant is trivial $\Delta_{mod}=0$ with $g_2g_3\neq 0$. The Weierstrass $\wp(z)$ function is singly periodic with respect to a lattice $L=\{m\omega| m\in \bZ\}$, with period $\omega$ purely real or purely imaginary.
        \item[III.] The modular discriminant is positive $\Delta_{mod}>0$. In this case the Weierstrass $\wp(z)$ function is non-degenerated \textit{i.e.} doubly periodic, associated to a \textit{rectangular lattice i.e.} with generators $2\omega_1\in \bR$, $2\omega_2\in i\bR$.
        \item[IV.] The modular discriminant is negative $\Delta_{mod}<0$. In this case the Weierstrass $\wp(z)$ function is non-degenerated \textit{i.e.} doubly periodic, associated to a \textit{rhombic lattice i.e.} with generators $2\omega_1\in \bC$, $2\omega_2\in \bC$ such that $\omega_2=\overline{\omega_1}$.
    \end{itemize}
    }
\end{rem}
Wente \cite[Examples 1-3]{Wentespherical} studied the Weierstrass data $\Phi$ associated to the degenerated cases I) and II), while I. Fernández, P. Mira and L. Hauswirth studied the case III). We explore the case IV) and show its importance to investigate the existence of a minimal Wente torus with ends in $\bR^3$.

Notice that the Gauss map can be calculated in terms of the function $\phi(z)$. In fact, by equation \eqref{fundamentalforms2},
    \begin{equation*}
   -\frac{1}{2}dz^2= II(\p_z,\p_z)dz^2=-\frac{1}{2}\frac{dg}{g}\otimes \phi(z)dz,
\end{equation*}
    so that
    \begin{equation*}
        \frac{g'}{g}=\frac{1}{\phi},
    \end{equation*}
    and therefore after a further rotation of the surface around the $x_3$ axis, there exists $\hat{g}_0\in \bR^+$ such that
    \begin{equation}\label{gaussmap}
    g(z)=\hat{g_0}\exp \left( \int_{\gamma_{z}}\frac{1}{\phi(\nu)}d\nu\right),
\end{equation}
where $\gamma_z$ is any path in $\Omega$ joining a fixed based point to $z\in U$.

Conversely if we have Weierstrass data as in equation \eqref{Weierstrassdata} defined on a simply connected region $\Omega\subset \bC$ satisfying equations \eqref{phifunction}, \eqref{cubicequationb} and \eqref{gaussmap}, then we construct through the Weierstrass representation formula a conformal minimal immersion of Enneper type $\Psi:\Omega\subset \bC\to \bR^3$ with fundamental forms \eqref{fundamentalforms1}, \eqref{fundamentalforms2}.

In fact we can make this work globally on the  whole complex plane punctured at the lattice points.
\begin{defn}
\normalfont
    Let $L$ be a lattice in the complex plane with real invariants and let $b\in \bR$, $\hat{g}_0\in \bR^+$. \textit{A Wente-Weierstrass data} $\Phi_{L,b,\hat{g}_0}$ is a one-form defined by \eqref{Weierstrassdata} satisfying the conditions \eqref{phifunction}, \eqref{cubicequationb} and \eqref{gaussmap}. 
\end{defn}
\begin{thm}\label{gaussmapformula}
     Let $\Phi_{L,b,\hat{g}_0}$ be a Wente-Weierstrass data defined on a simply connected open set $\Omega\subset \bC$ with $\Omega\cap L=\emptyset$. Then it extends to the whole $\bC\setminus L$, and moreover the Gauss map can be calculated as
    \begin{equation}\label{gaussmapsigmafunctions}
        g(z)=\hat{g_0}\exp(2\zeta(\mu|L)z)\frac{\sigma(\mu-z|L)}{\sigma(\mu+z|L)}.
    \end{equation}
\end{thm}
\begin{proof}
    The $\phi(z)$ function \eqref{phifunction} clearly extends to $\bC\setminus L$, so we focus on the Gauss map. The first thing to prove is that the Gauss map is well-defined on the whole punctured complex plane. The zeros $z=\pm \mu+L$ of $\phi(z)=b-4\wp(z)$ are all simple because $\wp'(\mu)\neq 0$, therefore the meromorphic differential \begin{equation*}
       \eta= \frac{dz}{\phi(z)}=\frac{dz}{b-4\wp(z)}
    \end{equation*}
    has simple poles at those points, with residues
    \begin{align*}
    \begin{split}
        Res_{z=\mu+L}\eta&=\lim_{z\to \mu}\frac{z-(\mu+L)}{b-4\wp(z)}=-\frac{1}{4}\lim_{z\to \mu}\frac{z-(\mu+L)}{\wp(z)-\wp(\mu+L)}=-\frac{1}{4\wp'(\mu+L)}=1,
        \end{split}
    \end{align*}
    \begin{gather*}
        Res_{z=-\mu+L}\eta=\lim_{z\to \mu}\frac{z-(-\mu+L)}{b-4\wp(z)}=-\frac{1}{4}\lim_{z\to \mu}\frac{z-(-\mu)}{\wp(z)-\wp(-\mu+L)}=-\frac{1}{4\wp'(-\mu+L)}=-1.
    \end{gather*}
    Hence
    \begin{equation*}
        Res_{z=\pm \mu+L}\eta\in\{-1,1\}.
    \end{equation*}
    This implies that along any two paths $\gamma_1,\gamma_2$ from a fixed based point to the point $z$
    \begin{equation*}
        \int_{\gamma_{1,z}}\frac{d\nu}{\phi(\nu)}-\int_{\gamma_{2,z}}\frac{d\nu}{\phi(\nu)}=2\pi i k,\quad k\in \bZ,
    \end{equation*}
    and therefore the Gauss map, as in equation \eqref{gaussmap}, is well-defined.
    
The rest of the proof follows \cite[Lemma 5.4]{Isabel}. Let $z\in \bC\setminus L$ and consider a simply connected open set $U\subset \bC$ containing the points $0,\ z$ avoiding the points $\mu,\ -\mu$ and its partners related by translations along the lattice $L$. In this way, the function $\frac{\sigma(\mu-z)}{\sigma(\mu+z)}$ is holomorphic and non-vanishing in $U$. Therefore there exists a branch of the logarithm in this region $U$ \cite[Theorem 6.2]{Stein},
\begin{equation*}
    F(z)=\log_U\frac{\sigma(\mu-z)}{\sigma(\mu+z)}=\int_0^z\frac{\left(\frac{\sigma(\mu-\xi)}{\sigma(\mu+\xi)}\right)'}{\left(\frac{\sigma(\mu-\xi)}{\sigma(\mu+\xi)}\right)}d\xi,
\end{equation*}
such that
\begin{equation}\label{exponentialF}
    \frac{\sigma(\mu-z)}{\sigma(\mu+z)}=e^{F(z)}\quad \text{on $U$}.
\end{equation}
Performing the calculation, we obtain
\begin{align*}
    F(z)&=\int_0^z\frac{\sigma(\mu+\xi)}{\sigma(\mu-\xi)}\cdot\frac{-\sigma(\mu+\xi)\sigma'(\mu-\xi)-\sigma(\mu-\xi)\sigma'(\mu+\xi)}{\sigma(\mu+\xi)^2}d\xi\\&=-\int_0^z\zeta(\mu-\xi)+\zeta(\mu+\xi)d\xi.
\end{align*}
Consider the identity \cite[Eq. 5.13]{Isabel}
\begin{equation}\label{identityzeta}
    \zeta(z_1+z_2)+\zeta(z_1-z_2)=2\zeta(z_1)+\frac{\wp'(z_1)}{\wp(z_1)-\wp(z_2)}.
\end{equation}
Taking $z_1=\mu$, $z_2=\xi$, and since $\wp'(\mu)=-\frac{1}{4}$, we obtain
\begin{equation*}
     \zeta(\mu+\xi)+\zeta(\mu-\xi)=2\zeta(\mu)+\frac{\wp'(\mu)}{\wp(\mu)-\wp(\xi)}=2\zeta(\mu)-\frac{1}{b-4\wp(\xi)},
\end{equation*}
so that
\begin{equation*}
    F(z)= -\int_0^z\zeta(\mu+\xi)+\zeta(\mu-\xi)d\xi=-2\zeta(\mu)+\int_0^z\frac{1}{b-4\wp(\xi)}d\xi,
\end{equation*}
that is,
\begin{equation}\label{integralequationphi}
    \int_0^z\frac{1}{b-4\wp(\xi)}d\xi=2\zeta(\mu)z+F(z).
\end{equation}
Taking the exponential function in both sides and using equation \eqref{exponentialF} we obtain 
\begin{equation*}
        g(z)=\hat{g_0}\exp(2\zeta(\mu|L)z)\frac{\sigma(\mu-z|L)}{\sigma(\mu+z|L)},
    \end{equation*}
and the proof is completed.
\end{proof}
\begin{thm}[{\textit{Cf.} \cite[Section 2.2]{Isabel}}]\label{EnnepertypeImmersion}
    Let $\Phi_{L,b,\hat{g}_0}$ be a Wente-Weierstrass data on $\bC\setminus L$. Then the map
    \begin{gather*}
\Psi:\bC\setminus L\to\bR^3
\end{gather*}
defined by 
\begin{equation*}
    \Psi(z)=Re\int_{z_0}^z\Phi_{L,b,\hat{g_0}}(\xi)d\xi,
\end{equation*}
is a well-defined complete minimal immersion in $\bR^3$, with infinitely many embedded flat ends at the lattice points $L$, and with the property that the $v$-lines $\Psi(u_0,\cdot)$ are spherical lines of curvature.
\end{thm}
\begin{proof}
We proceed to verify that the hypothesis in the global Weierstrass representation Theorem are satisfied. By the previous discussion the theorem is true in simply connected regions $\Omega\subset \bC$ and moreover, by Theorem \ref{gaussmapformula}, the Weierstrass data extends to the complex plane punctured at the lattice points. Next we show that the minimal immersion is well-defined globally in $\bC\setminus L$.

The first thing to prove is that the poles of the Gauss map are exactly the zeros of the one-form $\phi(z)dz$, with the same multiplicity. In fact, notice that, from \eqref{gaussmapsigmafunctions}, $g$ has simple poles at the points $z=-\mu+L$, whereas the one-form $\phi(z)dz=(b-4\wp(z))dz=4(\wp(\mu)-\wp(z))dz$ has simple zeros at $z=\mu+L$ and $z=-\mu+L$, since $\wp'(\mu)\neq 0$.

The next condition to be verified is that $dh\coloneqq \phi(z)dz$ has real residues at the lattice points $L$. Since $\phi(z)=b-4\wp(z)$ has only poles of order two at $L$ with no residue (by the representation in series of the Weierstrass elliptic function, see below \eqref{pseries}, this condition is satisfied.

For the final condition, we need to check that, at points $z_0\in L$,    \begin{equation}\label{residuecondition}
        Res_{z=z_0}\frac{dh}{g}+\overline{Res_{z=z_0}gdh}=0.
    \end{equation}
  We will prove that in fact each residue vanishes separately. Since the Gauss map $g(z)$ is regular at the lattice points, then both $\frac{dh}{g}$ and $gdh$ have poles of order exactly two at the lattice points. Hence, by a formula of complex analysis \cite[Theorem 1.4]{Stein},
  \begin{align*}
  \begin{split}
      Res_{z=z_0}\frac{dh}{g}&=\lim_{z\to z_0}\frac{d}{dz}(z-z_0)^2\frac{b-4\wp(z)}{g(z)}\\&=\lim_{z\to z_0}\frac{2(z-z_0)(b-4\wp(z))}{g(z)}-\frac{4(z-z_0)^2\wp'(z)}{g(z)}-\frac{(z-z_0)^2(b-4\wp(z))g'(z)}{g(z)^2}\\ &=\lim_{z\to z_0}\frac{2(z-z_0)(b-4\wp(z))}{g(z)}-\frac{4(z-z_0)^2\wp'(z)}{g(z)}-\frac{(z-z_0)^2}{g(z)}\\&=\frac{1}{g(z_0)}\lim_{z\to z_0}-8(z-z_0)\wp(z)-4(z-z_0)^2\wp'(z).
      \end{split}
  \end{align*}
  Next, we use the representation in series 
  \begin{equation}\label{pseries}
      \wp(z)=\frac{1}{(z-z_0)^2}+\sum_{n=2}^{\infty}c_n(z-z_0)^{2n-2},
  \end{equation}
  and
  \begin{equation*}
      \wp'(z)=-\frac{2}{(z-z_0)^3}+\sum_{n=2}^{\infty}(2n-2)c_n(z-z_0)^{2n-3}.
  \end{equation*}
  We see that
  \begin{align*}
  \begin{split}
      -8(z-z_0)\wp(z)-4(z-z_0)^2\wp'(z)&=\cancel{-\frac{8}{(z-z_0)}}-8\sum_{n=2}^{\infty}c_n(z-z_0)^{2n-1}\\&+\cancel{\frac{8}{(z-z_0)}}-4\sum_{n=2}^{\infty}(2n-2)c_n(z-z_0)^{2n-1},
       \end{split}
  \end{align*}
  so that
  \begin{equation}\label{limit0}
      \lim_{z\to z_0}-8(z-z_0)\wp(z)-4(z-z_0)^2\wp'(z)=0,
  \end{equation}
  and therefore
  \begin{equation*}
       Res_{z=z_0}\frac{dh}{g}=0.
  \end{equation*}
  Similarly
  \begin{align*}
  \begin{split}
    Res_{z=z_0}gdh&= \lim_{z\to z_0}\frac{d}{dz}(z-z_0)^2g(z)(b-4\wp(z))\\&=\lim_{z\to z_0}2(z-z_0)g(z)(b-4\wp(z))+(z-z_0)^2g'(z)(b-4\wp(z))-4(z-z_0)^2g(z)\wp'(z)\\&=\lim_{z\to z_0}-8(z-z_0)g(z)\wp(z)+(z-z_0)^2g(z)-4(z-z_0)^2g(z)\wp'(z)\\&=\lim_{z\to z_0}-8(z-z_0)g(z)\wp(z)-4(z-z_0)^2g(z)\wp'(z)\\&=g(z_0)\lim_{z\to z_0}-8(z-z_0)\wp(z)+(z-z_0)^2g(z)-4(z-z_0)^2\wp'(z)=0.
  \end{split}
  \end{align*}
  by equation \eqref{limit0}. Therefore we conclude that 
  \begin{equation*}
      Res_{z=z_0}gdh=0,
  \end{equation*}
  so that condition \eqref{residuecondition} is clearly satisfied and therefore we have a well-defined global minimal immersion on the whole $\bC\setminus L(\tau,s)$.
  
Moreover from our considerations we see that at each lattice point $z_0$
  \begin{equation*}
      ord_{z=z_0}dh= ord_{z=z_0}\frac{dh}{g}= ord_{z=z_0}gdh=2,
  \end{equation*}
  so that the winding number on Jorge-Meeks formula is exactly $1$ for each end, showing that the immersion is complete and all ends are embedded. Since the residue of the height differential is zero, we see that we have a flat planar end at each lattice point.
\end{proof}
Motivated by the previous construction, we introduce the following definitions
\begin{defn}
\normalfont
  Let $\Phi_{L,b,\hat{g}_0}$ be a Wente-Weierstrass data on $\bC\setminus L$. We define the map $\Psi_{L,b,\hat{g}_0}:\bC\setminus L\to \bR^3$ as the complete minimal immersion of Enneper type constructed from Theorem \ref{EnnepertypeImmersion}.
\end{defn}
\begin{defn}\label{minimalwentetorus}
\normalfont
Consider a lattice $L$ in the complex plane with real invariants, $b$ a real constant satisfying equation \eqref{cubicequationb} and $\hat{g}_0\in \bR$. Let $\Lambda$ be a sub-lattice of $L$ and $\vec{a}\in \bR^3$.
\begin{itemize}
    \item \textit{A minimal Wente torus with ends of type $(L,b,\hat{g}_0,\Lambda)$} is a conformal minimal immersion $$\hat{\Psi}_{L,b,\hat{g}_0}:\left(\bC\setminus L\right)/\Lambda\to \bR^3$$ such that $\hat{\Psi}_{L,b,\hat{g}_0}\circ \pi=\Psi_{L,b,\hat{g}_0}$, where $\pi:\bC\to \bC/\Lambda$ is the canonical projection map.
    \item \textit{A singly periodic minimal Wente torus with ends of type $(L,b,\hat{g}_0,\Lambda,\vec{a})$} is a conformal minimal immersion $$\hat{\Psi}_{L,b,\hat{g}_0}:\left(\bC\setminus L\right)/\Lambda\to \bR^3/\vec{a}$$ such that $\hat{\Psi}_{L,b,\hat{g}_0}\circ \pi=\pi_0\circ \Psi_{L,b,\hat{g}_0}$, where $\pi_0:\bR^3\to \bR^3/\vec{a}$ and $\pi:\bC\to \bC/\Lambda$ are the canonical projection maps.
\end{itemize}
\end{defn}
We introduce the natural generalization of the period map of \cite[Eq. 5.16 and Chapter 6]{Isabel}.
\begin{defn}
\normalfont
    Let $L=\text{span}\{2\omega_1,2\omega_2\}$ be a non-degenerate lattice with real invariants and let $b\in \bR$. We define the two period maps as the following integrals
    \begin{equation*}
        \kappa_j(b,L)\coloneqq-\frac{1}{2\pi i}\int_{\alpha_j}\frac{dz}{b-4\wp(z|L)}
    \end{equation*}
    where $\alpha_j:[0,1]\to \bC$ are lifts of the two generators of the integral homology group of the torus $\bC/L$.
\end{defn}
We now introduce a formula to calculate the periods.
\begin{prop}
 Let $L=\text{span}\{2\omega_1,2\omega_2\}$ be a non-degenerate lattice with real invariants and let $b\in \bR$. Then there exists an integer $N=N(b,L)\in \bZ$ such that
    \begin{equation}\label{Period}
        \kappa_2(b,L)=-\frac{1}{2\pi i}\left(4\zeta(\mu|L)\omega_2-4\zeta(\omega_2|L)\mu+2\pi i N\right).
    \end{equation}
\end{prop}
\begin{proof}
From equation \eqref{integralequationphi} we have in particular
\begin{equation*}
    \int_0^{2\omega_2}\frac{1}{b-4\wp(\xi)}d\xi=4\zeta(\mu)\omega_2+F(2\omega_2),
\end{equation*}
where by equation \eqref{exponentialF},
\begin{equation*}
    e^{F(2\omega_2)}=\frac{\sigma(\mu-2\omega_2)}{\sigma(\mu+2\omega_2)}.
\end{equation*}
Using the identity \cite[Eq. 5.14]{Isabel}
\begin{equation}\label{sigmaidentity}
    \sigma(z+2j\omega_1+2k\omega_2)=(-1)^{j+k+jk}\exp \left[(2j\zeta(\omega_1)+2k\zeta(\omega_2))(j\omega_1+k\omega_2+z)\right]\sigma(z),
\end{equation}
we have
\begin{equation*}
     e^{F(2\omega_2)}=e^{-4\zeta(\omega_2)\mu}.
\end{equation*}
This then implies that there exists an integer $N=N(b,L)\in \bZ$ such that
\begin{equation*}
    F(2\omega_2)=-4\zeta(\omega_2)\mu+2\pi i N,
\end{equation*}
which concludes the proof.
\end{proof}
In the particular case of rhombic lattices we can prove the result from the property $\overline{\wp}(\overline{z}|L)=\wp(z|L)$ of the Weierstrass elliptic function associated to real invariants.
  \begin{lema}\label{periodconjugates}
     Let $L=\text{span}\{2\omega_1,2\omega_2\}$ be a non-degenerate lattice of rhombic type with $2\omega_2=\overline{2\omega_1}$ and let $b\in \bR$. Then, the two periods maps are related to each other by the formula
    \begin{equation*}
        \kappa_1(b,L)=-\overline{\kappa_2(b,L)}.
    \end{equation*}
\end{lema}
\begin{lema}\label{gtranslated}
     Let $L=\text{span}\{2\omega_1,2\omega_2\}$ be a non-degenerate lattice with real invariants and let $b\in \bR$. Then \begin{equation*}
    \forall l_1,l_2\in \bZ \quad g(z+2l_1\omega_1+2l_2\omega_2)=g(z)e^{-2\pi i( l_1 \kappa_1(b,L)+ l_2\kappa_2(b,L))}.
\end{equation*}
\end{lema}
\begin{proof}
    In fact,
\begin{align*}
    \begin{split}
g(z+2l_1\omega_1+2l_2\omega_2)=&\hat{g}_0\exp\left(\int_{0}^{z+2l_1\omega_1+2l_2\omega_2}\frac{d\xi}{\phi(\xi)}\right)=\hat{g}_0\exp\left(\int_0^{2l_1\omega_1+2l_2\omega_2}\frac{d\xi}{\phi(\xi)}+\int_{2l_1\omega_1+2l_2\omega_2}^{z+2l_1\omega_1+2l_2\omega_2}\frac{d\xi}{\phi(\xi)}\right)\\=&\hat{g}_0\exp\left(\int_0^{2l_1\omega_1}\frac{d\xi}{\phi(\xi)}+\int_{2l_1\omega_1}^{2l_1\omega_1+2l_2\omega_2}\frac{d\xi}{\phi(\xi)}+\int_{0}^{z}\frac{d\xi}{\phi(\xi)}\right)\\=&g(z)\exp\left(\int_0^{2l_1\omega_1}\frac{d\xi}{\phi(\xi)}+\int_{0}^{2l_2\omega_2}\frac{d\xi}{\phi(\xi)}\right)\\=&g(z)\exp\left(\int_0^{2\omega_1}\frac{d\xi}{\phi(\xi)}+\ldots + \int_{(l_1-1)2\omega_1}^{2l_1\omega_1}\frac{d\xi}{\phi(\xi)}+\int_0^{2\omega_2}\frac{d\xi}{\phi(\xi)}+\ldots + \int_{(l_2-1)2\omega_2}^{2l_2\omega_2}\frac{d\xi}{\phi(\xi)}\right)\\=&g(z)\exp\left(l_1\int_{0}^{2\omega_1}\frac{d\xi}{\phi(\xi)}+l_2\int_{0}^{2\omega_2}\frac{d\xi}{\phi(\xi)}\right)
    \end{split}
\end{align*}
which proves the claim. 
\end{proof}
The next Theorem shows that it is impossible to produce a minimal Wente torus with ends starting from the rectangular lattices.
\begin{thm}\label{importancerombic}
    Suppose that $\wp$ is the Weierstrass function associated to a non-degenerate lattice $L=\text{span}\{2\omega_1,2\omega_2\}$ with real invariants. Let $b\in\bR$, $\hat{g}_0\in \bR^+$ and let $$\Lambda=\text{span}\{2m_1\omega_1+2m_2\omega_2,2n_1\omega_1+2n_2\omega_2\}\subset L$$ be any sub-lattice of $L$. If the minimal immersion $\Psi_{L,b,\hat{g}_0}$ factors to $\left(\bC\setminus L\right)/\Lambda$, then the lattice $L$ has to be rhombic and moreover the period maps $\kappa_j(b,L)$ are real rational numbers.
\end{thm}
\begin{proof}
Since the immersion $\Psi_{L,b,\hat{g}_0}$ factors to the quotient $\left(\bC\setminus L\right)/\Lambda$, this means, using the Cauchy-Riemann equations, that the Weierstrass data $\Phi_{L,b,\hat{g}_0}$ has to be \textit{elliptic} (invariant) with respect to the lattice $\Lambda$, \textit{i.e.} $\Phi_{L,b,\hat{g}_0}(z+\Lambda)=\Phi_{L,b,\hat{g}_0}(z)$. In particular, since the function $\phi$ is already elliptic in $L$, it is automatically elliptic in $\Lambda$ and we have that the Gauss map must be elliptic in $\Lambda$ as well. Suppose by contradiction that $\Delta_{mod}>0$, \textit{i.e.} the lattice $L$ is rectangular and generated by $2\omega_1\in \bR$, $2\omega_2\in i\bR$. Since the generators of $\Lambda$ are linearly independent, we must have that $m_1$ and $n_1$ can not vanish simultaneously. Without loss of generality, suppose $m_1\neq 0$. Then, by the ellipticity of $g(z)$ in $\Lambda$, we have
\begin{equation*}
    g(z+2m_1\omega_1+2m_2\omega_2)=g(z).
\end{equation*}
 Since by hypothesis $\wp(\mu)=\frac{b}{4},\wp'(\mu)=-\frac{1}{4}\in \bR$ we must have $\mu\in (0,2\omega_1).$ Consider a path $\alpha\coloneqq \alpha_1*\alpha_2*\alpha_3$ joining $0$ to $2m_1\omega_1+2m_2\omega_2$ such that $\alpha_1$ is a vertical path going from $0$ to $\omega_2$, $\alpha_2$ a horizontal path from $\omega_2$ to $2m_1\omega_1+\omega_2$ and finally $\alpha_3$ a vertical path from $2m_1\omega_1+\omega_2$ to $2m_1\omega_1+2m_2\omega_2$. We have
        \begin{gather*}
            \int_{\alpha}\frac{dz}{\phi(z)}=\int_{\alpha_1}\frac{dz}{\phi(z)}+\int_{\alpha_2}\frac{dz}{\phi(z)}+\int_{\alpha_3}\frac{dz}{\phi(z)}.
        \end{gather*}
        Let $\tilde{\alpha}_1$ be a vertical path going from $2m_1\omega_1$ to $2m_1\omega_1+\omega_2$. Then by the $2\omega_1$-periodicity of $\phi$, we have that
        \begin{equation*}
            \int_{\alpha_1}\frac{dz}{\phi(z)}=\int_{\tilde{\alpha}_1}\frac{dz}{\phi(z)}.
        \end{equation*}
        Therefore we can write
        \begin{equation*}
             \int_{\alpha}\frac{dz}{\phi(z)}= \int_{\alpha_2}\frac{dz}{\phi(z)}+ \int_{\tilde{\alpha}_1*\alpha_3}\frac{dz}{\phi(z)}.
        \end{equation*}
  The Weierstrass $\wp$ function on the rectangular lattice $L$ is real along the vertical lines $u=k\omega_1$ and along the horizontal lines $v=k\abs{\omega_2}$ with $k\in \bZ$. Moreover the $\wp$ function along $v=\abs{\omega_2}$ is strictly less than the $\wp$ function along the real axis $v=0$, so we have the following consequences:
        \begin{gather}\label{integralalpha2}
            \int_{\alpha_2}\frac{dz}{\phi(z)} =\frac{1}{4}\int_0^1\frac{2m_1\omega_1dt}{\wp(\mu)-\wp((1-t)\omega_2)+(2m_1\omega_1+\omega_2)t)}\in \bR^{*},
        \end{gather}
        \begin{gather}\label{integralalpha1alpha3}
            \int_{\tilde{\alpha}_1*\alpha_3}\frac{dz}{\phi(z)} =\frac{1}{4}\int_0^1\frac{2m_2\omega_2dt}{\wp(\mu)-\wp((1-t)2m_1\omega_1)+(2m_1\omega_1+2m_2\omega_2)t)}\in i\bR.
        \end{gather}
        Therefore, by using a change of variable and the $2\omega_1$,$2\omega_2$-periodicity of $\phi$, we have
        \begin{align*}
        \begin{split}
           g(z)&= g(z+2m_1\omega_1+2m_2\omega_2)=\hat{g}_0\exp\left(\int_0^{z+2m_1\omega_1+2m_2\omega_2}\frac{d\xi}{\phi(\xi)}\right)\\&=\hat{g}_0\exp\left(\int_0^{2m_1\omega_1+2m_2\omega_2}\frac{d\xi}{\phi(\xi)}+\int_{2m_1\omega_1+2m_2\omega_2}^{z+2m_1\omega_1+2m_2\omega_2}\frac{d\xi}{\phi(\xi)}\right)\\&=\hat{g}_0\exp\left(\int_{\alpha}\frac{d\xi}{\phi(\xi)}+\int_{0}^{z}\frac{d\xi}{\phi(\xi)}\right)=g(z)\exp\left(\int_{\tilde{\alpha}_1*\alpha_3}\frac{d\xi}{\phi(\xi)}\right)\exp\left(\int_{\alpha_2}\frac{d\xi}{\phi(\xi)}\right).
        \end{split}
        \end{align*}
       Therefore, by taking the module in both sides and using equation \eqref{integralalpha1alpha3}, we obtain
        \begin{equation*}
            \int_{\alpha_2}\frac{d\xi}{\phi(\xi)}=0,
        \end{equation*}
        which is a contradiction with \eqref{integralalpha2}. Hence the lattice $L$ must be rhombic.
        
For the second part of the proof, we show that the period maps $\kappa_1$ and $\kappa_2$ are real rationals. In fact, since $g$ is elliptic in $\Lambda$ we must have by Lemma \ref{gtranslated} that 
\begin{equation}\label{pairsintegers}
    m_1\kappa_1+m_2\kappa_2\in \bZ,\quad n_1\kappa_1+n_2\kappa_2\in \bZ 
\end{equation}
By Lemma \ref{periodconjugates} we must have that
\begin{equation}\label{periodscomplex}
    \kappa_1=a+ib,\quad \kappa_2=-a+ib
\end{equation}
After substituting equation \ref{periodscomplex} into equation \ref{pairsintegers} and taking real and imaginary parts we obtain
\begin{gather*}
    \begin{cases}
        (m_1-m_2)a\in \bZ,\quad (n_1-n_2)a\in \bZ\\ (m_1+m_2)b=0,\quad (n_1+n_2)b=0
    \end{cases}
\end{gather*}
We claim that $b=0$. Indeed, if we assume by contradiction that $b\neq 0$, then
\begin{equation*}
    m_2=-m_1,\quad n_2=-n_1
\end{equation*}
which would imply that the generators of $\Lambda$
\begin{gather*}
    \begin{cases}
        2m_1\omega_1+2m_2\omega_2=2m_1(\omega_1-\omega_2)\\ 2n_1\omega_1+2n_2\omega_2=2n_1(\omega_1-\omega_2)
    \end{cases}
\end{gather*}
are not linearly independent, which is a contradiction. Hence, it must be that $b=0$, and therefore $\kappa_1,\ \kappa_2\in \bQ$.
\end{proof}
\begin{rem}\label{singlyperiodicperiod}
\normalfont
    A slight modification on the proof of Theorem \ref{importancerombic} shows that for singly periodic minimal Wente torus with ends the Wente-Weierstrass data is also elliptic in the corresponding sublattice, implying that the lattice is rhombic and the period map is also a real rational number.
\end{rem}
We end this section with the following useful tool
\begin{lema}[{\textit{Cf. }\cite[Lemma 5.5]{Isabel}}]\label{translationalperiod}
 Let $L=\text{span}\{2\omega_1,2\omega_2\}$ be a lattice in the complex plane and let $\Phi_{L,b,\hat{g}_0}$ be a Wente-Weierstrass data on $\bC\setminus L$. If the period map $\kappa_2(b,L)=\frac{m}{n}$ for some $m,n\in \bZ$ with $\gcd(m,n)=1$, then, there exists a translational period $\vec{a}=\vec{a}_{L,b,\hat{g}_0}\in \bR^3$ such that
    \begin{gather}\label{translationalequation}
     \forall z\in \bC\setminus L\quad \Psi_{L,b,\hat{g_0}}(z+2n\omega_1)=\Psi_{L,b,\hat{g_0}}(z)+\vec{a},\quad \Psi_{L,b,\hat{g_0}}(z+2n\omega_2)=\Psi_{L,b,\hat{g_0}}(z)+\vec{a}.
    \end{gather}
\end{lema}
\begin{proof}
First, notice that by Theorem \ref{EnnepertypeImmersion} the Weierstrass data $\Phi_{L,b,\hat{g}_0}$ induces a complete minimal immersion $\Psi_{L,b,\hat{g}_0}:\bC\setminus L\to \bR^3$ such that every $v$-line $\Psi_{L,b,\hat{g}_0}(u_0,\cdot)$ is a spherical line of curvature.

The hypothesis $\kappa_2(b,L)\in \bQ$ implies that the lattice $L$ is of rhombic type by Theorem \ref{importancerombic}. Hence from Lemma \ref{periodconjugates} we have that 
    \begin{equation}\label{kapasrational}
        \kappa_1(b,L)=-\overline{\kappa_2(b,L)}=-\frac{m}{n}\in \bQ,
    \end{equation}
    On one hand, from the definition of $\phi(z)=b-4\wp(z)$ we see that, for $j=1,2$,
    \begin{gather*}
        \phi(z+2n\omega_j)=\phi(z).
    \end{gather*}
   On the other hand, by an application of Lemma \ref{gtranslated} we obtain
    \begin{equation*}
         g(z+2n\omega_j)=g(z)e^{-2\pi i n \kappa_j(b,L)},
    \end{equation*}
    which implies by equation \eqref{kapasrational} that the Gauss map satisfies
    \begin{equation*}
         g(z+2n\omega_j)=g(z),\quad j=1,2.
    \end{equation*}
    Therefore the Wente-Weierstrass data has the property
    \begin{equation*}
        \Phi_{L,b,\hat{g_0}}(z+2n\omega_j)= \Phi_{L,b,\hat{g_0}}(z),\quad j=1,2,
    \end{equation*}
    which implies that
    \begin{gather*}
        \Psi_{L,b,\hat{g_0}}(z+2n\omega_j)= \Psi_{L,b,\hat{g_0}}(z)+\vec{a}_j,
    \end{gather*}
    where
    \begin{equation*}
        \vec{a}_j\coloneqq Re\int_{z_0}^{z_0+2n\omega_j}\Phi_{L,b,\hat{g_0}}(\xi)d\xi,\quad j=1,2.
    \end{equation*}
    We can check that each $\vec{a}_j$ independ of the choice of base point $z_0$ due to the periodicity property of $\Phi_{L,b,\hat{g_0}}$.
    
    Now fix a point $z_{*}\in \bC\setminus L$. Notice that the direction $\omega_2-\omega_1$ is parallel to the $v$ coordinate axis which is the direction of the spherical line of curvature. We then have
    \begin{equation*}
         \Psi_{L,b,\hat{g_0}}(z_{*}+2n(\omega_2-\omega_1))=\Psi_{L,b,\hat{g_0}}(z_{*})+\vec{a}_2-\vec{a}_1.
    \end{equation*}
    Since the $v$-lines are spherical lines of curvature, and since spheres are compact (in particular fixing $z_{*}$ in a way that the $v$-line does not go through and end point), we conclude by iterating the previous two equations that we must have
    \begin{equation*}
        \vec{a}_2-\vec{a}_1=0,
    \end{equation*}
    which implies that $\vec{a}_1=\vec{a}_2\coloneqq \vec{a}$, finishing our proof.
\end{proof}
\section{Construction}\label{Section 5.3}
In this section we investigate the existence of a minimal Wente torus with ends in $\bR^3$. Motivated by the work \cite{Isabel}, we begin by establishing a construction which we prove to be canonical. 

Consider the two parameter space 
\begin{equation*}
    (\tau,s)\in \Omega=\bR^+\times(0,2)\subset\bR^2 .
\end{equation*}
For all $\tau \in \bR^+$, define half-periods
\begin{gather*}
    \tilde{\omega}_1(\tau)=1-i\tau,\quad \tilde{\omega}_2(\tau)=1+i\tau,\quad  \tilde{\omega}_3(\tau)=\tilde{\omega}_1(\tau)+\tilde{\omega}_2(\tau)=2,
\end{gather*}
which generate a rhombic lattice 
\begin{equation*}
    \Tilde{L}=\Tilde{L}(\tau)\coloneqq\text{span}\{2\Tilde{\omega}_1,2\Tilde{\omega_2}\}.
\end{equation*}
This lattice has associated real invariants
\begin{gather*}
    \Tilde{g}_2(\tau)=60\sum_{(m,n)\in \bZ_2^{*}}\frac{1}{(2m\Tilde{\omega}_1(\tau)+2n\Tilde{\omega}_2(\tau))^4}\in \bR,\quad  \Tilde{g}_3(\tau)=140\sum_{(m,n)\in \bZ_2^{*}}\frac{1}{(2m\Tilde{\omega}_1(\tau)+2n\Tilde{\omega}_2(\tau))^6}\in \bR,
\end{gather*}
which actually depend only on the lattice and not on its generators. By construction, the discriminant satisfies
\begin{equation*}
    \Tilde{\Delta}_{mod}=\Tilde{g}_2^3-27\Tilde{g}_3^2<0.
\end{equation*}
We consider the Weierstrass elliptic function associated to this lattice $\Tilde{L}$,
\begin{equation*}
\wp(z|\Tilde{L})=\wp(z|\tilde{\omega}_1(\tau),\tilde{\omega}_2(\tau)).
\end{equation*} There is an interesting characterization of the sign of the invariants $\Tilde{g}_2,\ \Tilde{g}_3$ in terms of the shape of the lattice
\begin{thm}[{\cite[Theorem 2.9]{duval}}]\label{signinvariants}
    Let $\tau\in \bR^+$ and $\theta_\tau\in (0,\frac{\pi}{2})$ such that $\tau=\tan \theta_\tau$. The invariants $\Tilde{g}_2(\tau)$ and $\Tilde{g}_3(\tau)$ associated to the latice $\Tilde{L}(\tau)$ satisfy
    \begin{gather*}
        \begin{cases}
         \Tilde{g}_2(\tau)<0\quad \text{if $\theta_\tau\in(\frac{\pi}{6},\frac{\pi}{3})$},\\\Tilde{g}_2(\tau)=0\quad \text{if $\theta_\tau\in \{\frac{\pi}{6},\frac{\pi}{3}\}$},\\ \Tilde{g}_2(\tau)>0\quad \text{if $\theta_\tau\in(0,\frac{\pi}{6})\cup(\frac{\pi}{3},\frac{\pi}{2})$}.
    \end{cases}\quad  \begin{cases}
         \Tilde{g}_3(\tau)<0\quad \text{if $\theta_\tau\in (0,\frac{\pi}{4})$},\\\Tilde{g}_3(\tau)=0\quad \text{if $\theta_\tau=\frac{\pi}{4}$},\\ \Tilde{g}_3(\tau)>0\quad \text{if $\theta_\tau\in(\frac{\pi}{4},\frac{\pi}{2})$}.
    \end{cases}
    \end{gather*}
\end{thm}
It is known \cite[Section 20]{duval} that the Weierstrass function is only real along the diagonals of the rhombus, and strictly decreasing moving along $[0,\tilde{\omega}_3]$, and then keep decreasing along $[\tilde{\omega}_3,2\tilde{\omega}_2]$. Therefore,
\begin{equation}\label{derivativeP}
    \forall\ 0<s<2=\Tilde{\omega}_3,\quad \wp'(s|\tilde{\omega}_1,\tilde{\omega}_2)<0.
\end{equation}
Consider
\begin{equation*}
    \lambda(\tau,s)\coloneqq \sqrt[3]{-4\wp'(s|\tilde{\omega}_1(\tau),\tilde{\omega}_2(\tau))}>0,
\end{equation*}
and define
\begin{equation*}
    M(\tau,s)=\lambda(\tau,s)\cdot s.
\end{equation*}
We now introduce a rescaling of the half periods of the rhombic lattice $\{\tilde{\omega}_1,\tilde{\omega_2},\tilde{\omega_3}\}$
\begin{gather*}
    \omega_1(\tau,s)=\lambda(\tau,s)\cdot \tilde{\omega}_1(\tau),\quad \omega_2(\tau,s)=\lambda(\tau,s)\cdot \tilde{\omega}_2(\tau),\quad \omega_3(\tau,s)=\lambda(\tau,s)\cdot \tilde{\omega}_3(\tau),
\end{gather*}
which induce a rescaled lattice
\begin{equation}\label{rescaledlattice}
    L=L(\tau,s)\coloneqq\text{span}\{2\omega_1(\tau,s),2\omega_2(\tau,s)\},
\end{equation}
with associated invariants $g_2,\ g_3$, which are also rescalings of $\Tilde{g}_2, \Tilde{g}_3$. By construction, 
\begin{equation*}
    0<M(\tau,s)<\omega_3(\tau,s).
\end{equation*}
Using the rescaling properties of the Weierstrass elliptic functions \cite[Section 18.2]{Handbook}, we have that 
\begin{equation}
    \wp'(M(\tau,s)|\omega_1,\omega_2)=\wp'(\lambda\cdot s|\lambda\cdot \tilde{\omega}_1,\lambda\cdot \tilde{\omega}_2)=\frac{1}{\lambda^3}\wp'(s|\tilde{\omega}_1,\tilde{\omega}_2)=-\frac{1}{4},
\end{equation}
that is
\begin{equation}\label{eqmu}
    \wp'(M(\tau,s)|\omega_1,\omega_2)=-\frac{1}{4}.
\end{equation}
Next, we define the real constant 
\begin{equation}\label{defb}
    b(\tau,s)\coloneqq 4\wp(M(\tau,s)|\omega_1(\tau,s),\omega_2(\tau,s))\in \bR.
\end{equation}
By the differential equation satisfied by the Weierstrass elliptic function \eqref{differentialequationp}, we obtain that $b(\tau,s)$ satisfies the equation
\begin{gather*}
    b^3-4g_2b-16g_3=16\left(4\wp^3(\mu)-g_2\wp(\mu)-g_3\right)=\left(4\wp'(\mu)\right)^2=1,
\end{gather*}
that is
\begin{equation}\label{cubicequation2}
    b^3-4g_2b-16g_3=1.
\end{equation}
Notice that by Definition \ref{mudefinition} and equations \eqref{eqmu}, \eqref{defb}, \eqref{cubicequation2} we must have
\begin{equation}\label{equivalentdefinitions}
    M(\tau,s)=\mu(b(\tau,s),L(\tau,s))
\end{equation}
\begin{cor}\label{construWenteWeierstrassdata}
    For each point $(\tau,s,\hat{g}_0)\in \Omega\times \bR^+$, the one-form $\Phi_{L(\tau,s),b(\tau,s),\hat{g}_0}$ given by equations \eqref{Weierstrassdata}, \eqref{phifunction}, \eqref{gaussmap}, \eqref{rescaledlattice} and \eqref{defb} is a well-defined Wente-Weierstrass data that induces by Theorem \ref{EnnepertypeImmersion}, a conformal minimal immersion
\begin{equation}\label{inducedminimalimmersion}
\Psi_{L(\tau,s),\hat{g}_0}\coloneq \Psi_{L(\tau,s),b(\tau,s),\hat{g}_0}:\bC\setminus L(\tau,s)\to \bR^3,
\end{equation}
 with infinitely many embedded planar ends at the lattice points $L(\tau,s)$ such that every $v$-line $\Psi_{L,b,\hat{g}_0}(u_0,\cdot)$ is a spherical line of curvature.
\end{cor}
Now we study the possibility to choose the parameters adequately in $\Omega\times \bR^+$ so that our minimal immersion descends to $\bC/\Lambda$, where $\Lambda$ is an appropriate sub-lattice of $L$. According to Theorem \ref{importancerombic} a necessary condition is that the period map $$\kappa_2(\tau,s)\coloneq \kappa_2(b(\tau,s),L(\tau,s))$$ is a real rational number. We break the analysis into two steps. The first step is to find a curve $C$ inside $\Omega$ such that the period map $\kappa_2$ is purely real along it. The second step is to prove that the period map $\kappa_2$ is not constant along $C$. As a consequence of these two steps, by continuity and the intermediate value theorem, the period $\kappa_2$ will assume rational values.
\begin{cor}\label{realimaginaryperiods}
    The real and imaginary parts of the period $\kappa_2(\tau,s)$ can be written as
    \begin{equation*}
         \pi\ Im\ \kappa_2(\tau,s)=2\zeta(s|\Tilde{L})-\zeta(2|\Tilde{L}) s,
    \end{equation*}
    \begin{equation*}
        -i\pi\ Re\ \kappa_2(\tau,s)-i\pi N(\tau,s)=2i\tau\zeta(s|\Tilde{L})-\zeta(2i\tau|\Tilde{L})s.
    \end{equation*}
\end{cor}
\begin{proof}
    From equation \eqref{equivalentdefinitions}, we have that can take the imaginary part in equation \eqref{Period}, we have
    \begin{gather*}
    Im\ \kappa_2(\tau,s)=\frac{2}{\pi}Re\ \left(\zeta(\lambda\cdot s|\lambda\cdot \Tilde{\omega}_1,\lambda\cdot \Tilde{\omega}_2)\lambda\cdot \Tilde{\omega}_2-\zeta(\lambda\cdot \Tilde{\omega}_2|\lambda\cdot \Tilde{\omega}_1,\lambda\cdot \Tilde{\omega}_2)\lambda\cdot s\right).
\end{gather*}
Using the homogeneity relations \cite[Section 18.2]{Handbook} and the identity \cite[Eq. 18.4.3]{Handbook}
\begin{equation*}
\zeta(\Tilde{\omega}_1|\Tilde{L})+\zeta(\Tilde{\omega}_2|\Tilde{L})=\zeta(\Tilde{\omega}_1+\Tilde{\omega}_2|\Tilde{L}),
\end{equation*}
we obtain
\begin{align*}
\begin{split}
    Im\ \kappa_2(\tau,s)&=\frac{2}{\pi}Re\ \left(\zeta( s| \Tilde{L}) \Tilde{\omega}_2-\zeta( \Tilde{\omega}_2| \Tilde{L})\ s\right)\\&=\frac{1}{\pi}\left(\zeta( s| \Tilde{L}) \Tilde{\omega}_2-\zeta( \Tilde{\omega}_2| \Tilde{L})\ s+\overline{\zeta( s|\Tilde{L}) \Tilde{\omega}_2-\zeta( \Tilde{\omega}_2| \Tilde{L})\ s}\right)\\&=\frac{1}{\pi}\left(\zeta( s| \Tilde{L}) \Tilde{\omega}_2-\zeta( \Tilde{\omega}_2| \Tilde{L}) s+\zeta( \overline{s}|\Tilde{L}) \overline{\Tilde{\omega}_2}-\zeta( \overline{\Tilde{\omega}_2}| \Tilde{L}) \overline{s}\right)\\&= \frac{1}{\pi}\left(\zeta( s| \Tilde{L}) \Tilde{\omega}_2-\zeta( \Tilde{\omega}_2|\Tilde{L}) s+\zeta(s| \Tilde{L}) \Tilde{\omega}_1-\zeta( \Tilde{\omega}_1| \Tilde{L}) s\right)\\&= \frac{1}{\pi}\left(\zeta(s| \Tilde{L}) \Tilde{\omega}_3-\zeta( \Tilde{\omega}_3| \Tilde{L}) s\right).
    \end{split}
\end{align*}
A similar argument taking real part in equation \eqref{Period} shows the second claim.
\end{proof}
In order to identify the rhombic lattices for which the period $\kappa_2$ is purely real we need some preliminary lemmas
\begin{lema}\label{functionF}
    Consider the function $F:\bR^+\to \bR$ defined by
    \begin{equation*}
        F(\tau)\coloneqq \zeta(2|\Tilde{L}(\tau))+2\wp(2|\Tilde{L}(\tau)).
    \end{equation*}
    Then this function has the following properties
    \begin{enumerate}
    \item[I.] $F$ is monotone increasing.
     \item[II.] $\lim_{\tau\to +\infty}F(\tau)=\frac{\pi^2}{8}>0$.
    \item[III.] $\lim_{\tau\to 0^+}F(\tau)=-\infty$.
    \item[IV.] There exists a unique critical value $\tau_0<1$ such that $F(\tau_0)=0$.
        \item[V.] $\forall \tau\in(0,\tau_0)$, $F(\tau)<0$.
        \item[VI.] $\forall \tau\in(\tau_0,\infty)$, $F(\tau)>0$.
        
     \end{enumerate}
\end{lema}

\begin{proof}
    We have the following uniformly convergent Fourier series of the Weierstrass $\wp$ function \cite[Eq. 23.8.1]{manual}
    \begin{equation*}
        \wp(2|\Tilde{L})+\frac{\zeta(2|\Tilde{L})}{2}-\frac{\pi^2}{16}=-\frac{\pi^2}{2}\sum_{n=1}^\infty\frac{nq^{2n}}{1-q^{2n}}\cos\left(n\pi \right),
    \end{equation*}
    where according to our choice of lattice $\Tilde{L}(\tau)$
    \begin{equation*}
        q=e^{\frac{i\pi \Tilde{\omega}_2}{\Tilde{\omega}_3}}=e^{i\pi \frac{1+i\tau}{2}}=ie^{-\frac{\pi \tau}{2}}.
    \end{equation*}
    Therefore
    \begin{equation*}
        F(\tau)=\frac{\pi^2}{8}-\pi^2\sum_{n=1}^\infty\frac{n}{e^{\pi n\tau}-(-1)^{n}}.
    \end{equation*}
    Since the convergence is uniformly we can take the derivative by differentiating each term of the series
    \begin{equation*}
        F'(\tau)=\pi^3\sum_{n=1}^{\infty}\frac{ n^2e^{\pi n \tau}}{\left(e^{\pi n\tau}-(-1)^n\right)^2}>0.
    \end{equation*}
    This shows that the function $F(\tau)$ is monotone increasing and proves Item I. For Item II, we start by fixing $\tau>0$ large enough so that $e^{\pi\tau}\geq 2$. Then for all $n\geq 1$ we have that
    \begin{equation*}
        1-(-1)^ne^{-n\pi\tau}\geq \frac{1}{2}.
    \end{equation*}
    On the other hand it also holds
    \begin{equation*}
        e^{\frac{n\pi\tau }{2}}\geq 1+\frac{n\pi \tau}{2}\geq \frac{n\pi \tau}{2},
    \end{equation*}
    so that for the choice of $\tau$ large enough, we have
    \begin{equation*}
        \frac{n}{e^{n\pi \tau}-(-1)^n}=\frac{1}{e^{\frac{n\pi \tau}{2}}}\left(\frac{\frac{n}{e^{\frac{n\pi \tau}{2}}}}{1-(-1)^ne^{-n\pi\tau}}\right)\leq \frac{4}{\pi\tau}\frac{1}{e^{\frac{n\pi \tau}{2}}}.
    \end{equation*}
    Now since the geometric series converges to
    \begin{equation*}
        \frac{4}{\pi \tau}\sum_{n=1}^{\infty}\left(e^{-\frac{\pi\tau}{2}}\right)^n= \frac{4}{\pi \tau}\frac{e^{-\frac{\pi \tau}{2}}}{1-e^{-\frac{\pi\tau}{2}}}\to0\quad \text{as\ } \tau\to +\infty,
    \end{equation*}
    it follows that $\lim_{\tau\to +\infty}F(\tau)=\frac{\pi^2}{8}.$ For Item III, we chose $\tau>0$ small enough so that $\frac{1}{2}\leq e^{\pi \tau}\leq 2$. This implies that, for all $n\geq 6$,
    \begin{equation*}
 \frac{n}{e^{\pi \tau n}-(-1)^n}\geq \frac{1}{e^{\pi (n-1)\tau}}.
    \end{equation*}
    Since
    \begin{equation*}
        \sum_{n=0}^{\infty}e^{-n\pi \tau}=\frac{1}{1-e^{-\pi \tau}}\to +\infty\ \text{as\ }\ \tau\to 0^+,
    \end{equation*}
    it follows that $\lim_{\tau\to 0^+}F(\tau)=-\infty$.
\end{proof}
The existence part of Item IV follows as a consequence of the intermediate value theorem, whereas the uniqueness follows from Item 1 and Rolle's Theorem. We point out that $\tau_0<1$ because it is known that if $\tau\geq 1$, in which case $\Tilde{g}_3\geq 0$ by Theorem \ref{signinvariants}, then from \cite[Eq. 18.3.7, 18.3.10]{Handbook} $\zeta(2|\Tilde{L})>0$ and $\wp(2|\Tilde{L})\geq 0$, and therefore $F(\tau)>0$ for $\tau\geq 1$. Items V and VI follow automatically.
\begin{lema}\label{criterio}
        Fix $\tau\in\bR^+$ and consider the function
        \begin{equation*}
            f_{\tau}:(0,2]\to \bR
        \end{equation*}
        given by
        \begin{equation*}
            f_{\tau}(s)=2\zeta(s|\Tilde{L})-\zeta(2|\Tilde{L}) s.
        \end{equation*}
        \begin{itemize}
            \item[I.] If $0<\tau<\tau_0$ then the function $f_{\tau}$ has a unique zero $s_0$ in the interval $(0,2)$. Moreover $f_{\tau}'(s_0)< 0$.
            \item[II.] If $\tau_0\leq\tau$  then the function $f_{\tau}$ does not have any zeros on the interval $(0,2)$.
        \end{itemize}
    \end{lema}

    \begin{figure}
     \centering
     \begin{subfigure}[b]{0.45\textwidth}
         \centering
         \includegraphics[width=\textwidth]{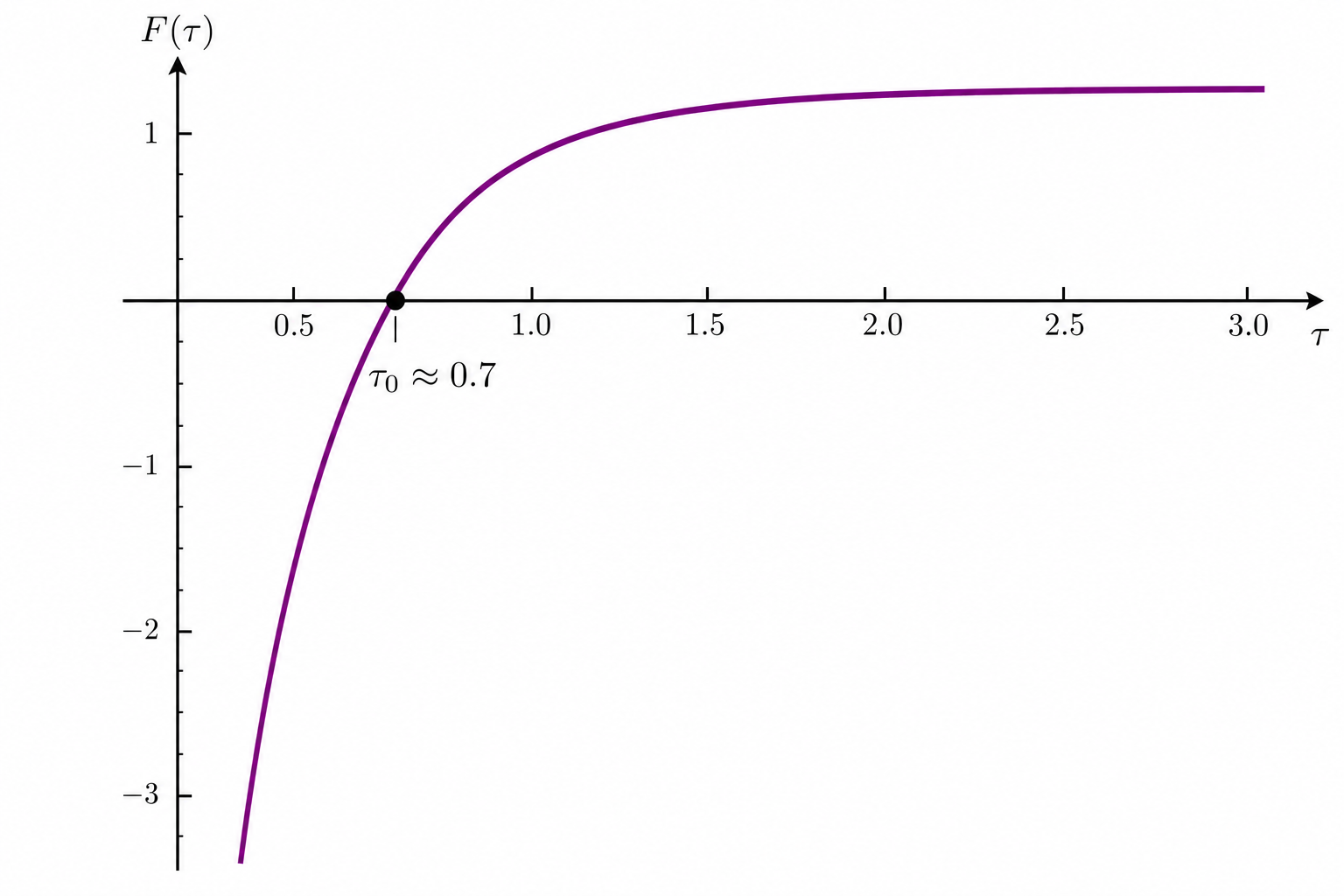}
         \caption{The function $F:\bR^+\to \bR$.}
         \label{fig:Ftau}
     \end{subfigure}
     \hfill
     \begin{subfigure}[b]{0.45\textwidth}
         \centering
         \includegraphics[width=\textwidth]{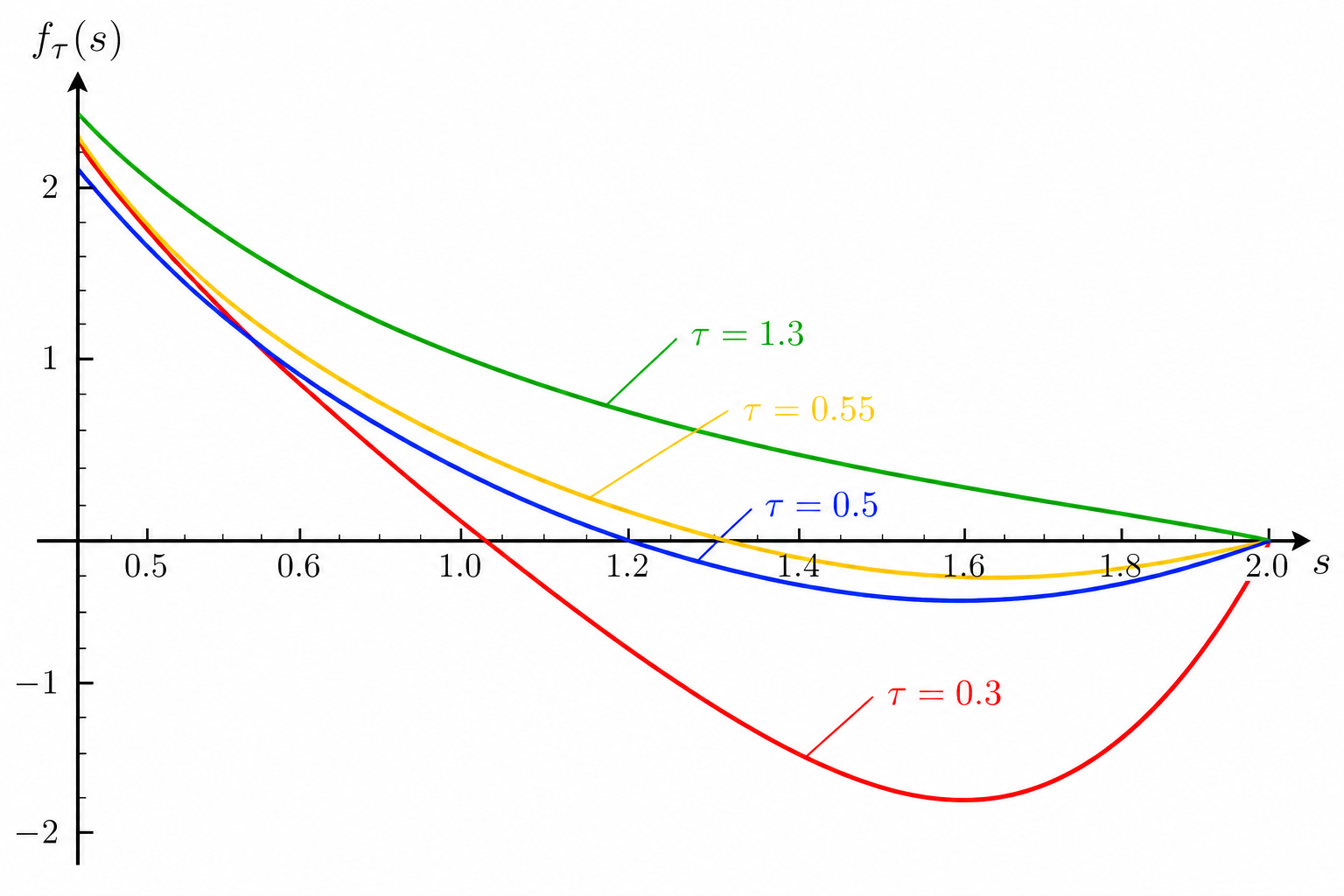}
         \caption{The function $f_{\tau}:(0,2]\to \bR$ for different values of the $\tau$ parameter.}
         \label{fig:sturmliouville}
     \end{subfigure}
     \caption{The functions $F$ and $f_\tau$ of Lemma \ref{functionF} and Lemma \ref{criterio}, respectively.}
     \end{figure}
    \begin{proof}
    Notice that, independently of the value of $\tau\in \bR^+$, we have the general properties from the definition
    \begin{equation*}
        \lim_{s\to 0^+}f_{\tau}(s)=+\infty,\quad f_{\tau}(2)=0.
    \end{equation*}
    Furthermore
    \begin{equation*}
        f_{\tau}'(s)=-2\wp(s|\Tilde{L})-\zeta(2|\Tilde{L}),\quad f_{\tau}''(s)=-2\wp'(s|\Tilde{L}).
    \end{equation*}
    By equation \eqref{derivativeP} we see that 
    \begin{equation*}
        f_{\tau}''(s)>0,\ \forall s\in(0,2).
    \end{equation*}
        \begin{itemize}            \item[I.] In this case we have that 
        \begin{equation*}
            \lim_{s\to 2^-}f_{\tau}'(s)=-F(\tau)>0,
        \end{equation*}
        and since
        \begin{equation*}
            \lim_{s\to 0^+}f_{\tau}'(s)=-\infty,
        \end{equation*}
        then, by the intermediate value theorem, there there exists $s_1\in (0,2)$ such that $f_{\tau}'(s_1)=0$. Notice that it cannot be that $f_{\tau}(s_1)>0$, otherwise since the function $f_{\tau}$ is concave up, then $f_{\tau}(s_1)>0$ will be a global minimum of the function $f_{\tau}$ in the interval $(0,2]$ contradicting the fact that $f_{\tau}(2)=0$. Therefore it has to be that $f_{\tau}(s)<0$ and then, by the intermediate value theorem, there exits $s_0\in (0,s_1)$ such that $f_{\tau}(s_0)=0$. Notice also that at this point we have $f'_{\tau}(s_0)<0.$
        
        For the uniqueness, notice that since $f_{\tau}''(s)>0$ and since $f_{\tau}'(s_1)=0$ then $f_{\tau}(s_1)<0$ is a global minimum of the function so that for all $s\in(0,s_1)$ $f_{\tau}'(s)<0$ and for all $s\in(s_1,2)$ $f_{\tau}'(s)>0$. Assume by contradiction that there exists a point $s_0^*\neq s_0$ $s_0^*\in (0,2)$ such that $f_{\tau}(s_0^*)=0$, then clearly $s_0^*\neq s_1$ so that $s_0^*\in (0,s_0)\cup (s_0,s_1)\cup (s_1,2)$. In any case by Rolle's theorem there will exists $s_2\in (0,s_0)\cup (s_0,s_1)\cup (s_1,2)$ such that $f_{\tau}'(s_2)=0$ contradicting the fact that $f_{\tau}$ is strictly decreasing in $(0,s_1)$ and strictly increasing in $(s_1,2)$ showing the uniqueness of $s_0\in (0,2)$ such that $f_{\tau}(s_0)=0.$ 
            \item[II.] In this case we have that, for all $s\in(0,2)$,
            \begin{gather*}
                 f_{\tau}'(s)=-2\wp(s|\Tilde{L})-\zeta(2|\Tilde{L}) \leq-2\wp(2|\Tilde{L})-\zeta(2|\Tilde{L})=-F(\tau)<0, 
            \end{gather*}  because the Weierstrass elliptic function is real and decreasing along the horizontal diagonal segment $(0,2)$. If we suppose by contradiction that there exists $s_0\in (0,2)$ such that $f_{\tau}(s_0)=0$, then by Rolle's theorem there will exists a time $s_1\in (s_0,2)$ such that $f_{\tau}'(s_1)=0$ contradicting the fact that $f_{\tau}$ is strictly decreasing.
        \end{itemize}
    \end{proof}
    \begin{cor}\label{curveCs}
        There exists an open subset $\Omega_{\tau_0}\coloneqq (0,\tau_0)\times (0,2)\subset \Omega$ and a well-defined smooth function
        \begin{equation*}
            s:(0,\tau_0)\to (0,2)
        \end{equation*}
        such that
        \begin{equation*}
            s'(\tau)=-\frac{\p Im\ \kappa_2(\tau,s)/\p \tau}{\p Im\ \kappa_2(\tau,s)/\p s},
        \end{equation*}
        and 
        \begin{equation*}
            \forall\tau\in (0,\tau_0),\quad Im\ \kappa_2(\tau,s(\tau))=0.
        \end{equation*}
    \end{cor}
    \begin{proof}
        Indeed, consider the smooth function $Im\ \kappa_2(\tau,s)$ defined on the open set $\Omega_{\tau_0}\subset \bR^2$. From Lemma \ref{criterio}, for any point $t^*\in (0,\tau_0)$ there exists a unique $s^*\in (0,2)$ such that $Im\ \kappa_2(t^*,s^*)=0$, and it holds that
        \begin{equation*}
            \eval{\frac{\p}{\p s}Im\ \kappa_2(\tau,s)}_{(t^*,s^*)}=\frac{1}{\pi}f'_{t^*}(s^*)<0.
        \end{equation*}
        The conclusion follows from the application of the Implicit Function Theorem \cite[pag. 161]{analiseelon}.
    \end{proof}
    The last part of the argument is to show that the real part of the period is not constant along the curve 
    \begin{equation*}
        C_s\coloneqq\text{graph(s)}=\{(\tau,s(\tau))|\tau\in(0,\tau_0)\}\subset \Omega_{\tau_0}.
    \end{equation*}
\begin{cor}\label{realperiodsimplified}
    Along the curve $C_s\subset \Omega_{\tau_0}$, the real part of the period can be written as
     \begin{equation*}
      \ Re\ \kappa_2(\tau,s(\tau))=-\frac{1}{2}s(\tau)- N(\tau,s(\tau)),
    \end{equation*}
    where $N(\tau,s)\in\bZ$
\end{cor}
\begin{proof}
    From Corollary \ref{realimaginaryperiods} we have that along $C_s$
   \begin{equation}\label{eqrealperiod}
        -i\pi\ Re\ \kappa_2(\tau,s(\tau))-i\pi N(\tau,s(\tau))=2i\tau\zeta(s(\tau)|\Tilde{L})-\zeta(2i\tau|\Tilde{L})s(\tau),
    \end{equation}    
    where $N(\tau,s(\tau))\in \bZ$. On the other hand, along the curve $C_s$, we have, by Corollary \ref{realimaginaryperiods} and Corollary \ref{curveCs}, that
    \begin{equation}\label{eq1}
        2\zeta(s(\tau)|\Tilde{L})=\zeta(2|\Tilde{L})s(\tau).
    \end{equation}
     Using the Legendre relation \cite[Section 20.411]{Whittaker}
    \begin{equation*}
        \zeta(\Tilde{\omega}_1|\Tilde{L})\Tilde{\omega}_2-\zeta(\Tilde{\omega}_2|\Tilde{L})\Tilde{\omega}_1=\frac{i\pi}{2},
    \end{equation*}
   and defining $\Tilde{\omega}'_3\coloneqq \Tilde{\omega}_2-\Tilde{\omega}_1=2i\tau$, we obtain, with respect to the lattice $\Tilde{L}$, that
    \begin{equation}\label{eq2}
        \zeta(2i\tau\Tilde{L})=-\frac{i\pi}{2}+ i\tau \zeta(2|\Tilde{L}).
    \end{equation}
    Substituting equations \eqref{eq1} and \eqref{eq2} into \eqref{eqrealperiod}, we obtain
    \begin{equation*}
         -i\pi\ Re\ \kappa_2(\tau,s(\tau))-i\pi N(\tau,s(\tau))=i\tau s(\tau)\zeta(2|\Tilde{L})-\left(-\frac{i\pi}{2}s(\tau)+i\tau s(\tau)\zeta(2|\Tilde{L})\right),
    \end{equation*}
    so that
    \begin{equation*}
         -i\pi\ Re\ \kappa_2(\tau,s(\tau))-i\pi N(\tau,s(\tau))=\frac{i\pi}{2}s(\tau),
    \end{equation*}
    which proves our claim.
\end{proof}

\begin{thm}\label{smonotone}
    The function $s:(0,\tau_0)\to (0,2)$ is monotone increasing. Moreover 
    \begin{equation*}
        \lim_{\tau\to \tau_0}s(\tau)=2.
    \end{equation*}
\end{thm}
\begin{proof}
 From Corollary \ref{curveCs} and its proof, it is enough to show that $$\eval{\frac{\p}{\p \tau} Im\ \kappa_2(\tau,s)}_{(\tau,s(\tau))}>0.$$
We start by introducing the theta function $\vartheta(z|q)$, which depends upon two complex parameters $z\in \bC$ and $q\in \bC$, with $|q|<1$, and defined by the uniformly convergent series
\begin{equation*}
    \vartheta_1(z|q)=2q^{\frac{1}{4}}\sum_{n=0}^{\infty}(-1)^nq^{n(n+1)}\sin(2n+1)z.
\end{equation*}
This function satisfies the following heat equations which can be readily checked from the series definition (see also \cite[Section 21.4]{Whittaker} and \cite{Wolfram})
\begin{equation}\label{heat1}
    4q\frac{\p }{\p q}\vartheta_1(z|q)+\frac{\p^2}{\p z^2}\vartheta_1(z|q)=0,
\end{equation}
\begin{equation}\label{heat2}
    4q\frac{\p }{\p q}\frac{\p \vartheta_1(z|q)}{\p z}+\frac{\p^2}{\p z^2}\frac{\p \vartheta_1(z|q)}{\p z}=0.
\end{equation}
From \cite[Eq. 23.6.13]{manual} we have the following formula
\begin{equation}\label{eq3}
    \zeta(z|\Tilde{L})=\frac{\zeta(\Tilde{\omega}_3|\Tilde{L})z}{\Tilde{\omega}_3}+\frac{\pi}{2\Tilde{\omega}_3}\frac{\frac{\p}{\p z}\vartheta_1(\nu|q)}{\vartheta_1(\nu| q)},
\end{equation}
where
\begin{equation*}
    \nu=\frac{\pi z}{2\Tilde{\omega}_3}
\end{equation*}
and
\begin{equation*}
    q=e^{\frac{i\pi \Tilde{\omega}_2}{\Tilde{\omega}_3}}=e^{\frac{i\pi (1+i\tau)}{2}}=ie^{-\frac{\pi \tau}{2}},
\end{equation*}
so that, by the chain rule,
\begin{equation*}
    \frac{\p}{\p z}=\frac{\pi}{4}\frac{\p}{\p \nu}\quad \frac{\p}{\p \tau}=-\frac{\pi q}{2}\frac{\p}{\p q}.
\end{equation*}
Notice that, from \eqref{eq3}, we can rewrite
\begin{equation}\label{eq4}
   2\zeta(z|\Tilde{L})-\zeta(2|\Tilde{L})z=\frac{\pi}{2}\frac{\frac{\p}{\p z}\vartheta_1(\nu|q)}{\vartheta_1(\nu| q)}.
\end{equation}
We start by calculating
\begin{equation*}
    \frac{\p}{\p q}\frac{\frac{\p}{\p z}\vartheta_1(\nu|q)}{\vartheta_1(\nu| q)}=\frac{\vartheta_1(\nu|q) \frac{\p}{\p q}\frac{\p}{\p z}\vartheta_1(\nu|q)-\frac{\p}{\p z}\vartheta_1(\nu|q)\frac{\p}{\p q}\vartheta_1(\nu|q)}{\vartheta_1(\nu|q)^2}.
\end{equation*}
Using the heat equations \eqref{heat1} and \eqref{heat2} this can be rewritten as
\begin{align*}
\begin{split}
    \frac{\p}{\p q}\frac{\frac{\p}{\p z}\vartheta_1(\nu|q)}{\vartheta_1(\nu| q)}&=\frac{-\vartheta_1(\nu|q) \frac{\p^3}{\p z^3}\vartheta_1(\nu|q)+\frac{\p}{\p z}\vartheta_1(\nu|q)\frac{\p^2}{\p z^2}\vartheta_1(\nu|q)}{4q\vartheta_1(\nu|q)^2}\\&=-\frac{1}{4q}\frac{\p}{\p z}\left(\frac{\frac{\p^2}{\p z^2}\vartheta_1(\nu|q)}{\vartheta_1(\nu|q)}\right).
    \end{split}
\end{align*}
This implies that
\begin{equation}\label{equationpartial}
     \frac{\p}{\p \tau}\frac{\frac{\p}{\p z}\vartheta_1(\nu|q)}{\vartheta_1(\nu| q)}=\frac{\pi }{8}\frac{\p}{\p z}\left(\frac{\frac{\p^2}{\p z^2}\vartheta_1(\nu|q)}{\vartheta_1(\nu|q)}\right).
\end{equation}
But then, taking derivatives in equation \eqref{eq4} with respect to $\tau$ and using \eqref{equationpartial}, we obtain
\begin{equation}\label{eq5}
    \frac{\p}{\p \tau}\left( 2\zeta(z|\Tilde{L})-\zeta(2|\Tilde{L})z\right)=\frac{\pi^2}{16}\frac{\p}{\p z}\left(\frac{\frac{\p^2}{\p z^2}\vartheta_1(\nu|q)}{\vartheta_1(\nu|q)}\right).
\end{equation}
Now consider the derivative with respect to the $z$ coordinate on equation \eqref{eq4},
\begin{align*}
\begin{split}
    -2\wp(z|\Tilde{L})-\zeta(2|\Tilde{L})&=\frac{\pi}{2}\frac{\vartheta_1(\nu|q)\frac{\p^2}{\p z^2}\vartheta_1(\nu|q)-\left(\frac{\p}{\p z}\vartheta_1(\nu|q)\right)^2}{\vartheta_1(\nu|q)^2}\\&=\frac{\pi}{2}\frac{\frac{\p^2}{\p z^2}\vartheta_1(\nu|q)}{\vartheta_1(\nu|q)}-\frac{\pi}{2}\left(\frac{\frac{\p}{\p z}\vartheta_1(\nu|q)}{\vartheta_1(\nu|q)}\right)^2.
    \end{split}
\end{align*}
Taking derivative once more in this last equation we obtain
\begin{equation}\label{eq6}
    -2\wp'(z|\Tilde{L})=\frac{\pi}{2}\frac{\p}{\p z}\left(\frac{\frac{\p^2}{\p z^2}\vartheta_1(\nu|q)}{\vartheta_1(\nu|q)}\right)-\pi \left(\frac{\frac{\p}{\p z}\vartheta_1(\nu|q)}{\vartheta_1(\nu|q)}\right)\frac{\p}{\p z}\left(\frac{\frac{\p}{\p z}\vartheta_1(\nu|q)}{\vartheta_1(\nu|q)}\right).
\end{equation}
Plugging equation \eqref{eq4} and \eqref{eq5} into \eqref{eq6},
\begin{align*}
\begin{split}
     -2\wp'(z|\Tilde{L})&=\frac{8}{\pi} \frac{\p}{\p \tau}\left( 2\zeta(z|\Tilde{L})-\zeta(2|\Tilde{L})z\right)\\&-\frac{4}{\pi}\left( 2\zeta(z|\Tilde{L})-\zeta(2|\Tilde{L})z\right)\frac{\p}{\p z}\left( 2\zeta(z|\Tilde{L})-\zeta(2|\Tilde{L})z\right).
     \end{split}
\end{align*}
However, notice that along the curve $C_s$ we have that
\begin{equation*}
    0=\pi Im\ \kappa_2(\tau,s(\tau))=2\zeta(s(\tau)|\Tilde{L})-\zeta(2|\Tilde{L})s(\tau),
\end{equation*}
and 
\begin{equation*}
    \pi \eval{\frac{\p}{\p \tau} Im\ \kappa_2(\tau,s)}_{(\tau,s(\tau))}=\frac{\p}{\p \tau}\left( 2\zeta(s|\Tilde{L})-\zeta(2|\Tilde{L})s\right).
\end{equation*}
Therefore
\begin{equation*}
 -2\wp'(s(\tau)|\Tilde{L})=8\eval{\frac{\p}{\p \tau} Im\ \kappa_2(\tau,s)}_{(\tau,s(\tau))},
\end{equation*}
so that
\begin{equation*}
    \eval{\frac{\p}{\p \tau} Im\ \kappa_2(\tau,s)}_{(\tau,s(\tau))} =-\frac{1}{4}\wp'(s(\tau)|\Tilde{L})>0,
\end{equation*}
by equation \eqref{derivativeP}, since along the curve we have that $0<s(\tau)<\Tilde{\omega}_3=2$. This concludes the proof that $s'(\tau)> 0$, which means that the function $s(\tau)$ is monotone increasing, so in particular it is not constant and therefore must assume rational values.

For the last part of the proof, it is enough to prove that for any sequence $\tau_n\in (0,\tau_0)$ with $\lim_{n\to +\infty}\tau_n= \tau_0$ we have that $\lim_{n\to +\infty}s(\tau_n)=2.$

Let us define
\begin{gather*}
    \liminf_{n\to +\infty}s(\tau_n)=s_{*}\in [0,2],\\ \limsup_{n\to +\infty}s(\tau_n)=s^{*}\in [0,2].
\end{gather*}
Since
\begin{equation*}
    Im\ \kappa_2(\tau_n,s(\tau_n))= 0,\quad \forall n\in \bN
\end{equation*}
and since the function $Im\ \kappa_2(\tau,s)=2\zeta(s|\Tilde{L}(\tau)-\zeta(2|\Tilde{L}(\tau))s$ is continuous in the two variables $(\tau,s)$ we must have
\begin{gather*}
    0=\liminf_{n\to +\infty}\ Im\ \kappa_2 (\tau_n,s(\tau_n))=Im\ \kappa_2(\tau_0,s_{*}),\\ 0=\limsup_{n\to +\infty}\ Im\ \kappa_2 (\tau_n,s(\tau_n))=Im\ \kappa_2(\tau_0,s^{*}),
\end{gather*}
which means that $s_{*},s^{*}\in [0,2]$ are zeros of the function $f_{\tau_0}(s)$ of Lemma \ref{criterio}. By Item II of the same Lemma, $s_{*}=s^{*}=2$, which implies that $\lim_{n\to +\infty}s(\tau_n)=2$ as we wanted to show.
\end{proof}
\begin{rem}
\normalfont{
    From Figure \ref{fig:sturmliouville} it seems that $s(\tau)$ assumes all values in $(1,2)$. However, there is a problem in studying the asymptotic behavior of the Weierstrass elliptic functions when $\tau\to 0$, because the two generators of the lattice are becoming linearly dependent. I am unaware whether or not this limit is among the classical degenerated cases studied in the literature, so it seems we can not say too much about it.}
\end{rem}
\begin{cor}\label{periodisrationalalongCs}
    There exists a non-empty countable set $\cJ\subset (0,\tau_0)$ such that  for all $\tau\in \cJ$ the period map $\kappa_2(\tau,s(\tau))$ is a real rational number.
\end{cor}
\begin{proof}
    Theorem \ref{smonotone} implies that the function $s:(0,\tau_0)\to (0,2)$ is a bijective function into its image, hence the subset
\begin{equation}\label{countableJ}
        \cJ\coloneq s^{-1}\left(\bQ\cap (0,2)\right)\subset (0,\tau_0)
    \end{equation}
    is non-empty and countable. For all $\tau\in \cJ$, we have $(\tau,s(\tau))\in C_s\subset \Omega_{\tau_0}$ with $s(\tau)\in \bQ$ and therefore $\kappa_2(\tau,s(\tau))\in \bQ$ by Corollary \ref{curveCs} and Corollary \ref{realperiodsimplified}.
\end{proof}

\begin{thm}\label{singlyperiodic}
    If $\tau\in \cJ$ then there exist integers $m(\tau),n(\tau)\in \bZ$, $\gcd(m(\tau),n(\tau))=1$ such that the period map is rational $$\kappa_2(\tau,s(\tau))=\frac{m(\tau)}{n(\tau)}\in \bQ.$$ Furthermore, for any $\tau\in \cJ$ and all $\hat{g}_0\in \bR$ there exists $\vec{a}_{\tau,\hat{g}_0}\in \bR^3\setminus\{0\}$ and a singly periodic minimal Wente torus with $n(\tau)^2$ embedded flat ends and finite total curvature $-4\pi n(\tau)^2$ of type $$(L(\tau,s(\tau)),b(\tau,s(\tau)),\hat{g}_0,\Lambda_\tau, \vec{a}_{\tau,\hat{g}_0}),$$ where $$\Lambda_\tau=\langle 2n(\tau)\omega_1(\tau,s(\tau)),2n(\tau)\omega_2(\tau,s(\tau))\rangle.$$
\end{thm}
\begin{proof}
   In fact, for any $(\tau,s,\hat{g}_0)\in C_s\times \bR^{+}\subset \Omega_{\tau_0}\times \bR^+\subset \Omega\times \bR^+$ there exists a complete minimal immersion $\Psi_{L(\tau,s(\tau)),\hat{g}_0}:\bC\setminus L(\tau,s)\to \bR^3$ with infinitely many embedded flat ends at the lattice points $L(\tau,s)$ and such that the $v$-lines $ \Psi_{L(\tau,s),\hat{g}_0}(u_0,\cdot)$ are spherical lines of curvature, by Corollary \ref{construWenteWeierstrassdata}. We have that for all $\tau\in \cJ$ the period map $\kappa_2(\tau,s(\tau))\in \bQ$ by Corollary \ref{periodisrationalalongCs}. Let $m(\tau),n(\tau)\in \bZ$, $\gcd(m(\tau),n(\tau))=1$ such that $$\kappa_2(b(\tau,s),L(\tau,s))=\kappa_2(\tau,s(\tau))=\frac{m(\tau)}{n(\tau)}\in \bQ.$$ By Lemma \ref{translationalperiod} there exists a translational vector $\vec{a}_{\tau,\hat{g}_0}$ such that equation \eqref{translationalequation} holds, which means that the map $\Psi_{L(\tau,s(\tau)),\hat{g}_0}$ factors in the quotient to a well-defined map 
   \begin{equation}
       \hat{\Psi}_{L(\tau,s(\tau)),b(\tau,s(\tau)),\hat{g}_0}:\left(\bC/\Lambda_\tau\right)\setminus \{P_1,\ldots P_{n(\tau)^2}\}\to \bR^3/\vec{a}_{\tau,\hat{g}_0},
   \end{equation}
   where $\Lambda_\tau=\langle 2n(\tau)\omega_1(\tau,s(\tau)),2n(\tau)\omega_2(\tau,s(\tau))\rangle$ and $P_j$ are the lattice points inside the fundamental parallelogram $\Lambda_\tau$, which are exactly $n(\tau)^2$.
   
   Moreover since the Gauss map has simple zeros at $\mu+2j_1\omega_1+2j_2\omega_2$ $j_1,j_2=1,\ldots n$ by Theorem \ref{gaussmapformula} and equation \eqref{equivalentdefinitions}, we see that the degree of the extension of the Gauss map defined on the torus
    \begin{equation*}
        g:\bC/\Lambda_\tau\to \bS^2,
    \end{equation*}
    is $n(\tau)^2$ and therefore the total curvature of $\hat{\Psi}_{L(\tau,s(\tau)),b(\tau,s(\tau)),\hat{g}_0}$ is $-4\pi n(\tau)^2.$ In order to prove that $\vec{a}_{\tau,\hat{
    g}_0}\neq 0$ we show that $\vec{a}_3\neq 0$ along $C_s\times \bR^{+}$. In fact the third component of the translational period is
   \begin{align*}
   \begin{split}
       \vec{a}_3&=Re\int_{z_0}^{z_0+2n\omega_2}\phi(\xi)d\xi=Re\int_{z_0}^{z_0+2n\omega_2}(b-4\wp(\xi))d\xi\\&=Re\int_{z_0}^{z_0+2n\omega_2}(b+4\zeta'(\xi))d\xi=Re\left(b(2n\omega_2)+4(\zeta(z_0+2n\omega_2)-\zeta(z_0)\right)\\&=2nb\lambda+8nRe\ \zeta(\omega_2)=8n\lambda\wp(\mu)+8n\left(\frac{\zeta(\omega_2)+\zeta(\omega_1)}{2}\right)\\&=8n\lambda\wp(\lambda s|L)+4n\zeta(2\lambda|L)=\frac{4n}{\lambda}\left(2\wp(s|\Tilde{L})+\zeta(2|\Tilde{L)}\right).
       \end{split}
   \end{align*}
   Therefore
   \begin{equation*}
       \vec{a}_3=-\frac{4n}{\lambda}\eval{f_{\tau}'(s)}_{s=s(\tau)}\neq 0,
   \end{equation*}
   according with Lemma \ref{criterio}.
\end{proof}
\begin{rem}
\normalfont{
    From Theorem \ref{smonotone} we see that as $\tau\to \tau_0$ then $s\to 2$ and $f'_{\tau}(s(\tau))\to 0$ so that $\vec{a}_3\to 0$, however the conformal factor collapses to zero $\lambda=\sqrt[3]{-4\wp'(s|\Tilde{L}})\to 0$.}
\end{rem}
Now we explain that the construction we have described is canonical in the following sense
\begin{thm}\label{canonicalconstruction}
    Let $\hat{\Psi}_{L^*,b^*,\hat{g}_0^*}$ denote either a minimal Wente torus with ends of type $(L^*,b^*,\hat{g}_0^*,\Lambda^*)$ or a singly periodic minimal Wente torus with ends of type $(L^*,b^*,\hat{g}_0^*,\Lambda^*,\vec{a}^*)$, and denote by $$\Psi_{L^*,b^*,\hat{g}_0^*}:\bC\setminus L^*\to \bR^3$$ the corresponding lifting complete minimal immersion of Enneper type. Then there exists a point in the parameter space $(\tau,s,\hat{g}_0^*)\in C_s\times \bR^+\subset  \Omega\times \bR^+$ such that 
    \begin{equation*}
\Psi_{L^*,b^*,\hat{g}_0^*}=\Psi_{L(\tau,s),\hat{g}_0^*}
    \end{equation*}
\end{thm}
\begin{proof}
   By Definition \ref{minimalwentetorus}, Theorem \ref{importancerombic} and Remark \ref{singlyperiodicperiod}, the lattice $L^*$ is rhombic and the period map $\kappa_2(b^*,L^*)$ is a real rational number. Let us choose $2\omega_1^*$ and $2\omega_2^*$ the generators of $L^*$ such that $\omega_2^*=\overline{\omega_1^*}$ and $Re\left(\omega_1^*\right)>0$, $Im\left(\frac{\omega_2^*}{\omega_1^*}\right)>0$. Let us define $\omega_3^*=\omega_1^*+\omega_2^*\in \bR^+\setminus \{0\}$ and $\omega_3'^*=\omega_2^*-\omega_1^*\in i\bR$. Consider the constants
   \begin{equation*}
       \tau\coloneq \frac{\abs{\omega_3'^*}}{\omega_3^*}\in \bR^+,\quad \lambda^*\coloneq \frac{\omega_3^*}{2}\in \bR^+ 
   \end{equation*}
   Then we have that $L^{*}=\lambda^*\cdot \tilde{L}(\tau)$ which implies that
   \begin{equation}\label{rescalings}
       2\omega_1^*=2\lambda^*-2\lambda^*i\tau,\quad 2\omega_2^*=2\lambda^*+2\lambda^*i\tau
   \end{equation}
   We know that the Weierstrass function $\wp(z|L^*)$ is only real along the diagonals of the rhombus and it is a diffeomorphism mapping the straight line segments $[0,\omega_3^*]\cup [\omega_3^*,2\omega_2^*]$ into the real line. By hypothesis we know that $b^*\in \bR$, so there exists $\mu^*\in [0,\omega_3^*]\cup [\omega_3^*,2\omega_2^*]$ such that $b^*=4\wp(\mu^*|L^*).$ Moreover, since by hypothesis $b^*s$ satisfies the equation \eqref{cubicequationb}, using the differential equation satisfied by $\wp$ we show that $\wp'(\mu^*|L^*)\in \bR\setminus \{0\}$, which implies that $\mu^*\in (0,\omega_3^*)\subset \bR.$ Since the period map $\kappa_2(b^*,L^*)$ is a real rational number, we have according to equation \eqref{Period} that
   \begin{equation*}\label{muestrela}
       Re\left[\zeta(\mu^*|L^*)\omega_2^*-\zeta(\omega_2^*|L^*)\mu^*\right]=0.
   \end{equation*}
Substituting equation \eqref{rescalings} into equation \eqref{muestrela} and taking into account that $\mu^*\in \bR$ we obtain that
\begin{equation*}
    \zeta(\mu^*|L^*)(2\lambda^*)-\zeta(2\lambda^*|L^*)\mu^*=0
\end{equation*}
Therefore, using the homogeneity property of the $\zeta$ function we obtain the following equation
\begin{equation*}
    2\zeta\left(\frac{\mu^*}{\lambda^*}\Bigm|\tilde{L}(\tau)\right)-\zeta\left(2\Bigm| \tilde{L}(\tau)\right)\frac{\mu^*}{\lambda^*}=0,\quad \frac{\mu^*}{\lambda^*}\in (0,2)
\end{equation*}
Therefore by Lemma \ref{criterio} we must have that $\tau\in (0,\tau_0)$ and moreover by uniqueness $\mu^*=\lambda^*s(\tau)$. Let us define the dilation factor
\begin{equation*}
    \gamma=\frac{\lambda^*}{\lambda(\tau,s(\tau))}>0.
\end{equation*}
    Then we have that
    \begin{equation*}
        L^*=\lambda^*\tilde{L}(\tau)=\gamma L(\tau,s(\tau)),\quad \mu^*=\lambda^*s(\tau)=\gamma \mu(\tau,s(\tau)).
    \end{equation*}
    Using the homogeneity relations \cite[Eq. 18.2.2-18.2.5-18.2.6]{Handbook} we then obtain
    \begin{gather*}
        g_2^*(L^*)=\frac{g_2(L(\tau,s(\tau))}{\gamma^4},\hspace{0.1cm} g_3^*(L^*)=\frac{g_3(L(\tau,s(\tau))}{\gamma^6},\hspace{0.1cm} b^*=4\wp(\mu^*|L^*)=\frac{4\wp(\mu(\tau,s(\tau))|L(\tau,s(\tau)))}{\gamma^2}=\frac{b(\tau,s(\tau))}{\gamma^2}
    \end{gather*}
    On the other hand, we know that $b$ and $b^*$ satisfy the cubic equations
    \begin{gather*}
        (b^*)^3-4g_2^*(L^*)b^*-16g_3^*(L^*)=1,\quad b^3(\tau,s(\tau))-4g_2(L(\tau,s(\tau))b(\tau,s(\tau))-16g_3(L(\tau,s(\tau)))=1
    \end{gather*}
    Therefore we must have that $\gamma^6=1$ and since $\gamma\in \bR^+$ it must be that $\gamma=1$ implying that
    \begin{equation*}
        L^*=L(\tau,s(\tau)),\quad b^*=b(\tau,s(\tau)),
    \end{equation*}
    which is enough to conclude our proof.
\end{proof}
\begin{thm}\label{nonexistence}
   There are no minimal Wente torus with ends in the sense of Definition \ref{minimalwentetorus}. 
\end{thm}
\begin{proof}
    Suppose by contradiction that there exists $\hat{\Psi}_{L,b,\hat{g}_0}$ a minimal Wente torus with ends of type $(L,b,\hat{g}_0,\Lambda)$ and consider its associated complete minimal immersion of Enneper type $\Psi_{L,b,\hat{g}_0}:\bC\setminus L\to \bR^3$. By Theorem \ref{importancerombic} the period map $\kappa_2(b,L)\in \bQ$ and by the canonical construction Theorem \ref{canonicalconstruction} there exists $\tau\in (0,\tau_0)$ such that 
    \begin{equation*}
\Psi_{L,b,\hat{g}_0}=\Psi_{L(\tau,s(\tau)),\hat{g}_0},\quad L=L(\tau,s(\tau))=\text{span}\{2\omega_1(\tau,s(\tau)),2\omega_2(\tau,s(\tau))\}
    \end{equation*}
    This implies that $\kappa_2(\tau,s(\tau))=\frac{m(\tau)}{n(\tau)}$ where $m(\tau),n(\tau)\in \bZ$ and therefore, by the same argument as in the proof of Theorem \ref{singlyperiodic}, there exists a non-zero vector $\vec{a}_{\tau,\hat{g}_0}$ such that equation \eqref{translationalequation} holds, in particular
    \begin{equation}\label{particularequation}
        \forall z\in \bC\setminus L\quad \Psi_{L,b,\hat{g_0}}(z+2n(\tau)\omega_1(\tau,s(\tau)))=\Psi_{L,b,\hat{g_0}}(z)+\vec{a}_{\tau,\hat{g}_0}
    \end{equation}
    Now consider a countable collection of disks $D^j_{\epsilon}\subset \bC$ of radius $\epsilon$ sufficiently small centered at the vertex points of the lattice $L$. Let $K$ be the compact set defined as the intersection of the fundamental parallelogram defined by the sub-lattice $\Lambda$ with the complement of the collection of the disks. Then the subset $\Psi_{L,b,\hat{g}_0}(K)\subset \bR^3$ is also compact and therefore it is contained in a ball $B_R$ centered at the origin and of sufficiently large radius $R$. Fix a point $P\in K$ and consider the sequence of points
    \begin{equation*}
        P_j\coloneq P+2jn(\tau)\omega_1(\tau,s(\tau)),\quad j=1,2,\ldots, 
    \end{equation*}
    By applying appropriate translations using the generators of the sub-lattice $\Lambda$ we can produce a sequence of points $\{Q_j\}_{j\in \bN}$ contained in the compact set $K$ and such that
    \begin{equation}\label{pq}
        \forall j\in \bN\quad \pi(Q_j)=\pi(P_j)
    \end{equation}
    where $\pi:\bC\setminus L\to (\bC\setminus L)/\Lambda$ is the canonical projection map. On one hand we have by hypothesis and equation \eqref{pq} that
    \begin{equation*}
        \forall j\in \bN\quad \Psi_{L,b,\hat{g}_0}(P_j)=\hat{\Psi}_{L,b,\hat{g}_0}\circ \pi (P_j)=\hat{\Psi}_{L,b,\hat{g}_0}\circ \pi (Q_j)=\Psi_{L,b,\hat{g}_0}(Q_j)\subset B_R.
    \end{equation*}
    However, on the other hand, since $\vec{a}_{\tau,\hat{g}_0}\neq 0$, we have by equation \eqref{particularequation} that 
    \begin{equation*}
        \forall j\quad \Psi_{L,b,\hat{g}_0}(P_j)=\Psi_{L,b,\hat{g}_0}(P)+j\vec{a}_{\tau,\hat{g}_0},
    \end{equation*}
    meaning that the sequence $\{\Psi_{L,b,\hat{g}_0}(P_j)\}_{j\in \bN}$ eventually escapes from the ball $B_R$, which is a contradiction. Therefore there are no minimal Wente torus with ends of any type in the sense of Definition \ref{minimalwentetorus}.
\end{proof}
\section{Symmetry analysis}\label{Section 5.4}
In the subsequent work, we assume we are working  with a choice of the parameter $\tau$ such that along $C_s$, $\kappa_2=\frac{m}{n}$. There will be an important angle which we denote by \begin{equation*}
    \theta_{m,n}=\frac{2\pi m}{n}.
\end{equation*}
We split the analysis of the symmetries of the immersion $\Psi_{L,\hat{g}_0}$ into the general case $\hat{g}_0\in \bR^+$ and the special case $\hat{g}_0=1$. We also set the base point $z_0=\frac{\omega_3}{2}$ so that $\Psi(z_0)=0.$ According to Wente's construction, see \cite[Proof of Theorem 4.1]{Wentespherical}, the center of the spheres containing the spherical curvature lines belong to a fixed line, which we will denote by $l$ and which is parallel to the $x_3$-axis.

In the following analysis we make use of the transformation
    \begin{equation*}
        R_{x_3=0}=\begin{pmatrix}
            1&0&0\\0&1&0\\ 0&0&-1
        \end{pmatrix},
    \end{equation*}
    which is a reflection with respect to the plane $x_3=0.$ We also make use of the following lemma.
\begin{lema}
    Let $\alpha,\theta \in \bR$ and define the vector in the $x_1x_2$-plane 
    \begin{equation*}
        \vec{v}_{\theta}\coloneqq \left(\cos \left(\frac{\pi}{2}-\theta\right),\sin \left(\frac{\pi}{2}-\theta\right),0\right)\in \bR^3.
    \end{equation*}
    Then the matrix
  \begin{gather*}
       M_{{\alpha}} \coloneqq   \begin{pmatrix}
            \cos \alpha& -\sin \alpha & 0\\ -\sin \alpha & -\cos \alpha & 0\\ 0& 0& 1
        \end{pmatrix},
    \end{gather*}
    is a reflection with respect to a plane passing through the origin and perpendicular to $\vec{v}_{\frac{\alpha}{2}}$.
\end{lema}
\subsection{Symmetry for general $\hat{g}_0$}
\begin{thm}\label{symmetriesg0neq1}
    The minimal immersion $\Psi_{L,\hat{g}_0}$ is invariant under reflections with respect to planes $\Pi^1_{k,\theta_{m,n}}$ perpendicular to the vectors $\vec{v}_{k\theta_{m,n}}$ and containing the center axis $l$. Moreover the curvature lines $\Psi_{L,\hat{g}_0}(\cdot, k\abs{\omega_3'})$ lie in the correspondent $\Pi^1_{k,\theta_{m,n}}$ planes. The immersion $\Psi_{L,\hat{g}_0}$ has dihedral symmetry group $\cD_{n}$ for $n$ odd and $\cD_{\frac{n}{2}}$ for $n$ even.
\end{thm}
\begin{proof}
     Consider the involution in the complex plane
     \begin{equation*}
         T_k(z)=\overline{z-2k\omega_3'},
     \end{equation*}
     which is a reflection with respect to the horizontal line $y=k\abs{\omega_3'}$.
      We have by the properties of the Weierstrass $\wp$ function that
     \begin{gather*}
     \phi(T_k(z))=\overline{\phi(z)},
     \end{gather*}
     \begin{gather*}
         g(T_k(z))=\hat{g}_0\exp\left(\int_0^{T_k(z)}\frac{d\xi}{\phi(\xi)}\right)=\hat{g}_0\exp\left(\int_0^{1}\frac{\beta'(t)dt}{\phi(\beta(t))}\right),
     \end{gather*}
     where $\beta$ is a path such that $\beta(0)=0$ and $\beta(1)=T_k(z)=\overline{z-2k\omega_3'}$. Consider the path $\delta(t)=\overline{\beta(t)}+2k\omega_3'$ with $\delta(0)=2k\omega_3'$ and $\delta(1)=z$.
     
     \begin{align*}
     \begin{split}
         g(T_k(z))&=\hat{g}_0\exp\left(\int_0^{1}\frac{\overline{\delta'(t)}dt}{\phi(\overline{\delta(t)}-\overline{2k\omega_3'})}\right)=\hat{g}_0\exp\left(\int_0^{1}\frac{\overline{\delta'(t)}dt}{\phi(\overline{\delta(t)})}\right)\\&=\hat{g}_0\overline{\exp}\left(\int_{2k\omega_3'}^{z}\frac{d\xi}{\phi(\xi)}\right)=\overline{g(z)}\overline{\exp}\left(-\int_0^{2k\omega_3'}\frac{d\xi}{\phi(\xi)}\right).
         \end{split}
     \end{align*}
     Notice that
     \begin{gather*}
         \int_0^{2k\omega_3'}\frac{d\xi}{\phi(\xi)}=\int_0^{2k\omega_2-2k\omega_1}\frac{d\xi}{\phi(\xi)}=\int_0^{2k\omega_2}\frac{d\xi}{\phi(\xi)}+\int_{2k\omega_2}^{2k\omega_2-2k\omega_1}\frac{d\xi}{\phi(\xi)}.
     \end{gather*}
     By a change of variables and using the $2\omega_1,2\omega_2$- periodicity of $\phi$ and its parity we have
     \begin{align}\label{integral1sobrephi}
     \begin{split}
         \int_0^{2k\omega_3'}\frac{d\xi}{\phi(\xi)}&=k\int_0^{2\omega_2}\frac{d\xi}{\phi(\xi)}-k\int_0^{2\omega_1}\frac{d\xi}{\phi(\xi)}\\&=-2\pi i k(\kappa_2-\kappa_1)=-4\pi i k \frac{m}{n}=-2ik\theta_{m,n}.
          \end{split}
     \end{align}
     Therefore
     \begin{equation*}
         g(T_k(z))=\overline{g(z)}e^{-2ik\theta_{m,n}}.
     \end{equation*}
     Then
     \begin{align*}
     \begin{split}
         \Phi(T_k(z))&=\left(\frac{1}{2}\left(\frac{1}{g(T_k(z))}-g(T_k(z))\right),\frac{i}{2}\left(\frac{1}{g(T_k(z))}+g(T_k(z))\right),1\right)\phi(T_k(z))\\&=\left(\frac{1}{2}\left(\frac{e^{2ik\theta_{m,n}}}{\overline{g(z)}}-\overline{g(z)}e^{-2ik\theta{m,n}}\right),\frac{i}{2}\left(\frac{e^{2ik\theta_{m,n}}}{\overline{g(z)}}+\overline{g(z)}e^{-2ik\theta{m,n}}\right),1\right)\overline{\phi(z)}.
         \end{split}
     \end{align*}
     After some calculations
     \begin{gather*}
         T^{*}\left(\Phi(z)dz\right)=\overline{M_{{2k\theta_{m,n}}}\left(\Phi(z)dz\right)},
     \end{gather*}
     and therefore we have
     \begin{gather}\label{symmetryPsig0neq1}
         \Psi(T_k(z))=M_{2k\theta_{m,n}}\Psi(z)+\Psi(T_k(z_0)).
     \end{gather}
     Notice that since $T_k$ is an involution we have
     \begin{equation*}
     \Psi(z)=\Psi(T_k(T_k(z)))=M_{2k\theta_{m,n}}\Psi(T_k(z))+\Psi(T_k(z_0)),
     \end{equation*}
     so plugging in $z=z_0$ we see that
     \begin{equation*}
         M_{2k\theta_{m,n}}\Psi(T_k(z_0))=-\Psi(T_k(z_0)).
     \end{equation*}
     That is $\Psi(T_k(z_0))$ is an eigenvector of $M_{2k\theta_{m,n}}$ with eigenvalue $-1$. This implies that 
     \begin{equation*}
         \Psi(T_k(z_0))\quad|| \quad \vec{v}_{k\theta_{m,n}} \in x_1x_2\text{-plane}.
     \end{equation*}
     Since $M_{2k\theta_{m,n}}$ is a reflection matrix with respect to a plane passing through the origin and perpendicular to $\vec{v}_{k\theta_{m,n}}$ we see that equation \eqref{symmetryPsig0neq1} says that $\Psi(T_k(z))$ is the reflection of the point $\Psi(z)$ with respect to a plane $\Pi^1_{k,\theta_{m,n}}$ passing through $\frac{1}{2}\Psi(T_k(z_0))$ and perpendicular to the vector $\vec{v}_{k\theta_{m,n}}$.
     
     Notice that the transformation $T_k$ leaves invariant the spherical lines of curvature i.e for fixed $u$ we have $\Psi(T_k(u+i\bR))\subset \Psi(u+i\bR)$, therefore we must have that the plane $\Pi^1_{k,\theta_{m,n}}$ contains the axis $l$ of the centers.
     
     In conclusion we have that for all $k$, $\Psi(T_k(z))$ is a reflection with respect to a plane passing through the line $l$ and perpendicular to the vector $\vec{v}_{k\theta_{m,n}}$.
\end{proof}
\subsection{Additional symmetries for $\hat{g}_0=1$}
In the special case when $\hat{g}_0=1$, the minimal immersions $\Psi_{L,\hat{g}_0}$ still have the reflection symmetries of Theorem \ref{symmetriesg0neq1} with respect to planes containing the center line $l$, and moreover, they will have an extra reflection symmetry about planes perpendicular to $l$.
    \begin{thm}\label{symmetriesg0=1}
        Suppose $\hat{g}_0=1$. Then the minimal immersion $\Psi_{L,\hat{g}_0}$ is invariant under reflections with respect to a family of equally spaced planes $\Pi^2_k$ containing the spherical curvature lines $\Psi_{L,\hat{g}_0}(k\omega_3,\cdot)$ and perpendicular to the center axis $l$. Furthermore there exist a translational vector parallel to $l$ 
        \begin{equation*}
            \vec{v}\coloneqq (0,0,2b\omega_3+8\zeta(\omega_3)),
        \end{equation*}
        such that
        \begin{gather*}
            \Psi_{L,\hat{g}_0}(z+2\omega_3)=\Psi_{L,\hat{g}_0}(z)+\vec{v}.
        \end{gather*}
    \end{thm}
We omit the proof of Theorem \ref{symmetriesg0=1} and refer the reader to \cite{CarlosThesis}[{Theorem 5.4.3}].
    \section{Hamiltonian System}\label{Section 5.5}
     Consider the immersion $\Psi_{L(\tau,s),\hat{g}_0}$ with the parameters $(\tau,s,\hat{g}_0)\in \Omega \times \bR^+$. By the work of Wente, see \cite[equations 2.18 and 4.3]{Wentespherical}, there exists a constant $\delta\in \bR$ and two real functions $\alpha,\beta:\bR\to \bR$ depending only on the $u$-variable such that
    \begin{equation}\label{Enneperequation}
        2\omega_u(u,v)=\alpha(u)e^{\omega(u,v)}+\beta(u)e^{-\omega(u,v)},
    \end{equation}
    which are solutions to the Hamiltonian system 
    \begin{gather}\label{hamiltonianequations}
        \begin{cases}
            \alpha''=\delta \alpha-2\alpha^2\beta,\\ \beta''=\delta \beta-2\alpha \beta^2 -2\alpha,
        \end{cases}
    \end{gather}
    with prescribed initial conditions $\alpha(0),\beta(0),\alpha'(0),\beta'(0)$ depending on the parameter space.
    
    The radius $R(u)$ of the spheres $S(u)$ containing the lines of curvature $\Psi(u,\cdot)$ and the angle of intersection $\theta(u)$ between the surface and $S(u)$ are
\begin{equation}\label{radiusangle}
    R^2(u)=\frac{4+\beta(u)^2}{\alpha^2(u)},\quad \tan\theta(u)=\frac{2}{\beta(u)}.
\end{equation}
Moreover, the centers $c(u)$ of the spheres $S(u)$ are \cite[Proof of Theorem 4.1]{Wentespherical}:
\begin{equation}\label{centerequationgeneral}
    c(u)=\Psi(u,v)-\frac{2}{\alpha(u)}\frac{\Psi_u(u,v)}{\abs{\Psi_u(u,v)}}-\frac{\beta(u)}{\alpha(u)}N_{\Sigma}(\Psi(u,v)).
\end{equation}
The Hamiltonian system has two conserved quantities
    \begin{gather*}
            h=\alpha'\beta'+\alpha^2-\delta\alpha\beta+\alpha^2\beta^2,\quad 4k=(\alpha\beta'-\alpha'\beta)^2+4\alpha'^2+4\alpha^3\beta-4\delta\alpha^2.
    \end{gather*}
    Furthermore defining the polynomial
    \begin{equation}\label{pdefinition}
        p(u,X)=-\alpha^2X^4-4\alpha'X^3+(6\alpha\beta-4\delta)X^2+4\beta'X-(4+\beta^2),
    \end{equation}
    the conformal factor satisfies the equation
    \begin{equation}\label{polynomialpequation}
        4((e^{\omega})_v)^2=p(u,e^{\omega}).
    \end{equation}
    Notice that the function $\phi(z)=b-4\wp(z)$ satisfies the equation
    \begin{align*}
    \begin{split}
        \phi'(z)^2&=16\wp'(z)^2=64(\wp(z)-e_1)(\wp(z)-e_2)(\wp(z)-e_3)\\&=64\prod_{j=1}^3\frac{b-4\wp(z)+4e_j-b}{-4}=-\prod_{j=1}^3(\phi(z)-(b-4e_j)).
        \end{split}
    \end{align*}
    As in \cite[equation 5.1]{Isabel} we define
    \begin{equation}\label{rootproperties}
        r_j=b-4e_j,\quad r_1=\overline{r_2}\in \bC\setminus \bR,\ r_3\in \bR,
    \end{equation}
    and the polynomial \begin{equation}\label{polynomialq}
        \hat{q}(x)=-(x-r_1)(x-r_2)(x-r_3).
    \end{equation}
    Then
    \begin{equation}\label{phiequation}
        \phi'(z)^2=\hat{q}(\phi(z)).
    \end{equation}
    We have the well known identities 
    \begin{gather*}
            e_1+e_2+e_3=0,\quad e_1e_2+e_1e_3+e_2e_3=-\frac{g_2}{4},\quad  e_1e_2e_3=\frac{g_3}{4},
    \end{gather*}
    which implies that
    \begin{equation*}
        \sum r_j=3b,
    \end{equation*}
   and since $b$ satisfies equation \eqref{cubicequation2} we also obtain
    \begin{align*}
    \begin{split}
        r_1r_2r_3&=b^3-4(e_1+e_2+e_3)b^2+16(e_1e_2+e_1e_3+e_2e_3)b-64e_1e_2e_3\\&=b^3-4g_2b-16g_3.
        \end{split}
    \end{align*}
    that is
    \begin{equation*}
        r_1r_2r_3=1.
    \end{equation*}
   In particular we see that $r_3$ is positive,
    \begin{equation*}
        r_3=\frac{1}{r_1r_2}=\frac{1}{r_1\overline{r_1}}=\frac{1}{|r_1|^2}\in \bR^{+}.
    \end{equation*}
   The purpose now is to determine the constants $\delta,h,k$ and the dependence of the initial conditions of the Hamiltonian system in terms of the parameter space.
   \begin{prop}\label{initialconditionsparameterspace}
       Consider the minimal immersion $\Psi_{L(\tau,s),\hat{g}_0}$ with parameters $(\tau,s,\hat{g}_0)\in C_s\times \bR^+$ and let $C\coloneqq \frac{\hat{g_0}+\hat{g_0}^{-1}}{2}>0$. Then the initial conditions of the associated Hamiltonian system and its constants of motion satisfy
        \begin{gather*}
            \alpha(0)=0,\quad  \alpha'(0)=-\frac{1}{C},\quad \beta'(0)=C\sum_{i<j}r_ir_j,\quad \beta(0)^2=4(C^2-1),
    \end{gather*}
    \begin{gather*}
         \delta=3b,\quad    h=-(r_1r_2+r_1r_3+r_2r_3),\quad k=1.
    \end{gather*}
    Moreover, at $u=\omega_3$ the Hamiltonian system satisfy
    \begin{gather*}
        \alpha(\omega_3)=0,\quad  \alpha'(\omega_3)=-\frac{1}{C},\quad \beta'(\omega_3)=C\sum_{i<j}r_ir_j,\quad \beta^2(\omega_3)=4(C^2-1).
\end{gather*}
   \end{prop}
  \begin{proof}
We know that the conformal factor is
    \begin{equation}\label{conformalfactorimersion}
        e^{\omega}=\frac{|\phi|}{2}(|g|+|g|^{-1}).
    \end{equation}
    Along $(0,iv)$ we have that $\phi=4\wp(\mu)-4\wp(iv)>0$ and therefore 
    \begin{equation*}
        |g(iv)|=\hat{g}_0,
    \end{equation*}
    \begin{equation*}
        e^{\omega(0,v)}=C(\hat{g_0})\phi(iv),\quad C=C(\hat{g_0})=\frac{\hat{g_0}+\hat{g_0}^{-1}}{2}.
    \end{equation*}
    Notice that since $\hat{g_0}\in \bR^{+}$ then $C\geq 1$. Then by equation \eqref{polynomialpequation} 
    \begin{equation*}
       ( e^{\omega(0,v)})_v=iC\phi'(iv),
    \end{equation*}
    \begin{equation*}
        \frac{p(0,e^{\omega(0,v)})}{4}=(( e^{\omega(0,v)})_v)^2=-C^2\phi'(iv)^2.
    \end{equation*}
    Therefore using equations \eqref{pdefinition} and \eqref{phiequation} we obtain
    \begin{align*}
    \begin{split}
        -\frac{1}{4}\alpha(0)^2C^4\phi(iv)^4-&\alpha'(0)C^3\phi(iv)^3+\frac{6\alpha(0)\beta(0)-4\delta}{4}C^2\phi(iv)^2+\beta'(0)C\phi(iv)-\frac{4+\beta(0)^2}{4}\\&=-C^2\hat{q}(\phi(iv))=C^2(\phi(iv)-r_1)(\phi(iv)-r_2)(\phi(iv)-r_3)\\& =C^2\phi(iv)^3-C^2(r_1+r_2+r_3)\phi(iv)^2+C^2(r_1r_2+r_1r_3+r_2r_3)\phi(iv)-C^2r_1r_2r_3\\&=C^2\phi(iv)^3-3C^2b\phi(iv)^2+C^2\sum_{i<j} r_ir_j \phi(iv)-C^2.
        \end{split}
    \end{align*}
    We conclude that
    \begin{gather*}
            \alpha(0)=0,\quad  \alpha'(0)=-\frac{1}{C},\quad \delta=3b,\quad \beta'(0)=C\sum_{i<j}r_ir_j,\quad \beta(0)^2=4(C^2-1).
    \end{gather*}
     Therefore the constants of motion are
        \begin{gather*}
                h=-(r_1r_2+r_1r_3+r_2r_3),\quad k=1.
        \end{gather*}
Notice that $h,k,\delta$ are independent of $\hat{g}_0$, which implies that the Hamiltonian system equations \eqref{hamiltonianequations} are independent of $\hat{g}_0$. However we see that the \textit{initial conditions $\alpha'(0),\beta(0),\beta'(0)$ depend on the choice of $\hat{g}_0$}.

We want to study the value of $\alpha,\alpha',\beta, \beta'$ at $\omega_3$. Notice that
\begin{gather*}
    g(\omega_3)=\hat{g}_0\exp\left(2\zeta(\mu)\omega_3\right)\frac{\sigma(\mu-\omega_3)}{\sigma(\mu+\omega_3)}.
\end{gather*}
Using the identity \eqref{sigmaidentity} we have
\begin{equation*}
    \sigma(\mu+\omega_3)=-\exp\left(2\zeta(\omega_3)\mu\right)\sigma(\mu-\omega_3),
\end{equation*}
and along the curve $C_s$, using the rescaling properties and equation \eqref{eq1} we have $\zeta(\mu)\omega_3=\zeta(\omega_3)\mu$ and therefore
\begin{equation}\label{g(omega3)}
    g(\omega_3)=-\hat{g}_0\in \bR^{-}.
\end{equation}
Since $$\phi(\omega_3+iv)>0,$$ and $$g(\omega_3+iv)=g(\omega_3)\exp\left(\int_{\omega_3}^{\omega_3+iv}\frac{1}{\phi(\nu)}d\nu\right),$$ we have
\begin{equation*}
    \abs{g(\omega_3+iv)}=\hat{g}_0,
\end{equation*}
so that
\begin{equation*}
    e^{\omega(\omega_3,iv)}=C\phi(\omega_3,iv).
\end{equation*}We can prove in an analogous way that
\begin{gather*}
        \alpha(\omega_3)=0,\quad  \alpha'(\omega_3)=-\frac{1}{C},\quad \beta'(\omega_3)=C\sum_{i<j}r_ir_j,\quad \beta^2(\omega_3)=4(C^2-1).
\end{gather*}
 \end{proof}
 Even though $\beta^2(0)=\beta^2(\omega_3)$, we want to determine the sign relation between them. The following analysis studies that question.
 
 We begin with a simple lemma which comes from the properties of the Weiertrass $\wp$ function
\begin{figure}[h]
    \centering
    \includegraphics[width=0.7\linewidth]{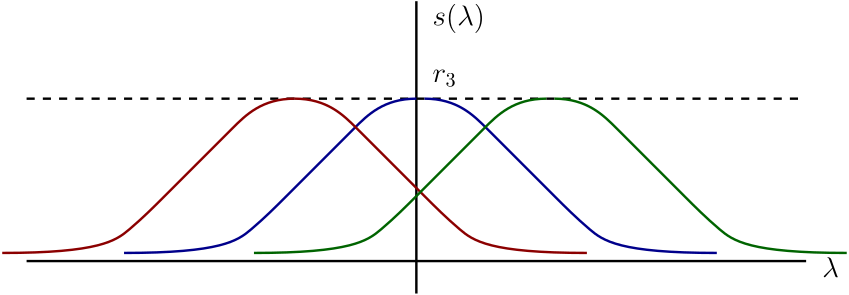}
    \caption{The $s(\lambda)$ orbits for the system $s'(\lambda)^2=s^2(\lambda)\hat{q}(s(\lambda))$.}
    \label{fig:s orbit}
\end{figure}
\begin{lema}\label{signphi}
    \begin{gather*}
    \forall |u|<\omega_3,\quad \phi(u)=4(\wp(\mu)-\wp(u))\in \begin{cases}
        \bR^{-}\quad \text{if}\quad  |u|< \mu,\\\bR^{+}\quad \text{if}\quad \mu <|u|< \omega_3.
    \end{cases}
\end{gather*} 
\end{lema}
We also have a similar property of the Gauss map
\begin{lema}\label{signg}
\begin{gather*}
    \forall |u|<\omega_3,\quad g(u)\in \begin{cases}
        \bR^{+}\quad \text{if}\quad  |u|< \mu,\\\bR^{-}\quad \text{if}\quad \mu< |u|< \omega_3.
    \end{cases}
\end{gather*}
\end{lema}
\begin{proof}
In fact since $\frac{1}{\phi(\nu)}$ is free of poles on the interval $(-\mu,\mu)$ we have
\begin{equation*}
    \forall |u|<\mu,\quad g(u)=\hat{g}_0\exp\left(\int_0^u\frac{1}{\phi(t)}dt\right)\in \bR^+.
\end{equation*}
We can not apply this argument beyond $\mu$ because $\frac{1}{\phi(\nu)}$ has a pole at $\nu=\pm \mu$. However for all $\mu<|u|<\omega_3$, using equation \eqref{g(omega3)} we can write
\begin{align*}
\begin{split}
    g(u)&=\hat{g}_0\exp\left(\int_0^u\frac{1}{\phi(\nu)}d\nu\right)=\hat{g}_0\exp\left(\int_0^{\omega_3}\frac{1}{\phi(\nu)}d\nu+\int_{\omega_3}^{u}\frac{1}{\phi(\nu)}d\nu\right)\\&=g(\omega_3)\exp\left(\int_{0}^1\frac{(u-\omega_3)dt}{\phi(\omega_3(1-t)+tu)}\right)\in \bR^{-}.
    \end{split}
\end{align*}
which proves the claim.
\end{proof}
\begin{lema}
    The conformal factor along the real line is $2\omega_3$-periodic i.e
    \begin{equation*}
        \forall u\in \bR,\quad e^{\omega(u+2\omega_3,0)}=e^{\omega(u,0)}.
    \end{equation*}
\end{lema}
\begin{proof}
    Since the function $\phi(u)$ is clearly $2\omega_3$-periodic, it is enough to prove that $g(u,0)$ is $2\omega_3$-periodic. In fact for all $u\in \bR$
    \begin{align*}
    \begin{split}
        g(u+2\omega_3)&=\hat{g}_0e^{2\zeta(\mu)(u+2\omega_3)}\frac{\sigma(\mu-u-2\omega_3)}{\sigma(\mu+u+2\omega_3)}=\hat{g}_0e^{2\zeta(\mu)(u+2\omega_3)}\frac{e^{2\zeta(\omega_3)(\mu-u-\omega_3)}}{e^{2\zeta(\omega_3)(\mu+u+\omega_3)}}\frac{\sigma(\mu-u)}{\sigma(\mu+u)}\\&=\hat{g}_0e^{2\zeta(\mu)(u+2\omega_3)}e^{-4\zeta(\omega_3)\mu}\frac{\sigma(\mu-u)}{\sigma(\mu+u)}=g(u).
        \end{split}
    \end{align*}
    that is
    \begin{equation}\label{gaussmap2omega3periodicrealline}
         \forall u\in \bR,\quad g(u+2\omega_3)=g(u).
    \end{equation}
\end{proof}
\begin{cor}\label{conformalfactorimmersionrealaxis}
    The conformal factor can be written as
    \begin{gather*}
        \forall u\in \bR,\quad e^{\omega(u,0)}=-\frac{\phi(u,0)}{2}\left(\frac{1}{g(u,0)}+g(u,0)\right).
    \end{gather*}
\end{cor}
\begin{proof}
    In fact, by the periodicity property, we only need to study the conformal factor in the interval $(-\omega_3,\omega_3)$ and since in that interval $\phi(u)$ and $g(u)$ have opposite signs according to equation \eqref{conformalfactorimersion} and Lemmas \ref{signphi}, \ref{signg} we prove our claim.
\end{proof}
\begin{prop}\label{beta(omega3)}
    The $\beta$ function satisfies
    \begin{equation*}
         \beta(\omega_3)=\hat{g}_0-\frac{1}{\hat{g}_0}.
    \end{equation*}
\end{prop}
\begin{proof}
    Taking derivative in the $u$ direction in Corollary \ref{conformalfactorimmersionrealaxis} we obtain
\begin{equation}\label{derivativeconformal}
   \forall u\in \bR,\quad \omega_u(u,0)=\frac{\phi_u(u,0)}{\phi(u,0)}+\frac{1}{\phi(u,0)}\left(\frac{g^2(u,0)-1}{g^2(u,0)+1}\right).
\end{equation}
Since the conformal factor $e^{\omega(\omega_3,0)}$ is finite and since $$\phi_u(\omega_3,0)=-4\wp'(\omega_3)=0,$$ then using equations \eqref{Enneperequation}, \eqref{derivativeconformal}, Proposition \ref{initialconditionsparameterspace} and Corollary \ref{conformalfactorimmersionrealaxis} we obtain
\begin{equation*}
    \beta(\omega_3)=\frac{1}{g(\omega_3,0)}-g(\omega_3,0)=\hat{g}_0-\frac{1}{\hat{g}_0},
\end{equation*}
as we wanted to prove.
\end{proof}
We can not use the same analysis to study the sign of $\beta(0)$ since the conformal factor explodes at $(0,0)$. We then need to study another symmetry property of the conformal factor

\begin{lema}\label{conformalfactoratiomega3'level}
    \begin{equation*}
        \forall u\in \bR,\quad e^{\omega(u,i\abs{\omega_3'})}=e^{\omega(u-\omega_3,0)}.
    \end{equation*}
\end{lema}
\begin{proof}
First notice that for all $u\in \bR$
\begin{equation*}
    \phi(u+\omega_3')=b-\wp(u+\omega_3')=b-\wp(u-\omega_3)=\phi(u-\omega_3).
\end{equation*}
On the other hand
\begin{gather*}
    g(u+\omega_3')=\hat{g}_0\exp\left(\int_0^{u+\omega_3'}\frac{1}{\phi(\nu)}d\nu\right)=\hat{g}_0\exp\left(\int_0^{2\omega_2}\frac{1}{\phi(\nu)}d\nu+\int_{2\omega_2}^{u+\omega_3'}\frac{1}{\phi(\nu)}d\nu\right)\\=\hat{g}_0\exp(-2\pi i\kappa_2)\exp\left(\int_0^{u-\omega_3}\frac{1}{\phi(\nu)}d\nu\right)=\exp(-2\pi i\kappa_2)g(u-\omega_3).
\end{gather*}
Therefore along the curve $C_s$
\begin{equation}\label{modulusglevelomega3'omega3}
    \abs{g(u+\omega_3')}=\abs{g(u-\omega_3)},
\end{equation}
which implies
\begin{align*}
\begin{split}
e^{\omega(u,i\abs{\omega_3'})}&=\frac{\abs{\phi(u+\omega_3')}}{2}\left(\frac{1}{\abs{g(u+\omega_3')}}+\abs{g(u+\omega_3')}\right)\\&=\frac{\abs{\phi(\omega_3-u)}}{2}\left(\frac{1}{\abs{g(u-\omega_3)}}+\abs{g(u-\omega_3)}\right)=e^{\omega(u-\omega_3,0)}.
\end{split}
\end{align*}
\end{proof}
\begin{prop}\label{betafunctionat0omega3}
    The $\beta$ function satisfies
    \begin{equation*}
        \beta(0)=\beta(\omega_3)=\hat{g}_0-\frac{1}{\hat{g}_0}.
    \end{equation*}
\end{prop}
\begin{proof}
Using Corollary \ref{conformalfactorimmersionrealaxis} and Lemma \ref{conformalfactoratiomega3'level} we have
\begin{equation*}
    \forall u\in \bR\quad e^{\omega(u,i\abs{\omega_3'})}= -\frac{\phi(u-\omega_3)}{2}\left(\frac{1}{g(u-\omega_3)}+g(u-\omega_3)\right).
\end{equation*}
In the same way as before, taking derivative in the $u$ direction and evaluating at $u=0$ we obtain
\begin{equation*}
    \beta(0)=-\left(g(-\omega_3)-\frac{1}{g(-\omega_3)}\right).
\end{equation*}
From equation \eqref{gaussmap2omega3periodicrealline} we saw that the Gauss map is $2\omega_3$-periodic on the real line, so using equation \eqref{g(omega3)} and Proposition \ref{beta(omega3)} we have
\begin{equation*}
    \beta(0)=-\left(g(\omega_3)-\frac{1}{g(\omega_3)}\right)=\hat{g}_0-\frac{1}{\hat{g}_0}=\beta(\omega_3).
\end{equation*}
\end{proof}
By the uniqueness of the solution to the Hamiltonian system with initial conditions, we have
\begin{cor}\label{alphabetaomega3periodic}
The Hamiltonian system $(\alpha(u),\beta(u))$ is $\omega_3$ periodic, that is
    \begin{gather*}
        \forall u\in \bR,\quad \begin{cases}
            \alpha(u+\omega_3)=\alpha(u),\\ \beta(u+\omega_3)=\beta(u).
        \end{cases}
    \end{gather*}
\end{cor}
\begin{cor}
    The radii of the spheres $R(u)$ of the spherical curvature lines and the intersection angle  $\theta(u)$ of the surface with the corresponding spheres are $\omega_3$ periodic.
\end{cor}

\section{Associated $(s,t)$ System}\label{Section 5.6}
In order to analyze the behavior of the center we proceed to study the solutions to the Hamiltonian system motivated by the works \cite{Wentespherical} and \cite{Isabel}.

Consider the solution $(s(\lambda),t(\lambda))\in \bR^+\times \bR^-$ to the autonomous ODE system
\begin{gather}\label{stsystem}
    \begin{cases}
        s'(\lambda)^2=s^2(\lambda)\hat{q}(s(\lambda)),\\t'(\lambda)^2=t^2(\lambda)\hat{q}(t(\lambda)).
    \end{cases}
\end{gather}
and define
\begin{equation}\label{udefinition}
    u'(\lambda)=\frac{s(\lambda)-t(\lambda)}{2}>0.
\end{equation}
Then the transformation
\begin{gather}\label{diffeostalphabeta}
    \begin{cases}
        \alpha(u(\lambda))\beta(u(\lambda))=s(\lambda)+t(\lambda),\\ \alpha^2(u(\lambda))=-s(\lambda)t(\lambda),
    \end{cases}
\end{gather}
defines a diffeorphism between the phase space $(s,t)\in \bR^+\times \bR^-$ to either the half plane $(\alpha,\beta)\in \bR^+\times \bR$ or $(\alpha,\beta)\in \bR^-\times \bR$ such that the pair $(\alpha(u),\beta(u))$ is a solution to the Hamiltonian system \ref{hamiltonianequations}. Moreover every solution to the Hamiltonian system \eqref{hamiltonianequations} comes from this procedure. Following \cite[equation 3.3]{Isabel}, the $(s,t)$ system can be broken into the following four systems depending on $\epsilon\in \{1,-1\}$
\begin{gather*}
   a) \begin{cases}
        s'(\lambda)=\epsilon s(\lambda)\sqrt{\hat{q}(s(\lambda)})\\ t'(\lambda)=\epsilon t(\lambda)\sqrt{\hat{q}(t(\lambda)})
    \end{cases}\quad  b)\begin{cases}
        s'(\lambda)=-\epsilon s(\lambda)\sqrt{\hat{q}(s(\lambda)})\\ t'(\lambda)=\epsilon t(\lambda)\sqrt{\hat{q}(t(\lambda)})
    \end{cases}.
\end{gather*}
We proceed to describe the solutions of the $(s,t)$ system coming from the Hamiltonian system associated to the minimal immersion $\Psi_{L(\tau,s),\hat{g}_0}$. 
\begin{figure}
    \centering
    \includegraphics[width=0.7\linewidth]{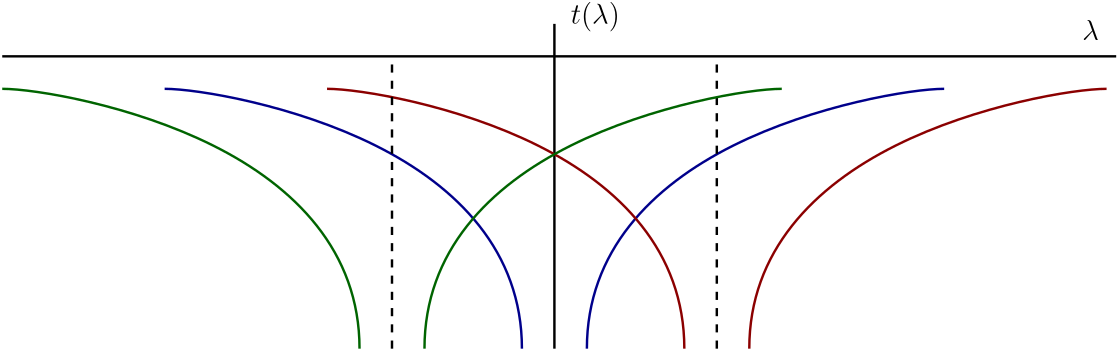}
    \caption{The $t(\lambda)$ orbits for the system $t'(\lambda)^2=t^2(\lambda)\hat{q}(t(\lambda))$}.
    \label{fig:t orbit}
\end{figure}
\subsection{The $s$ orbits}
Observe that from equations \eqref{polynomialq}, \eqref{rootproperties}
\begin{equation}\label{rewriteqpolynomial}
    \hat{q}(x)=(r_3-x)Q(x),\quad r_3\in \bR^+,
\end{equation}
where the polynomial $Q(x)=(x-r_1)(x-r_2)$ has non-real complex conjugate roots $r_2=\overline{r_1}$ so that
\begin{equation}\label{lowerboundQ}
    \exists C>0,\ \forall x\in \bR,\quad  Q(x)=x^2-2Re(r_1)x+|r_1|^2>C.
\end{equation}
\begin{lema}
    $s(\lambda)\in (0,r_3)$.
\end{lema}
\begin{proof}
    In fact since by equation \eqref{stsystem}$, \hat{q}(s(\lambda))>0$, it follows from equations \eqref{rewriteqpolynomial} and \eqref{lowerboundQ} that $r_3-s(\lambda)>0$ where the conclusion follows.
\end{proof}
Let's fix initial conditions $s(0)=s_0\in (0,r_3)$. Recall that since we are dealing with an autonomous system, if $s(\lambda)$ is a solution with initial condition $s(0)=s_0$ then $\tilde{s}(\lambda)\coloneqq s(\lambda-\lambda_*)$ is a solution to the ODE with initial condition $\tilde{s}(\lambda_*)=s_0.$ We begin by proving that $s(\lambda)$ is not asymptotic to the upper bound $r_3.$
\begin{lema}
    The solution $s'(\lambda)=s(\lambda)\sqrt{\hat{q}(s(\lambda))}$ with $s(0)=s_0$ reaches $s=r_3$ in finite $\lambda$-time, \textit{i.e.} at $\lambda<+\infty$.
\end{lema}
\begin{proof}
    \begin{equation*}
     \int_{s_0}^{s(\lambda)} \frac{ds}{s\sqrt{\hat{q}(s)}}=\int_0^\lambda d\lambda=\lambda.
    \end{equation*}
    Make the change of variables $\xi=r_3-s$
    then
    \begin{equation*}
          \int_{s_0}^{r_3} \frac{ds}{s\sqrt{\hat{q}(s)}}=\int_{0}^{r_3-s_0}\frac{d\xi}{(r_3-\xi)\sqrt{\xi}\sqrt{Q(r_3-\xi})}\leq \frac{2\sqrt{r_3-s_0}}{s_0\sqrt{C}}<+\infty.
    \end{equation*}
\end{proof}
\begin{lema}\label{sdecaysexponentially}
    The $s$ solution to $s'(\lambda)=-s(\lambda)\sqrt{\hat{q}(s(\lambda)}$ converges exponentially to zero when $\lambda\to +\infty$ and it reaches $s=0$ at infinite $\lambda$-time, \textit{i.e.} at $\lambda=+\infty$.
\end{lema}
\begin{proof}
    Since $s_0\in (0,r_3)$ and since $s'(\lambda)<0$ we have that $r_3-s(\lambda)>r_3-s_0$ and then
    \begin{equation*}
        -\frac{s'(\lambda)}{s(\lambda)}=\sqrt{\hat{q}(s(\lambda))}\geq \sqrt{r_3-s_0}\sqrt{C}>0,
    \end{equation*}
    so that
    \begin{equation*}
        s(\lambda)\leq s_0e^{-\sqrt{r_3-s_0}\sqrt{C}\lambda},
    \end{equation*}
    which proves that $s(\lambda)\to 0$ as $\lambda\to +\infty.$
    
We next prove that 
    \begin{equation*}
        \int_{s_0}^0\frac{ds}{-s\sqrt{\hat{q}(s)}}=+\infty.
    \end{equation*}
    In fact since $s$ is bounded by above, so is $\hat{q}(s)$ $$\forall s\in (0,s_0),\quad 0<\hat{q}(s)\leq K\ne0,$$ and then
    \begin{equation*}
        \int_{\epsilon}^{s_0}\frac{ds}{s\sqrt{\hat{q}(s(\lambda))}}\geq \frac{1}{\sqrt{K}}(\ln (s_0)-\ln (\epsilon))\to +\infty ,\quad \epsilon\to 0.
    \end{equation*}
\end{proof}
\begin{cor}\label{integralsfinite}
Let $s(\lambda)$ be the analytic continuation solution to the system $s'(\lambda)^2=s^2(\lambda)\hat{q}(s(\lambda))$ with initial condition $s(0)=s_0$ and a choice of sign of $s'(0)$. Then
    \begin{equation*}
        \int_{-\infty}^{\infty}s(\lambda)d\lambda<+\infty.
    \end{equation*}
\end{cor}
\subsection{The $t$ orbits}
We again fix initial conditions $t(0)=t_0\in (-\infty,0)$. We begin by proving that the system $t'(\lambda)=t(\lambda)\sqrt{\hat{q}(t(\lambda))}$ develops a singularity in finite $\lambda$-time.
\begin{lema}\label{tdevelopssingularity}
    \begin{equation}
        \int_{t_0}^{-\infty}\frac{dt}{t\sqrt{\hat{q}(t)}}<+\infty.
    \end{equation}
\end{lema}
\begin{proof}
    In fact, make the change of variable $\xi=-t$ and notice that there exists a constant $c>0$ such that for $\xi$ large enough we have $x\sqrt{\hat{q}(-x)}\geq cx^{\frac{5}{2}}$. The result follows by comparison with the convergent integral
    \begin{equation*}
        \int_{1}^{\infty}\frac{d\xi}{\xi^{\frac{5}{2}}}<+\infty.
    \end{equation*}
\end{proof}
The proof of the following Lemma is analogous to that of Lemma \ref{sdecaysexponentially}.
\begin{lema}
    The solution to $t'(\lambda)=-t(\lambda)\sqrt{\hat{q}(t(\lambda))}$ converges exponentially to $0$ and takes infinite $\lambda$-time to reach $0$.
\end{lema}
\begin{lema}\label{integraltfinite}
Let $t(\lambda)$ be the analytic continuation solution to the system $t'(\lambda)^2=t^2(\lambda)\hat{q}(t(\lambda))$ with initial condition $t(0)=t_0$ and a choice of sign of $t'(0)$. Then
    \begin{equation*}
        \int_{-\infty}^{\infty}t(\lambda)d\lambda<+\infty.
    \end{equation*}
\end{lema}
\begin{proof}
    It is enough to prove that the solution to the system $t'(\lambda)=t(\lambda)\sqrt{\hat{q}(t(\lambda))}$ with initial condition $t(0)=t_0$ satisfies that
    \begin{equation*}
        \int_0^Lt(\lambda)d\lambda<+\infty,
    \end{equation*}
    where $L$ is the $\lambda$- time of Lemma \ref{tdevelopssingularity} where the system $t(\lambda)$ develops a singularity \textit{i.e.} $\lim_{\lambda\to L}t(\lambda)=-\infty.$ Make the change of variable $\xi=t(\lambda)$. Then
    \begin{equation*}
        d\xi=t'(\lambda)d\lambda=\xi\sqrt{\hat{q}(\xi)}d\lambda,
        \end{equation*}
        and then
        \begin{equation*}
            \int_0^Lt(\lambda)d\lambda=\int_{t_0}^{-\infty}\frac{d\xi}{\sqrt{\hat{q}(\xi)}}=-\int_{-t_0}^{+\infty}\frac{d\xi}{\sqrt{\hat{q}(-\xi)}}.
        \end{equation*}
        Notice that there exists a constant $c>0$ such that for $\xi$ large enough it holds that
        \begin{equation*}
            \hat{q}(-\xi)\geq c\xi^3.
        \end{equation*}
        The result follows by a comparison with the convergent integral
        \begin{equation*}
            \int_1^{+\infty}\frac{d\xi}{\xi^{\frac{3}{2}}}<+\infty.
        \end{equation*}
\end{proof}
Using  equation \eqref{udefinition}, Corollary \ref{integralsfinite} and Lemma \ref{integraltfinite} we see that the function $u(\lambda)$ sends the real line $(-\infty,+\infty)$ to a bounded interval $(a,b)$. Let us set $u(-\infty)=0$ so that
\begin{equation*}
    u(\lambda)=\frac{1}{2}\int_{-\infty}^{\lambda}(s(\nu)-t(\nu))d\nu.
\end{equation*}
Notice we have started from a minimal immersion $\Psi_{L(\tau,s),\hat{g}_0}$ with well-defined conformal factor in the strip $(0,\omega_3)\times \bR$, which has a well-defined solution to the system $(\alpha(u),\beta(u))$ in the $u$ interval $(0,\omega_3)$. On the other hand we have just described the behavior of the system $(s(\lambda),t(\lambda))$ from which we can recover the $(\alpha,\beta)$ solutions, so it should be that $u(+\infty)=\omega_3$ for the choice $u(-\infty)=0$. We now prove this fact 

\begin{prop}
     \begin{equation*}
        u(+\infty)=\omega_3.
    \end{equation*}
\end{prop}
\begin{proof}
    \begin{gather*}
        u(+\infty)=\frac{1}{2}\left(\int_{-\infty}^{\lambda_s}s(\lambda)d\lambda+\int_{\lambda_s}^{+\infty}s(\lambda)d\lambda\right)-\frac{1}{2}\left(\int_{-\infty}^{\lambda_s}t(\lambda)d\lambda+\int_{\lambda_s}^{+\infty}t(\lambda)d\lambda\right).
    \end{gather*}
    Making the change of variables $\xi=s(\lambda)$ and $\xi=t(\lambda)$ we have
    \begin{gather*}
        u(+\infty)=\frac{1}{2}\left(\int_0^{r_3}\frac{ds}{\sqrt{\hat{q}(s)}}+\int_{r_3}^0\frac{-ds}{\sqrt{\hat{q}(s)}}\right)-\frac{1}{2}\left(\int_0^{-\infty}\frac{dt}{\sqrt{\hat{q}(t)}}+\int_{-\infty}^0\frac{-dt}{\sqrt{\hat{q}(t)}}\right),
    \end{gather*}
    \begin{equation*}
    u(+\infty)=\int_{-\infty}^{r_3}\frac{dx}{\sqrt{\hat{q}(x)}}.
    \end{equation*}
    Since the restriction of the function 
    \begin{equation*}
        \phi(\cdot)=b-4\wp(\cdot):(0,\omega_3)\to (-\infty,b-4\wp(\omega_3))=(-\infty,r_3),
    \end{equation*}
    defines a diffeomorphism, we can make the change of variables $x=\phi(\xi)$
    \begin{equation*}
        u(+\infty)=\int_{0}^{\omega_3}\frac{\phi'(\xi)d\xi}{\sqrt{\hat{q}(\phi(\xi))}},
    \end{equation*}
    using equation \eqref{phiequation} and the fact that $\phi'(\cdot)=-4\wp(\cdot)>0$ on $ (0,\omega_3)$ then
    \begin{equation*}
        u(+\infty)=\int_{0}^{\omega_3}\frac{\phi'(\xi)d\xi}{\phi'(\xi)}=\int_0^{\omega_3}d\xi = \omega_3,
    \end{equation*}
   which concludes the proof.
\end{proof}
\begin{figure}
    \centering
    \includegraphics[width=0.5\linewidth]{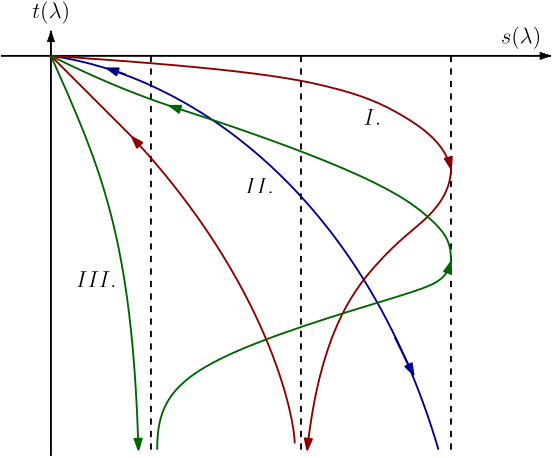}
    \caption{The orbits $(s(\lambda),t(\lambda))$.}
    \label{fig:st phase space}
\end{figure}
\begin{thm}\label{qualitativest}
    In phase space $(s,t)$ we have the following qualitative behavior. All the orbits start at $\lambda=-\infty$ and end at $\lambda=+\infty$ at the origin, converging exponentially to zero as $|\lambda|\to +\infty$. Along an orbit $(s(\lambda),t(\lambda))$ let us denote by $\lambda_s$ and $\lambda_t$ the finite $\lambda$-times where the $s$ solution reaches $s=r_3$ and where the $t$ orbit diverges $t\to -\infty$ respectively. We have three kinds of orbits:
\begin{itemize}
    \item[I.] $\lambda_s<\lambda_t$. In this case the orbit starts at the origin in $\lambda=-\infty$ and reaches the wall $s=r_3$ at $\lambda_s$-time and then diverges in $\lambda_t$ time. Finally the system converges to the origin in infinite $\lambda$-time.
\item[II.] $\lambda_s=\lambda_t$. The system starts at the origin, and then diverges to $(r_3,-\infty)$ in finite $\lambda$-time. Subsequently the system returns to the origin along the same orbit finishing at the origin.
\item[III.] $\lambda_s>\lambda_t$. In this case the system starts at the origin and diverges at $\lambda_t$-time, subsequently it hits the wall $s=r_3$ at $\lambda_s$ and then converges to the origin.
\end{itemize}
\end{thm}
We can now understand the qualitative behavior of the $\alpha(u),\beta(u))$ system for $u\in (0,\omega_3)$ associated to the minimal immersion $\Psi_{L(\tau,s),\hat{g}_0}$.
\begin{thm}\label{alphabetaqualitative}
    The functions $\alpha(u)$ and $\beta(u)$ have a unique singularity at $u(\lambda_t)\in (0,\omega_3)$ and their asymptotic behavior is given by
    \begin{gather*}
    \begin{cases}
        \lim_{u\to u^-(\lambda_t)}\alpha(u)=-\infty\\ \lim_{u\to u^-(\lambda_t)}\beta(u)=+\infty 
    \end{cases},\quad 
    \begin{cases}
        \lim_{u\to u^+(\lambda_t)}\alpha(u)=+\infty\\ \lim_{u\to u^+(\lambda_t)}\beta(u)=-\infty 
    \end{cases},\quad  \begin{cases}
        \alpha(u)\in \bR^-,\quad u\in (0,u(\lambda_t))\\ \alpha(u)\in \bR^+,\quad u\in (u(\lambda_t),\omega_3)
    \end{cases}.
\end{gather*}
Moreover the radii $R(u)$ of the spherical curvature lines are finite in the interval $(0,\omega_3)$ and the angle of intersection $\theta(u)\to0\ \text{or}\ \pi$ as $u\to u(\lambda_t)$.
\end{thm}
\begin{proof}
    Consider the associated $(s,t)$ system. We have the following two possible solutions according to equation \eqref{diffeostalphabeta}
\begin{gather}\label{alphabetast}
     \begin{cases}
        \alpha(u(\lambda))=\epsilon \sqrt{-s(\lambda)t(\lambda)}\\ \beta(u(\lambda))=\epsilon\left(\frac{\sqrt{s(\lambda)}}{\sqrt{-t(\lambda)}}-\frac{\sqrt{-t(\lambda)}}{\sqrt{s(\lambda)}}\right)
    \end{cases},\quad \epsilon\in \{-1,+1\}.
\end{gather}
Observe that by the previous equation and Theorem \ref{qualitativest}, $\alpha(u)$ and $\beta(u)$ have a unique singularity at $u(\lambda_t)\in (0,\omega_3)$. Now by Proposition \ref{initialconditionsparameterspace}, the equations $\alpha(0)=0$ and $\alpha'(0)<0$ imply that $\alpha(u)<0$ for $u$ sufficiently close to $0^{+}$, whereas the equations $\alpha(\omega_3)=0$ and $\alpha'(\omega_3)<0$ imply that $\alpha(u)>0$ for $u$ sufficiently near to $\omega_3^{-}$. Therefore we must have that
\begin{gather*}
    \forall u\in (0,u(\lambda_t)),\quad \begin{cases}
         \alpha(u(\lambda))=-\sqrt{-s(\lambda)t(\lambda)}\\ \beta(u(\lambda))=-\left(\frac{\sqrt{s(\lambda)}}{\sqrt{-t(\lambda)}}-\frac{\sqrt{-t(\lambda)}}{\sqrt{s(\lambda)}}\right)
    \end{cases},
\end{gather*}
and
\begin{gather*}
    \forall u\in (u(\lambda_t),\omega_3),\quad \begin{cases}
         \alpha(u(\lambda))=\sqrt{-s(\lambda)t(\lambda)}\\ \beta(u(\lambda))=\left(\frac{\sqrt{s(\lambda)}}{\sqrt{-t(\lambda)}}-\frac{\sqrt{-t(\lambda)}}{\sqrt{s(\lambda)}}\right).
    \end{cases}
\end{gather*}
From which the asympotic behavior follows. For the last claim, even though $(\alpha,\beta)$ have a singularity, notice that in terms of the associated $(s,t)$ system we have by equation \eqref{radiusangle} that
\begin{equation*}
R^2(u(\lambda))=\left(\frac{1}{t(\lambda)}-\frac{1}{s(\lambda)}\right)^2,
\end{equation*}
so the radius is finite for every $u\in (0,\omega_3)$.
\end{proof}

  \begin{figure}
     \centering
     \begin{subfigure}[b]{0.48\textwidth}
         \centering
         \includegraphics[width=\textwidth]{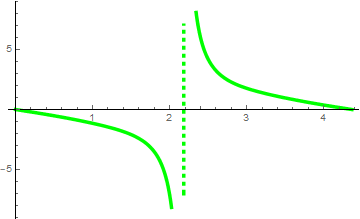}
         \caption{Numerical simulation of the function $\alpha(u)$ in the interval $(0,\omega_3)$.}
         \label{fig:alpha}
     \end{subfigure}
     \hfill
     \begin{subfigure}[b]{0.48\textwidth}
         \centering
         \includegraphics[width=\textwidth]{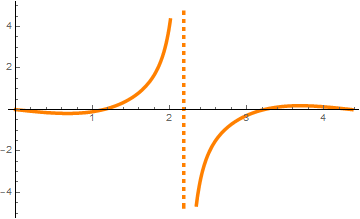}
         \caption{Numerical simulation of the function $\beta(u)$ in the interval $(0,\omega_3)$.}
         \label{fig:beta}
     \end{subfigure}
      \caption{Numerical simulation of the pair $(\alpha,\beta)$ for the choice of parameters $\tau\approx 0.503$ and $\hat{g}_0=1$.}
     \end{figure}
\section{Center Analysis}\label{Section 5.7}
We now proceed to make an analysis of the behavior of the center of the spheres containing the spherical lines of curvature of $\Psi_{L(\tau,s),\hat{g}_0}$. The next theorem shows that it is enough to understand how the center behaves in the special interval $(0,\omega_3).$
 
\begin{thm}\label{periodicitycenter}
    For all $\tau\in (0,\tau_0)$ there exists a non-zero real constant
    \begin{equation*}
        c_{\tau}\coloneqq 4\left(\omega_3\wp(\mu)+\zeta(\omega_3)\right),
    \end{equation*} 
    such that for all $\hat{g}_0$ and all $u\in \bR$
    \begin{equation*}
        c(u+\omega_3)=c(u)+(0,0,c_{\tau}).
    \end{equation*}
\end{thm}
\begin{proof}
Since the center lies on an axis parallel to the $x_3$-axis, the coordinates $c_1(u),c_2(u)$ are constant, so we restrict to the third coordinate of equation \eqref{centerequationgeneral}. By construction
\begin{equation}\label{derivativeuimmersion}
    \abs{\Psi_u(u,v)}=e^{\omega(u,v)},\quad  \Psi_3(u,v)=Re\int_{\frac{\omega_3}{2}}^{u+iv}\phi(\nu)d\nu.
\end{equation}
By the Cauchy-Riemann equations
\begin{equation*}
(\Psi_3(u,v))_u=Re\ \phi(u+iv).
\end{equation*}
Furthermore since $\wp(z)=-\zeta'(z)$
\begin{gather}\label{thirdcoordinatePsi}
    \Psi_3(u,v)=Re\int_{\frac{\omega_3}{2}}^{u+iv}(b+4\zeta'(\nu))d\nu=b(u-\frac{\omega_3}{2})+4\ (\frac{\zeta(u+iv)+\zeta(u-iv)}{2}-\zeta(\frac{\omega_3}{2})).
\end{gather}
We know that
\begin{equation}\label{normalthird}
    (N_{\Sigma}(\Psi(u,v))_3=\frac{|g(u,v)|^2-1}{|g(u,v)|^2+1}.
\end{equation}
First we analyze at $v=0$ and we obtain
\begin{equation}\label{centerthirdu}
    c_3(u)=\Psi_3(u,0)-\frac{2}{\alpha(u)}(e^{-\omega(u,0)}\phi(u))-\frac{\beta(u)}{\alpha(u)}(N_{\Sigma}(\Psi(u,0)))_3,
\end{equation}
\begin{equation}\label{alternc3u}
    c_3(u)=b(u-\frac{\omega_3}{2})+4(\zeta(u)-\zeta(\frac{\omega_3}{2}))+\frac{4}{\alpha(u)}\frac{g(u,0)}{g^2(u,0)+1}-\frac{\beta(u)}{\alpha(u)}\frac{g^2(u,0)-1}{g^2(u,0)+1}.
\end{equation}
\begin{rem}
\normalfont{
    This equation was already obtained in \cite[equation 5.26]{Isabel} for rectangular lattices.}
\end{rem}
Now we investigate what happens at $v=|\omega_3'|$ in equation \eqref{centerequationgeneral}. We study the terms separately. Notice that 
\begin{equation}\label{psiomega3'derivativeu}
    (\Psi_3(u+\omega_3'))_u=\phi(u+\omega_3')=\phi(u-\omega_3).
\end{equation}
We also have by equation \eqref{derivativeuimmersion} and Lemma \ref{conformalfactoratiomega3'level}
\begin{equation}\label{abspsiuomega3'}
    |\Psi_u(u+\omega_3')|=e^{\omega(u+\omega_3')}=e^{\omega(u-\omega_3)}.
\end{equation}
Also by equations \eqref{modulusglevelomega3'omega3} and \eqref{normalthird} we have
\begin{equation}\label{Normalthirdomega3'}
    (N_{\Sigma}(\Psi(u+\omega_3'))_3= (N_{\Sigma}(\Psi(u-\omega_3))_3.
\end{equation}
By equations \eqref{identityzeta} and \eqref{thirdcoordinatePsi} evaluated at $v=|\omega_3'|$
\begin{gather}\label{psiumaisomega3'}
    \Psi_3(u+\omega_3')=b(u-\frac{\omega_3}{2})+4(\zeta(u)-\zeta(\frac{\omega_3}{2}))+\frac{2\wp'(u)}{\wp(u)-\wp(\omega_3')}.
\end{gather}

On the other hand using equation \eqref{thirdcoordinatePsi} at $(u-\omega_3,0)$ we have
\begin{gather}\label{psiuminusomega3}
    \Psi_3(u-\omega_3,0)=b(u-\frac{\omega_3}{2})-b\omega_3+4(\zeta(u-\omega_3)-\zeta(\frac{\omega_3}{2})).
\end{gather}
By the identity
\begin{equation*}
    \zeta(z+y)-\zeta(z)-\zeta(y)=\frac{1}{2}\frac{\wp'(z)-\wp'(y)}{\wp(z)-\wp(y)},
\end{equation*}
we have
\begin{equation}\label{zetauminusomega3}
    \zeta(u-\omega_3)=\zeta(u)-\zeta(\omega_3)+\frac{1}{2}\frac{\wp'(u)}{\wp(u)-\wp(\omega_3)},
\end{equation}
and after algebraic manipulations with equations \eqref{psiumaisomega3'}, \eqref{psiuminusomega3} and \eqref{zetauminusomega3} we obtain
\begin{equation}\label{psiu+omega3'}
    \Psi_3(u+\omega_3')=\Psi_3(u-\omega_3,0)+4\left(\omega_3\wp(\mu)+\zeta(\omega_3)\right).
\end{equation}
Collecting the previous calculations \eqref{psiomega3'derivativeu}, \eqref{abspsiuomega3'}, \eqref{Normalthirdomega3'} and \eqref{psiu+omega3'}, the third coordinate of equation \eqref{centerequationgeneral} at $(u,i|\omega_3'|)$ can be written as
\begin{gather}\label{cthirdualternative}
    c_3(u)=\Psi_3(u-\omega_3,0)-\frac{2}{\alpha(u)}\left(e^{-\omega(u-\omega_3)}\phi(u-\omega_3)\right)-\frac{\beta(u)}{\alpha(u)}(N_{\Sigma}(\Psi(u-\omega_3)))_3+4\left(\omega_3\wp(\mu)+\zeta(\omega_3)\right).
\end{gather}
Using Corollary \ref{alphabetaomega3periodic} we have $\alpha(u)=\alpha(u-\omega_3)$ and $\beta(u)=\beta(u-\omega_3)$ so using equations \eqref{centerthirdu} evaluated at $u-\omega_3$ and \eqref{cthirdualternative}
\begin{equation*}
    c_3(u)=c_3(u-\omega_3)+4\left(\omega_3\wp(\mu)+\zeta(\omega_3)\right),
\end{equation*}
and the proof is completed after an application of Item I) of Lemma \ref{criterio}, together with the homogeneity property to conclude that $c_\tau\neq 0$.
\end{proof}
 \begin{figure}
    \centering
    \begin{subfigure}[b]{0.47\textwidth}
    \includegraphics[width=\linewidth]{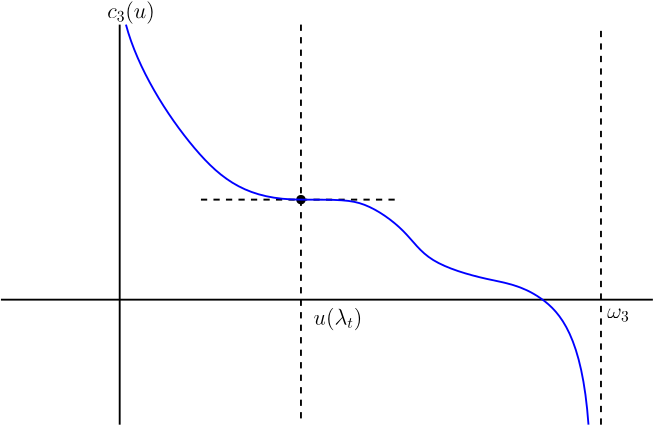}
    \caption{Qualitative behavior of the center of the spherical curvature lines for any $\hat{g}_0\in \bR^{+}.$}
    \label{fig:center graph}
    \end{subfigure}
     \hfill
     \begin{subfigure}[b]{0.47\textwidth}
         \centering
         \includegraphics[width=\textwidth]{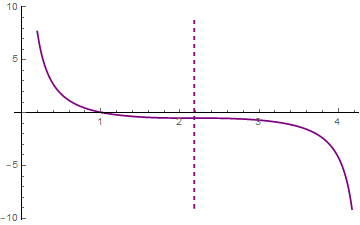}
         \caption{Numerical simulation of the function $c_3(u)$ in the interval $(0,\omega_3)$ for $\tau\approx 0.503$ and $\hat{g}_0=1$.}
         \label{fig:center}
     \end{subfigure}
     \caption{The height center function $c_3(u).$}
\end{figure}
\begin{thm}\label{centeranalysis}
The components of the center function $c(u)$ are finite for all $u\in (0,\omega_3)$. Moreover the height of the center $c_3(u)$ is strictly monotone decreasing in the whole interval, and has a unique critical point $u(\lambda_t)\in (0,\omega_3)$ where $c_3'(u(\lambda_t))=0.$
\end{thm}
\begin{proof}
We start by observing that from Theorem \ref{qualitativest} the functions $s(\lambda)$ and $t(\lambda)$ only vanish at $\lambda\to \pm \infty$. From equation \eqref{alphabetast} this implies that the $\alpha$ function is only zero at $u=0,\omega_3$ in the interval $[0,\omega_3]$. On the other hand by Theorem \ref{alphabetaqualitative} the only singularity of the $\alpha$, $\beta$ functions is at $u(\lambda_t)\in (0,\omega_3)$ in that interval. From equation \eqref{alphabetast}
\begin{equation*}
    \frac{\beta(u(\lambda))}{\alpha(u(\lambda))}=-\frac{1}{t(\lambda)}-\frac{1}{s(\lambda)}\to -\frac{1}{s(\lambda_t)}<+\infty,\quad \text{as}\quad  \lambda\to \lambda_t.
\end{equation*}
Therefore by equation \eqref{alternc3u} and since the coordinates $c_1(u),c_2(u)$ are independent of $u$, we check that the components of $c(u)$ are finite in the interval $(0,\omega_3).$

Following \cite[equation 4.14]{Wentespherical}, \cite[equation 5.29]{Isabel} we have that
\begin{equation}\label{derivativecenter}
    c'(u)=\frac{1}{\alpha(u)^2}\Vec{F}(u),
\end{equation}
where 
\begin{gather}\label{Fvector}
    \Vec{F}\coloneqq \left(2\alpha'+\alpha^2e^{\omega}-\alpha\beta e^{-\omega}\right)e^{-\omega}\Psi_u+\left(2\alpha\omega_v\right)e^{-\omega}\Psi_v+\left(2\alpha e^{-\omega}+\alpha'\beta-\beta'\alpha\right)N_{\Sigma},
\end{gather}
is a vector only dependent on $u$ and parallel to the $z$ axis, \textit{i.e.} $\vec{F}(u)=(0,0,F_3(u))$ and satisfies
\begin{equation*}
    |F|^2=4k.
\end{equation*}
Since we have shown in Proposition \ref{initialconditionsparameterspace} that $k=1$ also holds for rhombic lattices, we have generalized the following formula for such rhombic lattices
\begin{equation*}
    c_3'(u)^2=\frac{4}{\alpha(u)^2}.
\end{equation*}
We see that according with Theorem \ref{alphabetaqualitative} there exists a unique $u=u(\lambda_t)\in (0,\omega_3)$ such that $c_3'(u)=0$. This implies that $c_3(u)$ is monotone in both intervals $(0,u(\lambda_t))$ and $(u(\lambda_t),\omega_3)$. In order to determine the sign of the derivative we study the sign of $c_3'(u)$ as $u$ approaches $0$ and $\omega_3$. Notice that from equation \eqref{derivativecenter} the sign of $c_3'(u)$ is determined by the sign of $F_3(u)$.

Since the conformal factor explodes at $u=0$ along $v=0$ we have to analyze the limit of $F_3(u)$ along $v=|\omega_3'|$. Using the third component of equation \eqref{Fvector} evaluated at $(0,|\omega_3'|)$, Propositions \ref{initialconditionsparameterspace} \ref{betafunctionat0omega3}, equation \eqref{psiomega3'derivativeu}
\begin{gather*}
    \lim_{u\to 0^+}F_3(u)=2\left(-\frac{1}{C}\right)e^{-\omega(0,|\omega_3'|)}\phi(\omega_3')+\left(-\frac{1}{C}\right)\left(\hat{g}_0-\frac{1}{\hat{g}_0}\right)\left(N_{\Sigma}\right)_3(\omega_3').
\end{gather*}
Using equations \eqref{Normalthirdomega3'} and \eqref{normalthird}, the $2\omega_2$ periodicity of $\phi$, the  $2\omega_3$ periodicity of the Gauss map along the real axis \eqref{gaussmap2omega3periodicrealline} and equation \eqref{g(omega3)} we have
\begin{equation*}
     \lim_{u\to 0^+}F_3(u)=2\left(-\frac{1}{C}\right)e^{-\omega(0,|\omega_3'|)}\left(4\wp(\mu)-4\wp(\omega_3)\right)+\left(-\frac{1}{C}\right)\left(\hat{g}_0-\frac{1}{\hat{g}_0}\right)\left(\frac{\hat{g}_0^2-1}{\hat{g}_0^2+1}\right),
\end{equation*}
and therefore
\begin{equation*}
      \lim_{u\to 0^+}F_3(u)\in \bR^-,
\end{equation*}
so that
\begin{equation*}
    \forall u\in (0,u(\lambda_t)),\quad c_3'(u)<0.
\end{equation*}
For the analysis of the sign of $c_3'(u)$ at the limit $u=\omega_3$ we focus at points $v=0$ since the conformal factor is finite at $\omega_3$. Performing a completely analogous calculation we can check that
\begin{equation*}
    \lim_{u\to \omega_3^-}F_3(u)\in \bR^-
\end{equation*}
so that 
\begin{equation*}
    \forall u\in (u(\lambda_t),\omega_3),\quad c_3'(u)<0,
\end{equation*}
We conclude that $c_3(u)$ is strictly decreasing in $(0,\omega_3)$ except at the unique critical point $u(\lambda_t)$. 
\end{proof}
\begin{figure}
     \centering
     \begin{subfigure}[b]{0.3\textwidth}
         \centering
         \includegraphics[width=\textwidth]{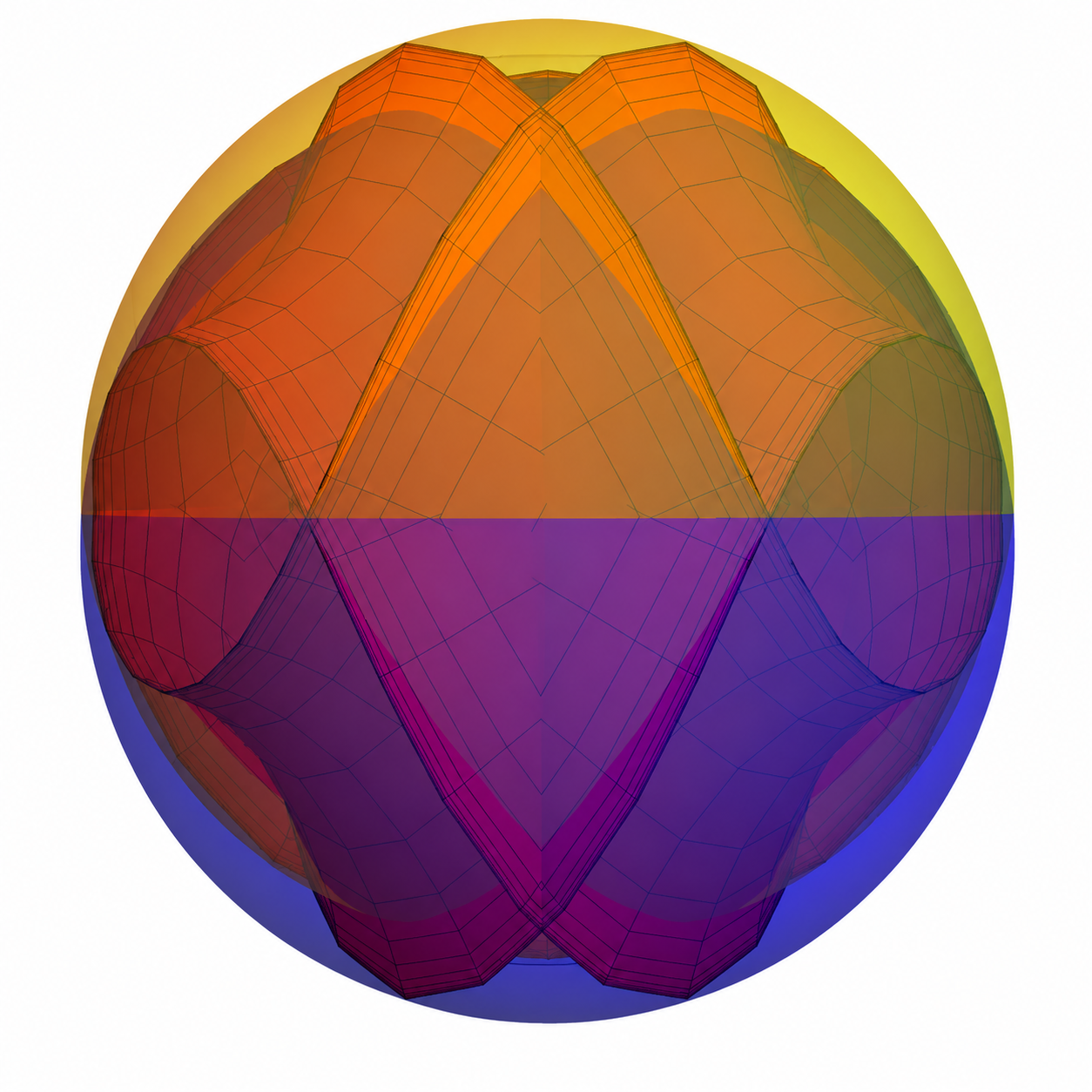}
    \label{fig:capillary n=3}
     \end{subfigure}
      \hfill
      \begin{subfigure}[b]{0.3\textwidth}
         \centering
         \includegraphics[width=\textwidth]{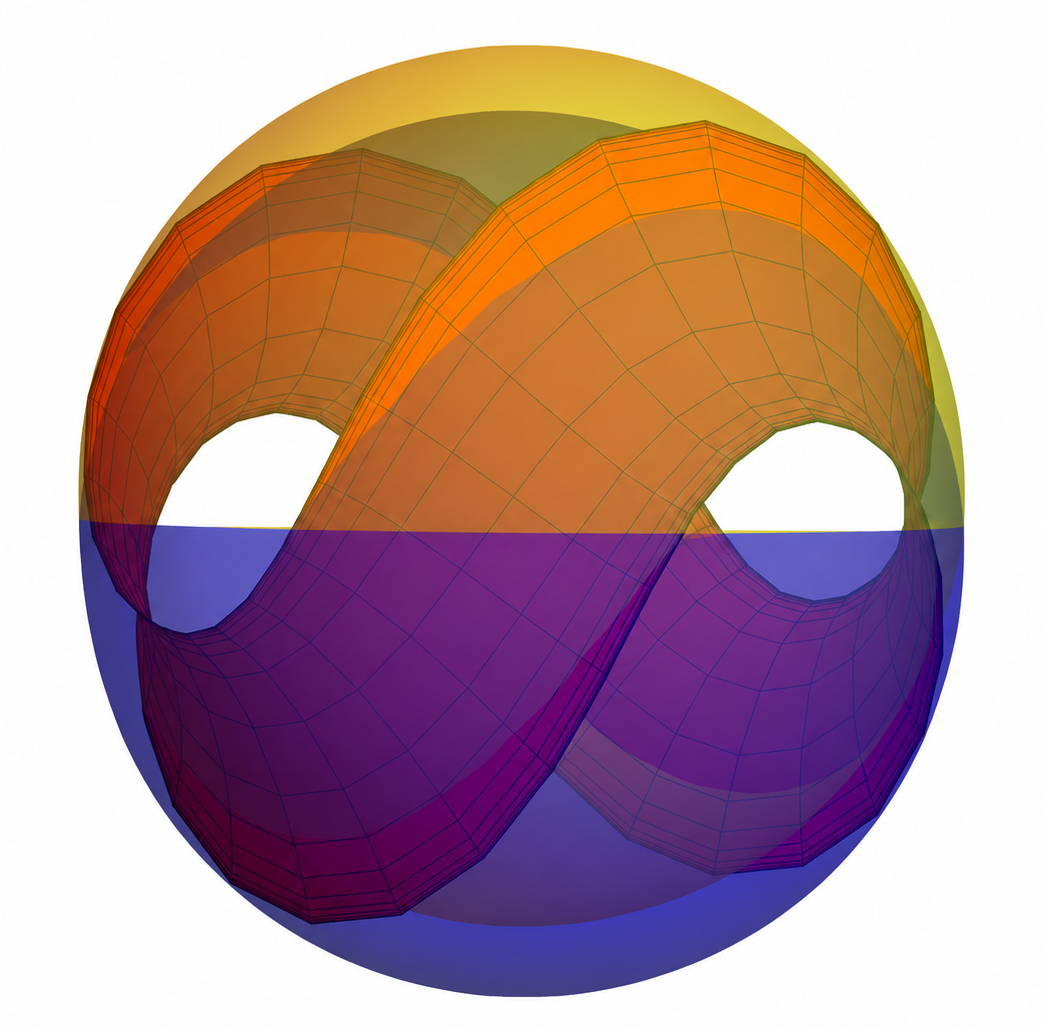}
         \label{fig:capillary n=4}
     \end{subfigure}
     \hfill
     \begin{subfigure}[b]{0.32\textwidth}
         \centering
         \includegraphics[width=\textwidth]{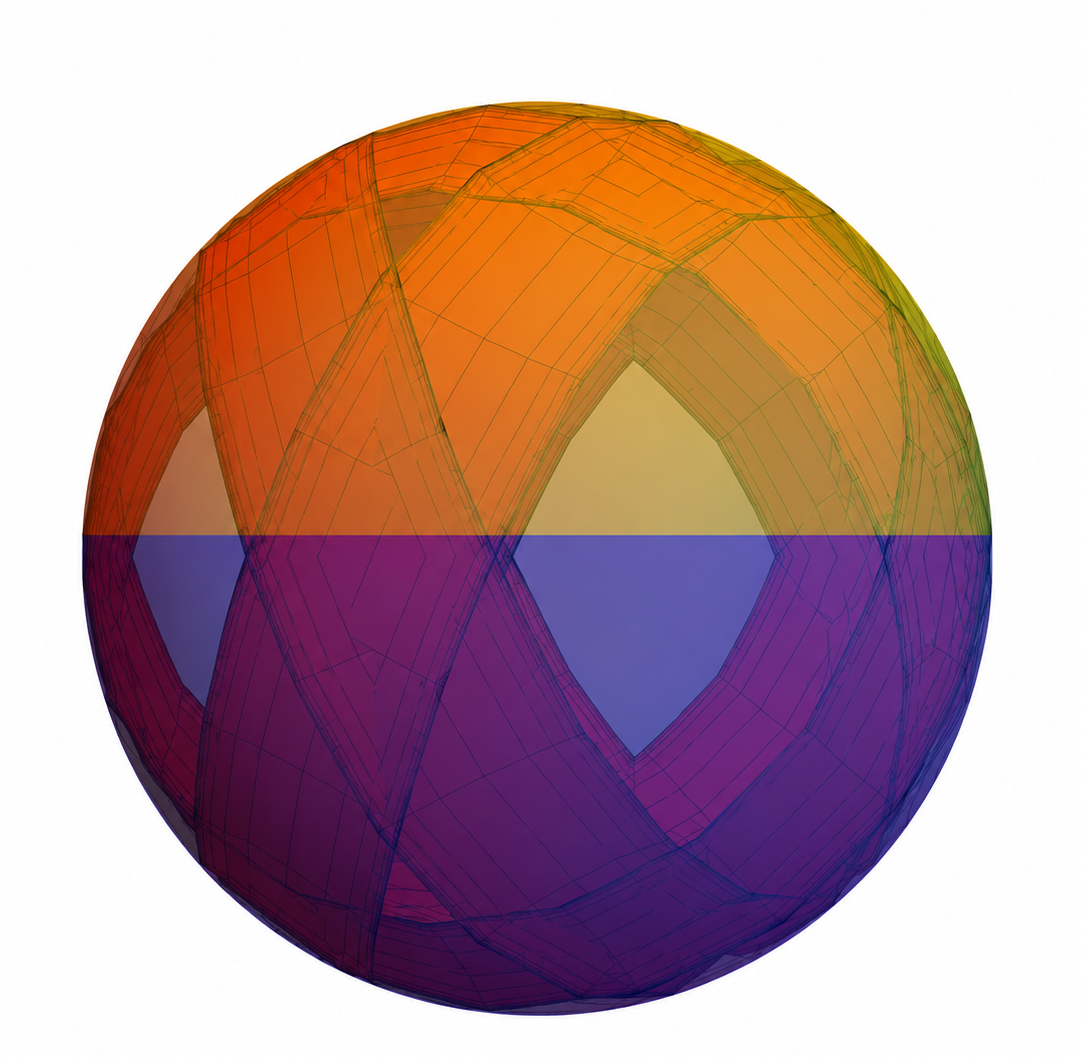}
         \label{fig:capillary n=5}
     \end{subfigure}
     
     \caption{Immersed and embedded capillary minimal annuli in two spheres of same radius and supplementary intersection angles.}
     \end{figure}
\begin{rem}
\normalfont
    It is interesting to remark that from Theorem \ref{canonicalconstruction} and from Theorem \ref{centeranalysis}, it is impossible to produce a free boundary minimal annuli in one single sphere which is a piece of an extended singly periodic minimal Wente torus with ends.
\end{rem}
\section{Application}\label{Section 5.8}
In this section we study the special case $\hat{g}_0=1$ in order to obtain a better control of the Hamiltonian system. As a consequence we obtain an interesting geometric application. Consider
\begin{equation}\label{u0}
    u_0\coloneqq \min \{u(\lambda_t),\omega_3-u(\lambda_t)\}
\end{equation}
\begin{thm}\label{hamiltoniansystemg0=1}
    Let $(\alpha(u),\beta(u))$ be the Hamiltonian system associated to the minimal immersion $\Psi_{L(\tau,s),\hat{g}_0}$ with $(\tau,s)\in C_s$ and $\hat{g}_0=1$. Then
    \begin{gather*}
    \forall u\in (0,u_0)\quad \begin{cases}
        \alpha(-u)=\alpha(\omega_3-u)=-\alpha(u),\\ \beta(-u)=\beta(\omega_3-u)=-\beta(u).
    \end{cases}
\end{gather*}
\end{thm}
\begin{proof}
Recall that by Theorem \ref{alphabetaqualitative} the system $(\alpha(u),\beta(u))$ is well-defined and smooth in the interval $u\in (0,u(\lambda_t))\cup (u(\lambda_t),\omega_3)$. Using Corollary \ref{alphabetaomega3periodic}, we can introduce the well-defined and smooth system $(\tilde{\alpha}(u),\tilde{\beta}(u))$ in $(0,u_0)$ by
\begin{gather*}
    \begin{cases}
        \Tilde{\alpha}(u)\coloneq\alpha(-u)=\alpha(\omega_3-u),\\ \Tilde{\beta}(u)\coloneq\beta(-u)=\beta(\omega_3-u).
    \end{cases}
\end{gather*}
Notice that the pair $(\Tilde{\alpha},\Tilde{\beta})$ is a solution to the Hamiltonian system \eqref{hamiltonianequations} with initial conditions
\begin{gather*}
    \begin{cases}
        \Tilde{\alpha}(0)=\alpha(0)=0=-\alpha(0)\\\Tilde{\beta}(0)=\beta(0)=0=-\beta(0)
    \end{cases},\quad \begin{cases}
        \Tilde{\alpha}'(0)=-\alpha'(0)\\ \Tilde{\beta}'(0)=-\beta'(0).
    \end{cases}
\end{gather*}
Since the system $(-\alpha(u),-\beta(u))$ is also a solution to the Hamiltonian system \eqref{hamiltonianequations} with the same initial conditions as the system $(\Tilde{\alpha},\Tilde{\beta})$, and since they are both well-defined in the interval $(0,u_0)$ we must have, by uniqueness, that
\begin{equation*}
   \forall u\in (0,u_0),\quad \Tilde{\alpha}(u)=-\alpha(u),\quad \Tilde{\beta}(u)=-\beta(u),
\end{equation*}
as we wanted to prove.
\end{proof}

\begin{cor}\label{u(lambdat)=omega32}
Consider the minimal immersion  $\Psi_{L(\tau,s),\hat{g}_0}$ with $(\tau,s)\in C_s$ and $\hat{g}_0=1$. Then 
    \begin{gather*}
     \forall u\in (0,u_0)\quad  \begin{cases}
        R(\omega_3-u)=R(u)\\ \tan\theta(\omega_3-u)=-\tan\theta(u)
    \end{cases},\quad  c_3(\omega_3-u)=-c_3(u)+\frac{2\wp''(\frac{\omega_3}{2})}{\wp'(\frac{\omega_3}{2})}.
    \end{gather*}
\end{cor}
\begin{proof}
    The first claim comes from Theorem \ref{hamiltoniansystemg0=1} and equation \eqref{radiusangle}. For the second claim, notice that from Theorem \ref{periodicitycenter}
    \begin{equation}\label{equationcomega3-u}
        c_3(\omega_3-u)-4\left(\omega_3\wp(\mu)+\zeta(\omega_3)\right)=c_3(-u).
    \end{equation}
    By equation \eqref{alternc3u} we compute
    \begin{equation}\label{c3develpoment}
        c_3(-u)=b\left(-u-\frac{\omega_3}{2}\right)+4\left(\zeta(-u)-\zeta\left(\frac{\omega_3}{2}\right)\right)+\frac{4}{\alpha(-u)}\frac{g(-u)}{g(-u)^2+1}-\frac{\beta(-u)}{\alpha(-u)}\frac{g(-u)^2-1}{g(-u)^2+1}.
    \end{equation}
    From equation \eqref{gaussmapsigmafunctions} and the hypothesis $\hat{g}_0=1$ we see that $g(-u)=\frac{1}{g(u)}$. On the other hand, by Theorem \ref{hamiltoniansystemg0=1} we have $\alpha(-u)=-\alpha(u)$, $\beta(-u)=-\beta(u)$ for all $u\in (0,u_0)$ so we can write equation \eqref{c3develpoment}
    \begin{gather*}
        c_3(-u)=-b\left(u-\frac{\omega_3}{2}\right)-4\left(\zeta(u)-\zeta\left(\frac{\omega_3}{2}\right)\right)-\frac{4}{\alpha(u)}\frac{g(u)}{g(u)^2+1}+\frac{\beta(u)}{\alpha(u)}\frac{g(u)^2-1}{g(u)^2+1}-b\omega_3-8\zeta\left(\frac{\omega_3}{2}\right),
    \end{gather*}
    \begin{equation}\label{equationc3-u}
        c_3(-u)=-c_3(u)-b\omega_3-8\zeta\left(\frac{\omega_3}{2}\right).
    \end{equation}
    From equations \eqref{equationcomega3-u} and \eqref{equationc3-u} we obtain
    \begin{equation*}
        c_3(\omega_3-u)=-c_3(u)+4\left(\zeta(\omega_3)-2\zeta\left(\frac{\omega_3}{2}\right)\right).
    \end{equation*}
    By identity \cite[Example 20.23]{Whittaker} we conclude that
    \begin{equation*}
         c_3(\omega_3-u)=-c_3(u)+2\frac{\wp''\left(\frac{\omega_3}{2}\right)}{\wp'\left(\frac{\omega_3}{2}\right)},
    \end{equation*}
    and the proof is complete.
\end{proof}
\begin{cor}\label{calculationu0}
    Consider the minimal immersion  $\Psi_{L(\tau,s),\hat{g}_0}$ with $(\tau,s)\in C_s$ and $\hat{g}_0=1$. Then
    \begin{equation}
       u_0= u(\lambda_t)=\frac{\omega_3}{2}.
    \end{equation}
\end{cor}
\begin{proof}
    Indeed, by Corollary \ref{u(lambdat)=omega32} we have that 
    \begin{equation*}
        \forall u\in (0,u_0)\quad -c_3'(\omega_3-u)=-c_3'(u)
    \end{equation*}
    In particular, using the continuity of the function $c_3'$ we have in particular that
    \begin{equation}\label{step1}
        c_3'(\omega_3-u_0)=c_3'(u_0).
    \end{equation}
We know by equation \ref{u0} that either $u_0=u(\lambda_t)$ or $u_0=\omega_3-u(\lambda_t)$. In any case we obtain by equation \ref{step1} and Theorem \ref{centeranalysis} that
\begin{equation*}
    c_3'(\omega_3-u(\lambda_t))=c_3'(u(\lambda_t))=0.
\end{equation*}
Applying the uniqueness part of Theorem \ref{centeranalysis} we obtain that
\begin{equation*}
    \omega_3-u(\lambda_t)=u(\lambda_t),
\end{equation*}
concluding the proof.
\end{proof}
We can also prove the following theorem
\begin{thm}\label{freeboundaryrhombictwospheres}
    Consider the minimal immersion  $\Psi_{L(\tau,s),\hat{g}_0}$ with $(\tau,s)\in C_s$ and $\hat{g}_0=1$. If $\sum_{i<j} r_ir_j<0$ then there exist $u^*\in (0,\frac{\omega_3}{2})$ such that 
    \begin{equation*}
        \beta(u^*)=0,\quad  \beta(\omega_3-u^*)=0.
    \end{equation*}
\end{thm}
\begin{proof}
    In fact, since $\hat{g}_0=1$ we have by Proposition \ref{initialconditionsparameterspace} and the hypothesis $\sum_{i<j} r_ir_j<0$, that $\beta(0)=0$ and $\beta'(0)\in \bR^-$. Therefore there exists $\epsilon>0$ such that \begin{equation*}
        \beta(u)\in \bR^-,\quad u\in (0,\epsilon),
    \end{equation*}
By the asymptotic behavior of $\beta(u)$ of Theorem \ref{alphabetaqualitative} and from Corollary \ref{calculationu0} we obtain that
\begin{equation*}
    \lim_{u\to \frac{\omega_3^{-}}{2}}\beta(u)=+\infty.
\end{equation*}
By the intermediate value theorem, there exists $u^*\in (0,\frac{\omega_3}{2})$ such that $\beta(u^*)=0$. By Theorem \ref{hamiltoniansystemg0=1} and Corollary \ref{calculationu0} we have $\beta(\omega_3-u^*)=-\beta(u^*)=0$. Therefore, by Corollary \ref{u(lambdat)=omega32} we construct a free boundary minimal annuli in two spheres of different centers and same radius.
\end{proof}

\section{Appendix}
Numerically we know that for $\tau=\tan \theta_\tau$, $\theta_\tau\in \left(0,\frac{\pi}{6}\right)$, the condition $\sum_{i<j}r_ir_j<0$ is satisfied. We next provide our partial explanation for this numerical evidence.
\begin{figure}[h]
    \centering
    \includegraphics[width=0.5\linewidth]{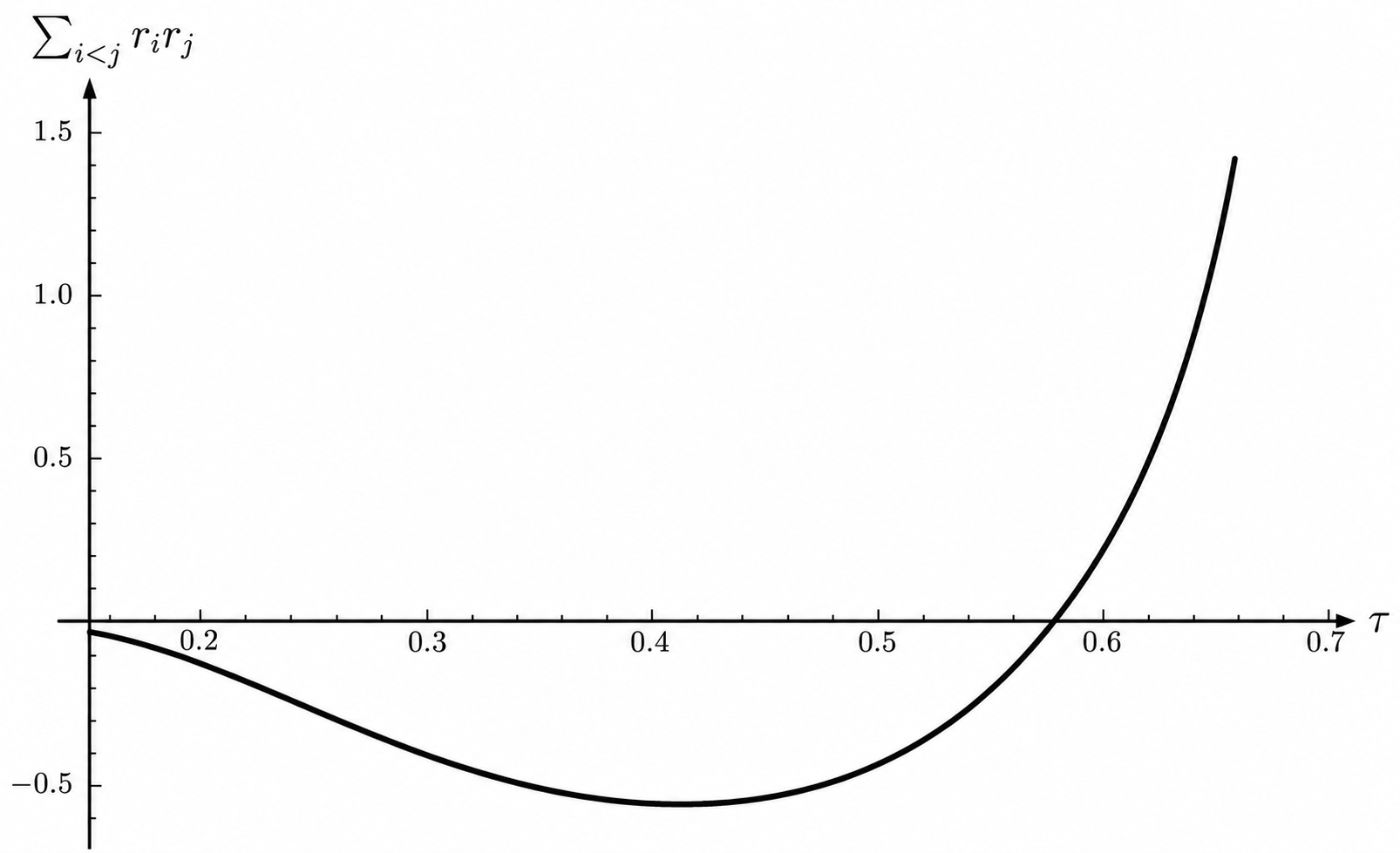}
    \caption{Numerical simulation of the condition $\sum_{i<j} r_ir_j$ as a function of the $\tau$ parameter.}
    \label{fig:Sumrirj}
\end{figure}
\begin{lema}
    Along the curve $C_s$, we have
     \begin{equation}
         \sum_{i<j}r_ir_j=\frac{8}{\lambda^4}\wp''(s(\tau)|\tilde{L}(\tau)).
    \end{equation}
\end{lema}
\begin{proof}
     Indeed,
    \begin{align*}
    \begin{split}
        \sum_{i<j}r_ir_j=&(b-4e_1)(b-4e_2)+(b-4e_1)(b-4e_3)+(b-4e_2)(b-4e_3)\\=&3b^2-8b(e_1+e_2+e_3)+16(e_1e_2+e_1e_3+e_2e_3)\\=&3b^2-4g_2\\=&48\left(\wp^2(\mu)-\frac{g_2}{12}\right)\\=&\frac{48}{\lambda^4}\left(\wp^2(s(\tau)|\tilde{L}(\tau))-\frac{\tilde{g}_2(\tau)}{12}\right)\\=&\frac{8}{\lambda^4}\wp''(s(\tau)|\tilde{L}(\tau)).
         \end{split}
    \end{align*}
\end{proof}

\begin{lema}
    For all $\theta_\tau\in \left(0,\frac{\pi}{6}\right)$ there exists a non-empty interval $(a_1(\tau),a_2(\tau))\subset (0,2)$ such that the Weierstrass function $\wp(v|\Tilde{L}(\tau))$ is concave down.
\end{lema}
\begin{proof}
    Suppose that $\wp''(x|\tilde{L})=0$. This implies that
\begin{equation*}
    \wp^2(x|\tilde{L})=\frac{\tilde{g}_2}{12}.
\end{equation*}
By Theorem \ref{signinvariants} we know that for $\theta_\tau\in \left(0,\frac{\pi}{6}\right)$, $\tilde{g}_2>0$, $\tilde{g}_3<0$. This implies 
\begin{equation*}
    \wp(x|\tilde{L})=\pm \sqrt{\frac{\tilde{g}_2}{12}}\in \bR.
\end{equation*}
Therefore $x$ must belong to the diagonals of the rhombus. Let $a_1,\ a_2\in [0,\tilde{\omega}_3]\cup [\tilde{\omega}_3,2\tilde{\omega}_2]$ such that
\begin{equation*}
    \wp(a_1|\tilde{L})=\sqrt{\frac{\tilde{g}_2}{12}},\quad  \wp(a_2|\tilde{L})=-\sqrt{\frac{\tilde{g}_2}{12}}.
\end{equation*}
Notice that in the rhombic case,
\begin{equation*}
    \tilde{g}_3=4\tilde{e}_1\tilde{e}_2\tilde{e}_3=4|\tilde{e}_1|^2\tilde{e}_3
\end{equation*}
which implies that 
\begin{equation*}
    \forall \theta_\tau\in \left(0,\frac{\pi}{6}\right),\quad \wp(\tilde{\omega}_3|\tilde{L})=\tilde{e}_3<0.
\end{equation*}
Also notice that in the rhombic case,
\begin{align*}
    \begin{split}
        |\tilde{e}_1-\tilde{e}_3|^2=&(\tilde{e}_1-\tilde{e}_3)(\tilde{e}_1-\tilde{e}_2)=\tilde{e}_1\tilde{e}_2-(\tilde{e}_1+\tilde{e}_2)\tilde{e}_3+\tilde{e}_3^2\\=&2\tilde{e}_3^2+\tilde{e}_1\tilde{e}_2=3\tilde{e}_3^2+(\tilde{e}_1\tilde{e}_2-\tilde{e}_3^2).
    \end{split}
\end{align*}
However, we can write
\begin{equation*}
    \tilde{g}_2=-4(\tilde{e}_1\tilde{e}_2+\tilde{e}_1\tilde{e}_3+\tilde{e}_2\tilde{e}_3)=-4(\tilde{e}_1\tilde{e}_2-\tilde{e}_3^2).
\end{equation*}
Therefore
\begin{equation*}
    \tilde{e}_3^2-\frac{\tilde{g}_2}{12}>0
\end{equation*}
We next prove that $a_1,\ a_2\in (0,\tilde{\omega}_3)=(0,2)$. Indeed, suppose by contradiction that $a_j\in [\tilde{\omega}_3,2\tilde{\omega}_2]$ for some $j=1,2$. Then by the monotonicity property of the Weierstrass function, we have
\begin{equation*}
    \wp(a_j|\tilde{L})<\wp(\tilde{\omega}_3|\tilde{L})=\tilde{e}_3<0.
\end{equation*}
This implies that
\begin{equation*}
    0=\wp^2(a_j|\tilde{L})-\frac{\tilde{g}_2}{12}>\tilde{e}_3^2-\frac{\tilde{g}_2}{12}>0,
\end{equation*}
which is a contradiction. Therefore we can assume that $0<a_1<a<a_2<\omega_3$, where $\wp(a|\tilde{L})=0$. We can analyze the concavity of the Weierstrass $\wp(v|\tilde{L})$ function along the interval $(0,\tilde{\omega}_3)$. 
\begin{itemize}
    \item $v\in (0,a_1)$. $\wp(v|\tilde{L})>\wp(a_1|\tilde{L})>0$. $\frac{1}{6}\left(\wp^2(v|\tilde{L})-\frac{\tilde{g}_2}{12}\right)>\frac{1}{6}\left(\wp^2(a_1|\tilde{L})-\frac{\tilde{g}_2}{12}\right)$. $\wp''(v|\tilde{L})>0$.
     \item $v\in (a_1,a)$. $\wp(a_1|\tilde{L})>\wp(v|\tilde{L})>\wp(a|\tilde{L})=0$. $\wp''(v|\tilde{L})<0$.
     \item $v\in (a,a_2)$. $0=\wp(a|\tilde{L})>\wp(v|\tilde{L})>\wp(a_2|\tilde{L})$. $\wp''(v|\tilde{L})<0$.
     \item $v\in (a_2,\tilde{\omega}_3)$. $0>\wp(a_2|\tilde{L})>\wp(v|\tilde{L})$. $\wp''(v|\tilde{L})>0$.
\end{itemize}
\end{proof}

\begin{cor}
    We have the equivalence
    \begin{equation}
    \forall \theta_\tau\in \left(0,\frac{\pi}{6}\right),\ \sum_{i<j}r_ir_j<0 \iff s(\tau)\in (a_1,a_2)
\end{equation}
\end{cor}
Recall that from Lemma \ref{criterio}, we know that $s(\tau)$ is the unique solution in the interval $(0,2)$ to the transcendental equation
\begin{equation}
     2\zeta(s(\tau)|\Tilde{L}(\tau))-\zeta(2|\Tilde{L}(\tau))s(\tau)=0
\end{equation}
 Alternatively, using Corollary \ref{curveCs} and equation \ref{eq4}, we have that $s(\tau)$ is a critical point of the classical theta function $\vartheta(z|q)$, i.e.
\begin{equation}
    \eval{\frac{\p}{\p z}\vartheta_1(\nu|q)}_{z=s(\tau)}=0,\quad \nu=\frac{\pi z}{4},\quad q=ie^{-\frac{\pi \tau}{2}}
\end{equation}
We hope these observations help to find the definitive explanation of the numerical evidence.

\printbibliography

\end{document}